\documentclass[11pt]{amsart}[draft]
\usepackage{upgreek}
\usepackage{enumerate}
\usepackage{a4wide}
\usepackage{geometry}
\usepackage{cancel}
\usepackage{float}
\usepackage{subcaption}
\usepackage{xcolor,graphicx,graphics,colortbl}
\usepackage{tikz}
\usepackage{tikz-cd}
\usetikzlibrary{shapes.geometric}
\usetikzlibrary{decorations}
\usepackage{amsmath,amsthm,amssymb,latexsym,amsfonts}
\usepackage{multicol, multirow, longtable, mathtools}
\usepackage{hyperref}

\newtheorem{theorem}{Theorem}[section]
\newtheorem{lemma}[theorem]{Lemma}
\newtheorem{proposition}[theorem]{Proposition}

\newtheorem{corollary}[theorem]{Corollary}
\newtheorem{maintheorem}{Theorem}

\theoremstyle{remark}
\newtheorem{remark}[theorem]{Remark}
\theoremstyle{remark}
\newtheorem{example}[theorem]{Example}

\newcommand{\C}{\ensuremath{\mathbb{C}}}
\newcommand{\R}{\ensuremath{\mathbb{R}}}

\renewcommand{\H}{\ensuremath{\mathbb{H}}}
\renewcommand{\O}{\ensuremath{\mathbb{O}}}
\newcommand{\F}{\ensuremath{\mathbb{F}}}
\definecolor{Gray}{gray}{0.85}
\newcommand{\g}[1]{\ensuremath{\mathfrak{#1}}}

\newcommand{\s}[1]{\ensuremath{\mathsf{#1}}}

\newcommand{\II}{\ensuremath{I\!I}}
\DeclareMathOperator{\Exp}{Exp}
\DeclareMathOperator{\Ad}{Ad}

\DeclareMathOperator{\arccosh}{arccosh}

\DeclareMathOperator{\Id}{Id}

\DeclareMathOperator{\Isom}{Isom}

\DeclareMathOperator{\proj}{proj}
\DeclareMathOperator{\rank}{rank}
\DeclareMathOperator{\spann}{span}

\DeclareMathOperator{\inj}{inj}

\DeclareMathOperator{\ImagPart}{Im}

\renewcommand{\Im}{\ImagPart}

\newcommand\strt{\rule[-.15em]{0em}{1.2em}}

\begin{document}
	\title[Hopf fibrations and totally geodesic submanifolds]{Hopf fibrations and totally geodesic submanifolds}
		\author[C.~E. Olmos]{Carlos E. Olmos}
	\address{Universidad Nacional de Córdoba, Argentina}
	\email{olmos@famaf.unc.edu.ar}
	
\author[A. Rodr\'{\i}guez-V\'{a}zquez]{Alberto Rodr\'{\i}guez-V\'{a}zquez}
\address{KU Leuven, Department of Mathematics, Celestijnenlaan 200B, Leuven, Belgium}
\email{alberto.rodriguezvazquez@kuleuven.be}
	\thanks{On the one hand, the first author was supported by  Famaf-UNC and CIEM-Conicet. On the other hand, the second author has been supported by the projects PID2019-105138GB-C21, PID2022-138988NB-I00
		(MCIN/AEI/10.13039/501100011033/FEDER, EU, Spain) and ED431F 2020/04 (Xunta de Galicia, Spain), the FPU program (Ministry of Education, Spain), the Methusalem grant METH/15/026 of the Flemish Government (Belgium), and the FWO Postdoctoral grant with project number 1262324N.}
	
	\subjclass[2010]{Primary 53C30, Secondary 53C35, 53C40}
	\keywords{Totally geodesic, Hopf fibrations}
\maketitle

\begin{abstract}
We classify totally geodesic submanifolds in Hopf-Berger spheres, which constitute a special family of homogeneous spaces diffeomorphic to spheres constructed via Hopf fibrations.  As a byproduct of our investigations, we have discovered very intriguing examples of totally geodesic submanifolds. In particular, we would like to highlight the following three:  totally geodesic submanifolds isometric to  real projective spaces,   uncountably many isometric but non-congruent totally geodesic submanifolds, and a totally geodesic submanifold that is not extrinsically homogeneous. 
Remarkably, all these examples only arise in  certain Hopf-Berger spheres with positive curvature.
\end{abstract}


\section{Introduction}
\label{sect:intro}

During the 1930s, Heinz Hopf \cite{hopf1,hopf2} introduced the fibrations bearing his name. These are constructed using complex numbers $\C$, quaternions $\H$, and octonions $\O$ and comprise the following:
\[\s{S}^1\rightarrow \s{S}^{2n+1}\rightarrow \C \s P^n,\qquad\s{S}^3\rightarrow \s{S}^{4n+3}\rightarrow \H\s  P^n,\qquad \s{S}^7\rightarrow \s{S}^{15}\rightarrow \O \s{P}^1.  \]
We will refer to these fibrations as the complex, quaternionic, or octonionic Hopf fibrations, respectively. These fibrations play a fundamental role in both geometry and topology. From a topological perspective, they contributed to the area of homotopy theory by providing an example of a homotopically non-trivial map from $\s{S}^3$ to $\s{S}^2$, which generates $\pi_3(\s{S}^2)= \mathbb{Z}$. Moreover, Adams \cite{adams} proved that a fiber bundle where the fiber, total space, and base space are connected spheres must be a Hopf fibration with base equal to $\C\s{P}^1=\s{S}^2$, $\H\s{P}^1=\s{S}^4$, or $\O\s{P}^1=\s{S}^8$. From a geometric point of view, Hopf fibrations have spurred the search of Riemannian manifolds with positive curvature by providing positively-curved non-symmetric homogeneous spaces, as for example, Berger spheres discovered in~\cite{berger}. Furthermore, Gromoll, Grove and Wilking~\cite{grove,wilkingrigidity} proved that a Riemannian submersion from a round sphere with connected fibers of positive dimension is metrically congruent to a Hopf fibration.

The  purpose of this article is two-fold. Firstly, to classify totally geodesic submanifolds in Hopf-Berger spheres, which constitute a special family of homogeneous spaces diffeomorphic to spheres constructed using Hopf fibrations.  Secondly, to provide certain general tools for the study of totally geodesic submanifolds in homogeneous spaces.

Let us equip the base and the total spaces of the Hopf fibrations with the corresponding symmetric metrics of diameter $\pi/2$, and sectional curvature equal to $1$, respectively. Under these conditions Hopf fibrations become Riemannian submersions with totally geodesic fibers.  Let us consider, for each $\tau>0$, the total space  of the complex, quaternionic, or octonionic Hopf fibration endowed with the Riemannian metric obtained from rescaling  the  metric tensor of the round sphere  by a factor $\tau$ in the vertical directions.  Such Riemannian homogeneous spaces are denoted by  $\s{S}_{\C,\tau}^{2n+1}$, $\s{S}_{\H,\tau}^{4n+3}$, and $\s{S}_{\O,\tau}^{15}$, depending on whether the Hopf fibration considered is the complex, the quaternionic, or the octonionic one, respectively. We refer to such spaces as the \textit {Hopf-Berger spheres}. Also, Hopf-Berger spheres can be regarded as geodesic spheres of rank one symmetric spaces, see~Subsection~\ref{subsect:geomhopf}.

The problem of classifying totally geodesic submanifolds in a given Riemannian manifold is a classical topic in the field of submanifold geometry. In the setting of Riemannian symmetric spaces, we recall that this problem has been extensively studied. However, despite all the efforts toward a general classification, we only have classifications for symmetric spaces of rank one~\cite{wolfrankone}, symmetric spaces of rank two~\cite{chen1, chen2, kleindga,kleintams, kleinosaka}, exceptional symmetric spaces~\cite{exceptional},  and some special classes of totally geodesic submanifolds such as reflective ones~\cite{Leung,LeungIII,LeungIV}, non-semisimple maximal ones~\cite{BO1}, or products of rank one symmetric spaces~\cite{RV:prod}. The curvature tensor of a symmetric space is parallel under the Levi-Civita connection and can be expressed by means of an easy formula in terms of Lie brackets. Hence, the  problem of classifying totally geodesic submanifolds turns out to be equivalent to classifying subspaces of the tangent spaces that satisfy the harmless-looking but extremely complicated condition of being a Lie triple system, see~for instance~\cite[\S 11.1.1]{BCO}. Additionally, totally geodesic submanifolds in symmetric spaces are extrinsically homogeneous, i.e.\ orbits of the isometry group of the ambient space. In the more general setting of Riemannian homogeneous spaces, the classification of totally geodesic submanifolds is a harder problem for a number of reasons. Firstly, in this setting, we need to look for subspaces of the tangent space that are not only invariant by the curvature tensor, which has a much more involved expression, see~Equation~(\ref{eq:curvature31}), but also by all its covariant derivatives, see~Theorem~\ref{th:fundamentaltg}. Secondly, although complete totally geodesic submanifolds of homogeneous spaces are intrinsically  homogeneous (see~\cite[Corollary 8.10]{kobayashi}), they are not necessarily extrinsically  homogeneous, see for instance~\cite{nikolayevskytg}. This picture suggests the necessity of developing new tools to classify totally geodesic submanifolds in homogeneous spaces.

Moreover, the utilization of totally geodesic submanifolds has been of capital importance when studying positively-curved spaces. As a matter of fact, one example is   the investigation of pinching constants for homogeneous spaces of positive curvature, see~\cite{puttmann}. The reason for this is that every critical point for the sectional curvature of a totally geodesic submanifold is a critical point for the sectional curvature of the ambient space. Another example is Wilking's connectedness principle~\cite[Theorem 1]{wilkingacta}, which provides an obstruction for the existence of totally geodesic submanifolds in compact Riemannian manifolds with positive curvature. This has allowed to develop Grove's program toward the obtention of rigidity results for positively-curved spaces under the presence of symmetry, which holds the ultimate hope that by relaxing the symmetry hypotheses new examples would be constructed.  Fixed-point components of isometries are extrinsically homogeneous totally geodesic submanifolds~(see~\cite[Lemma A.1]{puttmann}), and thus Wilking's connectedness principle provides a very useful tool to obtain such rigidity results, and its application has been quite successful in the last years, see for instance~\cite{kennard} or \cite{wilkingannals}. However, in this article we  present numerous  totally geodesic submanifolds which are not fixed-point components of isometries (those which are not well-positioned, see Lemma~\ref{lemma:fixedpoint}); or even more, an example which is not even  extrinsically homogeneous, see~Theorem~\ref{th:4inhomo}. This indicates that the existence of totally geodesic submanifolds is not completely linked to the presence of symmetry, and therefore the scope of Wilking's connectedness principle could  possibly be broadened. 

Certainly, the setting of homogeneous spaces of positive curvature turns out to be an interesting context where to study totally geodesic submanifolds since every totally geodesic submanifold of dimension $d\ge 2$ is again a homogeneous space with positive curvature. Hopf-Berger spheres provide some of the easiest examples of non-symmetric homogeneous metrics with positive sectional curvature. Indeed, they have positive sectional curvature if and only if $\tau\in(0,4/3)$. This was proved by Volper and Berestovskii in~\cite{Vol1,Vol2,Vol3},  see also~\cite{verdiani-ziller}. Interestingly, the condition $\tau\in(0,4/3)$ also characterizes the metrics with positive sectional curvature of the space $\C\s{P}_{\tau}^{2n+1}$,  defined analogously as the Hopf-Berger spheres but using the Riemannian submersion $\s{S}^2\rightarrow \C\s{P}^{2n+1}\rightarrow \H\s{P}^n$, see~\cite{Vol4} for details.

Let us revisit some  widely known facts concerning totally geodesic submanifolds of round spheres. The round sphere $\s{S}_1^n$ can be seen as a geodesic sphere of $\R^{n+1}$ with radius~$1$. Furthermore, every complete totally geodesic submanifold of $\s{S}_1^n$ is the intersection of a complete totally geodesic submanifold of $\R^{n+1}$ passing through the origin (a linear subspace) with $\s{S}_1^n$. The following theorem generalizes this well-known geometric fact to the setting of 2-point homogeneous spaces and provides the classification of totally geodesic submanifolds in Hopf-Berger spheres with $\tau\ge 1/2$. 
\vspace{0.2cm}

\begin{maintheorem}
	\label{th:a}
	Let $\s{S}^n_{\F,\tau}$, with $\tau\neq1$ and $\F\in\{\C,\H,\O\}$,  be a Hopf-Berger sphere, $\Sigma$ a submanifold of $\s{S}^n_{\F,\tau}$ with dimension $d\ge 2$,  and $\bar{M}$ the symmetric space of rank one where $\s{S}^n_{\F,\tau}$  is realized as a geodesic sphere. Then, if $\tau\ge 1/2$ or $\dim\Sigma\ge \dim\F$, the following statements are equivalent:
	\begin{enumerate}[i)]
		\item $\Sigma$ is a complete totally geodesic submanifold of $\s{S}^n_{\F,\tau}$.
		\item $\Sigma$ is the intersection of $\s{S}^n_{\F,\tau}$ (regarded as a geodesic sphere of $\bar{M}$) with a complete totally geodesic submanifold $M$  of $\bar{M}$, with dimension $d+1$, containing the center of the geodesic sphere $\s{S}^n_{\F,\tau}$.
	\end{enumerate}
	Every pair $(\Sigma,\s{S}^n_{\F,\tau})$ satisfying the above equivalent conditions is listed in Table~\ref{table:hopfbergertgwellpos}. Conversely, each pair listed in Table~\ref{table:hopfbergertgwellpos} such that $\tau\ge1/2$ or $\dim \Sigma\ge \dim \F$ corresponds to a unique congruence class of  totally geodesic submanifolds  of $\s{S}^n_{\F,\tau}$.
	
\end{maintheorem}
\begin{table}[h!]
	\renewcommand{\arraystretch}{1.2}
	\centering
	\begin{tabular}{lll}
		\hline
		\multicolumn{1}{|l}{$\s{S}^{2n+1}_{\C,\tau}$} &                                                          & \multicolumn{1}{l|}{}                  \\ \hline
		\multicolumn{1}{|l|}{}               & \multicolumn{1}{l|}{$\s{S}^{2k+1}_{\C,\tau}$}                     & \multicolumn{1}{c|}{$1\leq k\leq n-1$} \\
		\multicolumn{1}{|l|}{}               & \multicolumn{1}{l|}{$\s{S}_1^k$}                     & \multicolumn{1}{c|}{$2\leq k\leq n$\phantom{.aaa}}  \\ \hline
		\multicolumn{1}{|l}{$\s{S}^{4n+3}_{\H,\tau}$} &                                                          & \multicolumn{1}{l|}{}                  \\ \hline
		\multicolumn{1}{|l|}{}               & \multicolumn{1}{l|}{$\s{S}^{4k+3}_{\H,\tau}$}                     & \multicolumn{1}{c|}{$1\leq k\leq n-1$} \\
		\multicolumn{1}{|l|}{}               & \multicolumn{1}{l|}{$\s{S}^{2k+1}_{\C,\tau}$}                     & \multicolumn{1}{c|}{$1\leq k\leq n$\phantom{.aaa}} \\
		\multicolumn{1}{|l|}{}               & \multicolumn{1}{l|}{$\s{S}^{k}_{1}$}                & \multicolumn{1}{c|}{$2\leq k\leq n$\phantom{.aaa}}  \\
		\multicolumn{1}{|l|}{}               & \multicolumn{1}{l|}{$\s{S}^{k}_{1/\tau}$}                & \multicolumn{1}{c|}{$2\leq k\leq 3$\phantom{.aaa}}\\
		\hline
		\multicolumn{1}{|l}{$\s{S}^{15}_{\O,\tau}$} &                                                          & \multicolumn{1}{l|}{}                  \\ \hline
		\multicolumn{1}{|l|}{}               & \multicolumn{1}{l|}{$\s{S}^{7}_{\H,\tau}$}                     & \multicolumn{1}{l|}{}  \\
		\multicolumn{1}{|l|}{}               & \multicolumn{1}{l|}{$\s{S}^{3}_{\C,\tau}$}                     & \multicolumn{1}{l|}{}  \\
		\multicolumn{1}{|l|}{}               & \multicolumn{1}{l|}{$\s{S}^k_{1/\tau}$}                     & \multicolumn{1}{c|}{$2\leq k \leq 7$ \phantom{aaa}} \\ \hline
		&                                                          &                                       
	\end{tabular}
	\caption{Well-positioned totally geodesic submanifolds of dimension $d\ge2$ in Hopf-Berger spheres. }
	\label{table:hopfbergertgwellpos}
\end{table}
The condition  $\tau\ge1/2$ has an interesting geometric interpretation. It turns out that  the closed geodesic balls with radius $t$ given by Equation \eqref{eq:radius_tau_homothety} are geodesically convex if and only if $\tau\ge 1/2$. This follows since the principal curvatures of the shape operator  of the  geodesic spheres of the corresponding rank one symmetric spaces with respect to an outer normal vector field are non-positive, see Equations~\eqref{eq:b1} and \eqref{eq:b2}.  

\begin{remark}
	Notice that for $\tau=1$ the sphere $\s{S}^n_{\F,\tau}$ is round. If $\tau\neq 1$, we have to assume that $\Sigma$ has dimension $d\ge 2$, since there always are non-closed geodesics in $\s{S}^n_{\F,\tau}$, see Lemma~\ref{lemma:slopegeod},  and a non-closed geodesic in $\s{S}^n_{\F,\tau}$  cannot arise as the intersection with a totally geodesic submanifold of $\bar{M}$.
\end{remark}

The totally geodesic submanifolds described in Theorem~\ref{th:a} satisfy a natural geometric property which will be fundamental for our study. We say that a  totally geodesic submanifold $\Sigma$ of a Hopf-Berger sphere $\s{S}^n_{\F,\tau}$ is \textit{well-positioned} if  $T_p \Sigma= (T_p\Sigma\cap \mathcal{H}_p)\oplus(T_p\Sigma\cap \mathcal{V}_p)$ for every $p\in M$, where $\mathcal{H}$ and $\mathcal{V}$ denote the horizontal and vertical distributions of the corresponding Hopf fibration, respectively. The examples of totally geodesic submanifolds in Table~\ref{table:hopfbergertgwellpos} are all obtained as  intersections of geodesic spheres of a rank one symmetric space with a complete totally geodesic submanifold passing through the center of the corresponding geodesic sphere. It turns out that this is equivalent to being well-positioned (see Theorem~\ref{th:wellpos}).  Furthermore, every well-positioned totally geodesic submanifold of a Hopf-Berger sphere is extrinsically homogeneous, see Remark~\ref{rem:congwellpos}.

The classification given in Theorem~\ref{th:a}  works for those Hopf-Berger spheres with $\tau\ge 1/2$. However, we managed to achieve the classification for each value of $\tau>0$.  Below, we  describe the classification of totally geodesic submanifolds in $\s{S}^n_{\F,\tau}$, for each $\F\in\{\C,\H,\O  \}$.

Let us start with $\F=\C$. In this case totally geodesic submanifolds are well-positioned and  the classification was already obtained by Torralbo and Urbano~\cite{torralbo} when $\tau<1$. Here the classification is a consequence of Theorem~\ref{th:a}.
Now we will discuss the case when $\F\in\{\H,\O \}$, where not well-positioned totally geodesic submanifolds do arise. It is important to highlight that these examples appear only in certain positively-curved Hopf-Berger spheres (those with $\tau<\tfrac{1}{2}$).

In order to study not well-positioned totally geodesic submanifolds, we will  introduce a pseudo-Riemannian metric $\upalpha$  on $\s{S}^n_{\F,\tau}$, with $\F\in\{\H,\O \}$. This pseudo-Riemannian metric $\upalpha$ is induced by the second fundamental form of the  embedding of the corresponding geodesic sphere into $\H \s P^n$ and $\O \s P^2$, respectively. In particular, we prove (see Theorem~\ref{th:notwellpos}) that a complete submanifold $\Sigma$ of $\s{S}^n_{\F,\tau}$ of dimension $d\ge2$, with $\tau<\tfrac{1}{2}$, is a not well-positioned totally geodesic submanifold   if and only if $T_p\Sigma$ is $\upalpha$-isotropic and curvature invariant for every $p\in\Sigma$. This allows to reduce the seemingly infinite number of conditions given by Theorem~\ref{th:fundamentaltg}, to check that the Riemannian exponential map of a subspace is a  totally geodesic submanifold, to just two conditions.

It turns out that a good amount of not well-positioned totally geodesic submanifolds in $\s{S}^n_{\F,\tau}$ are isometric to round spheres. The  construction of these not well-positioned totally geodesic spheres can be described in purely geometric terms. Let us consider a Helgason sphere $\s{S}_{4}^{\dim \F}$ of $\F \s P^n$, with $\F\in\{\H,\O\}$, i.e.\ a maximal totally geodesic sphere with maximal sectional curvature. The idea is to deform a totally geodesic hypersurface $\widehat{\Sigma}$ of $\s{S}_{4}^{\dim \F}$ in such a way that it is $\upalpha$-isotropic and curvature invariant. As a matter of fact, we prove that every not well-positioned totally geodesic sphere  is congruent to a totally geodesic submanifold of the deformation of $\widehat{\Sigma}$,  see Theorem \ref{th:notwellsphere}.

\medskip

\begin{maintheorem}
	\label{th:C}
	Let $\Sigma$ be a complete  submanifold of $\s{S}^{4n+3}_{\H,\tau}$, $n\ge 1$, of dimension $d\ge 2$. Then, $\Sigma$ is totally geodesic if and only if it is congruent to a well-positioned totally geodesic submanifold (see  Table~\ref{table:hopfbergertgwellpos}) or a not well-positioned totally geodesic submanifold (see~Table~\ref{table:quaternionictgnotwellpos}).
\end{maintheorem}
\begin{table}[H]
	\renewcommand{\arraystretch}{1.2}
	\begin{tabular}{|lll|}
		\hline
		$\s{S}^{4n+3}_{\H,\tau}$, & $\tau<1/2$                                &                                    \\ \hline
		\multicolumn{1}{|l|}{}    & \multicolumn{1}{l|}{$\s{S}^k_{4-4\tau}$}  & $k\in\{2,3\}$                      \\
		\multicolumn{1}{|l|}{}    & \multicolumn{1}{l|}{$\R\s{P}^2_{2/3}$}    & $\tau=1/3$                         \\
		\multicolumn{1}{|l|}{}    & \multicolumn{1}{l|}{$\R\s{P}^2_{1-\tau}$}  & $\tau<1/3$, $n\ge 2$               \\
		\multicolumn{1}{|l|}{}    & \multicolumn{1}{l|}{$\R\s{P}^3_{3/4}$}    & $\tau=1/4$, $n\ge 2$               \\
		\multicolumn{1}{|l|}{}    & \multicolumn{1}{l|}{$\R\s{P}^3_{1-\tau}$} & $\tau<1/4$, $n\ge 3$       \\
		\hline                  
	\end{tabular}
	\caption{Not well-positioned totally geodesic submanifolds of dimension $d\ge2$ in~$\s{S}^{4n+3}_{\H,\tau}$ for $n\ge 1$.}
	\label{table:quaternionictgnotwellpos}
\end{table}
Surprisingly, in $\s{S}^{4n+3}_{\H,\tau}$, we have discovered (see Table~\ref{table:quaternionictgnotwellpos}) not well-positioned totally geodesic submanifolds of $\s{S}^{4n+3}_{\H,\tau}$ isometric to $\R\s{P}^2_{1-\tau}$ when $\tau\leq 1/3$; and to $\R\s{P}^3_{1-\tau}$ when $\tau\leq 1/4$, where we denote by $\R\s{P}^\ell_{\kappa}$ the real projective space of dimension $\ell$ with constant sectional curvature equal to $\kappa>0$. The explicit description of these examples is given in Lemma~\ref{lemma:notwellposproj2} and Lemma~\ref{lemma:notwellposproj3}. These are, up to the knowledge of the authors, the first known examples of totally geodesic embeddings of real projective spaces in Riemannian manifolds diffeomorphic to spheres with positive curvature. Indeed, Wilking's connectedness principle (see~\cite[Theorem 1]{wilkingacta}) implies that $\s{S}^7$ cannot carry a metric of positive sectional curvature admitting a totally geodesic submanifold diffeomorphic to $\R \s{P}^k$ with $k\ge 4$. Since we have exhibited an example of a totally geodesic submanifold diffeomorphic to $\R \s{P}^2$ in $\s{S}^7_{\H,1/3}$,  it is natural to raise the question of whether there will be examples of Riemannian metrics on $\s{S}^7$ with positive sectional curvature admitting totally geodesic submanifolds diffeomorphic to $\R\s{P}^3$.

\begin{maintheorem}
	\label{th:D}
	Let $\Sigma$ be a complete  submanifold of $\s{S}^{15}_{\O,\tau}$ of dimension $d\ge 2$. Then, $\Sigma$ is totally geodesic if and only if it is congruent to a well-positioned totally geodesic submanifold listed in Table~\ref{table:hopfbergertgwellpos} or a not well-positioned totally geodesic submanifold listed in Table \ref{table:octonionictgnotwellpos}.
\end{maintheorem}
\begin{table}[H]
	\renewcommand{\arraystretch}{1.2}
	\begin{tabular}{|lll|}
		\hline
		$\s{S}^{15}_{\O,\tau}$,   & $\tau<1/2$                                &                                    \\ \hline
		\multicolumn{1}{|l|}{}    & \multicolumn{1}{l|}{$\s{S}^k_{4-4\tau}$}  & $k\in\{2,5,6,7\}$                  \\
		\multicolumn{1}{|l|}{}    & \multicolumn{1}{l|}{$\s{S}^3_{4-4\tau}$}  & Infinitely many congruence classes      \\
		\multicolumn{1}{|l|}{}    & \multicolumn{1}{l|}{$\s{S}^4_{4-4\tau}$}  & Not extrinsically homogeneous \\
		\multicolumn{1}{|l|}{}    & \multicolumn{1}{l|}{$\R\s{P}^2_{2/3}$}    & $\tau=1/3$                         \\ \hline
	\end{tabular}
	\caption{Not well-positioned totally geodesic submanifolds of dimension $d\ge2$ in~$\s{S}^{15}_{\O,\tau}$.}
	\label{table:octonionictgnotwellpos}
\end{table}
In this case a not well-positioned maximal totally geodesic sphere $\widehat{\Sigma}$ has dimension~$7$. We define a $3$-form $\upvarphi$ in $\s{S}^{15}_{\O,\tau}$ that plays an important role in the study of the congruence of the not well-positioned totally geodesic $3$-dimensional spheres. The $3$-form $\upvarphi$ is such that its restriction to the tangent space of a not well-positioned totally geodesic $7$-dimensional sphere $\widehat{\Sigma}$ agrees with the associative calibration defined by Harvey and Lawson in its foundational article on calibrated geometries, see~\cite{harvey-lawson}.  Moreover, the absolute value of $\upvarphi$ associates a real number in $[0,1]$ to the tangent space of a not well-positioned totally geodesic $3$-dimensional sphere in such a way that two not well-positioned totally geodesic $3$-dimensional spheres $\Sigma_1^3$ and $\Sigma_2^3$ passing through some point $o\in\s{S}^{15}_{\O,\tau}$ are congruent in $\s{S}^{15}_{\O,\tau}$ if and only if $|\upvarphi(T_o\Sigma^3_1)|=|\upvarphi(T_o\Sigma^3_2)|$, see Proposition~\ref{prop:calibcong} and Example~\ref{ex:3dimoctonionic}. All not well-positioned totally geodesic $3$-dimensional spheres are mutually isometric. In particular, this implies that there exists a family of uncountable isometric and non-congruent totally geodesic submanifolds in $\s{S}^{15}_{\O,\tau}$, when $\tau<1/2$. This is the first such family in  a homogeneous space described in the literature up to the knowledge of the authors.

As we pointed out, a totally geodesic submanifold of a homogeneous space must be intrinsically homogeneous but this does not imply that it is an orbit of some subgroup of the isometry group of the ambient space. Indeed, we prove that every totally geodesic submanifold of a Hopf-Berger sphere is extrinsically homogeneous  except not well-positioned totally geodesic $4$-dimensional spheres in $\s{S}^{15}_{\O,\tau}$, see Remark~\ref{rem:congwellpos}, Corollary~\ref{cor:extrinsichomquaternionic} and Theorem~\ref{th:4inhomo}. Hopf-Berger spheres are g.o.\ spaces, i.e.\  spaces such that all their geodesics are orbits. Thus, this example shows that the class of homogeneous spaces where all the totally geodesic submanifolds are orbits is strictly contained in the class of g.o.\ spaces, and as far as the authors are aware, this is the first  example of a not extrinsically homogeneous totally geodesic submanifold in a geodesic orbit space described in the literature.

This article is organized as follows. In Section~\ref{sect:prelim}, we summarize some standard facts about totally geodesic submanifolds, geodesic spheres and Riemannian submersions (Subsection~\ref{subsect:tggeodspheressubmersions}),  we fix some notation concerning homogeneous spaces (Subsection~\ref{subsect:homprelim}), we  explain some well-known facts about rank one symmetric spaces and their totally geodesic submanifolds~(Subsection~\ref{subsect:rkonetg}), and we summarize some aspects of the geometry of Hopf-Berger spheres~(Subsection~\ref{subsect:geomhopf}). The purpose of Section~\ref{sect:tgandhom} is to provide some general tools for studying totally geodesic submanifolds in homogeneous spaces. In particular, we prove the following three lemmas:  Lemma~\ref{lemma:DRinvariant}, which gives a simple sufficient condition to check that a submanifold is totally geodesic;  Lemma~\ref{lemma: g.o.tg}, which  provides an obstruction for the existence of totally geodesic submanifolds in compact g.o.\ spaces, and  Lemma~\ref{lemma:finclas}, which relates the number of isotropy-congruence classes to the homogeneity of a totally geodesic submanifold. In Section~\ref{sect:wellpos},  we characterize well-positioned totally geodesic submanifolds in Hopf-Berger spheres~(see~Theorem~\ref{th:wellpos}) as intersections of totally geodesic submanifolds of a rank one symmetric space $\bar M$ with a Hopf-Berger sphere regarded as a geodesic sphere of $\bar M$. Moreover, in Subsection~\ref{subsect:examplestg}, we analyze some examples of well-positioned totally geodesic submanifolds. In Section~\ref{sect:notwellpos}, we study not well-positioned totally geodesic submanifolds. Subsection~\ref{subsect:geodesics} is devoted to the $1$-dimensional examples, i.e.\ geodesics, and Subsection~\ref{subsect:charactnotwellpos} is devoted to give a characterization~(Theorem~\ref{th:notwellpos}) of not well-positioned totally geodesic submanifolds of Hopf-Berger spheres by means of the pseudo-Riemannian metric $\upalpha$ and the curvature tensor. Furthermore, in Subsection~\ref{subsect:notwellposspheres}, we completely characterize  not well-positioned maximal totally geodesic spheres (see Theorem~\ref{th:notwellsphere}). In Section~\ref{sect:quaternionic}, we construct and classify  totally geodesic projective spaces in $\s{S}^{4n+3}_{\H,\tau}$ (see Lemma~\ref{lemma:notwellposproj2} and Lemma~\ref{lemma:notwellposproj3}), which allows us to prove Theorem~\ref{th:C}. Finally, in Section~\ref{sect:octonionic}, we analyze the moduli space of congruence classes of not well-positioned totally geodesic spheres and the extrinsic homogeneity of these submanifolds. This will lead us to prove that there exists a family with uncountably many isometric non-congruent totally geodesic submanifolds~(see Theorem~\ref{th:congoctonionic}), and that every not well-positioned totally geodesic sphere of dimension $4$ is not extrinsically homogeneous (see Theorem~\ref{th:4inhomo}). Thus, we can conclude the proof of Theorem~\ref{th:D}.

\section{Preliminaries}
\label{sect:prelim}
\subsection[Totally geodesic submanifolds, geodesic spheres, Riemannian submersions]{Totally geodesic submanifolds, geodesic spheres, and Riemannian submersions} 
\label{subsect:tggeodspheressubmersions}
Let $(\bar{M},\langle \cdot,\cdot\rangle)$ be a  complete Riemannian manifold. We  denote by $\bar{\nabla}$ its Levi-Civita connection, by $\bar{R}$ its curvature tensor defined by the convention $\bar{R}(X,Y)=[\bar{\nabla}_X,\bar{\nabla}_Y]-\bar{\nabla}_{[X,Y]}$,
and  by $\inj(p)$  its injectivity radius at $p\in \bar{M}$.  

Let $M$ be an immersed submanifold of $\bar{M}$. We denote its induced Levi-Civita connection by $\nabla$, and its second fundamental form by $\II$. A connected submanifold $M$ is {totally geodesic} if every geodesic of $M$ is also a geodesic of $\bar{M}$. Equivalently,  a connected submanifold $M$ of $\bar{M}$ is totally geodesic if and only if $\II=0$. Moreover, every totally geodesic submanifold $M$ of $\bar{M}$ can be extended to a complete totally geodesic submanifold of $\bar{M}$, see Theorem~\ref{th:fundamentaltg}. Given a complete totally geodesic submanifold $M$ of $\bar{M}$ passing through $p\in\bar{M}$, we have $\overline{\exp}_p(V)=M$, for some linear subspace $V$ of $T_p \bar{M}$, where ${\overline{\exp}}_p$ denotes the Riemannian exponential map of $\bar{M}$ at $p\in\bar{M}$. 
We say that a subspace $V$ of $T_p \bar M$ is a \textit{totally geodesic subspace} if $\overline{\exp}_p(V)$ is a totally geodesic submanifold of $\bar M$.

Indeed, one has the following result, which characterizes those subspaces of the tangent space whose image under the exponential map are totally geodesic, see~\cite[Sections 10.3.1 and 10.3.2]{BCO}.
\begin{theorem}
	\label{th:fundamentaltg}
	Let $\bar{M}$ be a real analytic complete Riemannian manifold, $p\in\bar{M}$ and $V$ a linear subspace of $T_p\bar{M}$. There exists a complete totally geodesic immersed submanifold $M$ of $\bar{M}$ such that $p\in M$ and $T_pM=V$ if and only if $(\bar{\nabla}^k \bar{R})_p$ leaves $V$ invariant for every $k\ge 0$.
\end{theorem} 
The above well-known theorem follows from the main theorem in \cite{Her} (cf.~\cite[Section 10.3.2]{BCO}) and the following observation that uses the analyticity of $\bar{M}$: 
\noindent if $V\subset T_p \bar {M}$ is the tangent space  of a totally geodesic submanifold at $p\in \bar M$ and $\rho:=\inj(p)$, then 
$\Sigma = {\exp}_p(V_\rho)$ is totally geodesic in $\bar{M}$, where $V_\rho = \{ v\in V: \Vert v \Vert <\rho\}$.

In what follows, we  give an overview of some well-known facts about the extrinsic geometry of geodesic spheres  in Riemannian manifolds. Let $v(s)$, $s\in(-\varepsilon,\varepsilon)$, be a smooth curve in the unit tangent sphere $\s{S}(T_p\bar M)$ and let us consider the geodesic variation $\Gamma(s,t):=\gamma _{v(s)}(t)=\overline{\exp}_p(t v(s))$. Let us denote by $\tfrac{D \,}{\partial t}$ and $\tfrac{D \,}{\partial s}$ the partial covariant derivatives. Then, if $0<t<{\inj}(p)$, the field along $\gamma _{v(s)}(t)$ given by  $\dot\gamma_{v(s)}(t)=\tfrac{\partial \,}{\partial t}
\gamma _{v(s)}(t)$
is the outer unit normal at $\gamma _{v(s)}(t)$ to the geodesic sphere $\s{S}_t(p):= \overline{\exp} _p (t\s{S}(T_p\bar{M}))$ of radius $t>0$ and center $p\in\bar M$.
Since $ 
\Vert  \dot\gamma_{v(s)}(t)\Vert = 1$,  we have that
$\tfrac{D \,}{\partial s}_{\vert s= 0}\dot\gamma_{v(s)}(t)$ and $\dot\gamma_{v(s)}(t)$ are orthogonal. Thus,
\begin{equation*}\label{AA1}
	\tfrac{D \,}{\partial s}_{\vert s= 0}\dot\gamma_{v(s)}(t) = -\mathcal{S}_{\xi}^t J(t),
\end{equation*}
where $\mathcal{S}_\xi^t$ denotes the shape operator of the geodesic sphere $\s{S}_t(p)$ with respect to the outer unit direction $\xi$, and $J(t)=\tfrac{\partial}{\partial s}_{\vert s=0} \Gamma(s,t)$  is a  Jacobi field along 
$\gamma _{v(0)}(t)$ with initial conditions $J(0)=0$ and $\tfrac{D \,}{\partial t}_{\vert t=0}J(t) = \dot v(0)$. Now using that 
$\tfrac{D \,}{\partial s}_{\vert s= 0}\dot \gamma_{v(s)}(t) = 
\tfrac{D \,}{\partial t} J(t)$, we obtain the following expression for the shape operator of the geodesic sphere 
\begin{equation}\label{eq:jacobisphere}
	\mathcal{S}_{\xi}^t J(t) = -\tfrac{D \,}{\partial t} J(t).
\end{equation}

\begin{lemma}\label{lemma:riemannianintersectiontg} Let $\bar{M}$ be a Riemannian manifold and let $\Sigma$ be an embedded totally geodesic submanifold of $\bar{M}$. Let $M$ be an embedded submanifold of $\bar{M}$ satisfying the following conditions:
	\begin{enumerate}[i)]
		\item $\Sigma \cap M$ is non-empty.
		\item $T_q\Sigma$ contains the normal space $\nu _qM$ of $M$ for every $q\in \Sigma \cap M$.
	\end{enumerate}	
	Then,  any connected component of $\Sigma \cap M$ is an embedded totally geodesic submanifold of $M$ and  $T_q(\Sigma \cap M)$ is invariant under the shape operator $\mathcal{S}_\xi$ of $M$, for every $\xi \in \nu _q M$ and $q\in \Sigma \cap M$.
\end{lemma}

\begin{proof}
	By \textit{i)} and \textit{ii)}, $\Sigma$ and $M$ are transverse at every  $q\in\Sigma\cap M$. Then, it follows that $\Sigma\cap M$ is an embedded submanifold of $\bar{M}$. Moreover, $T_q(\Sigma \cap M)=(T_q\Sigma) \cap (T_qM)$ and $T_q\Sigma =  T_q(\Sigma \cap M)\oplus \nu _qM$,  for every $q\in \Sigma\cap M$.

	We may assume that $\dim (\Sigma \cap M)> 0$. Let $X,Y$ be  tangent   fields of $ \Sigma \cap M$ defined around $q$.  Thus $(\bar{\nabla} _XY)_q \in T_q\Sigma = T_q(\Sigma \cap M)\oplus \nu _qM$. Hence, the tangent projection of $(\bar{\nabla} _XY)_q$  to $M$  belongs to $T_q(\Sigma \cap M)$. This shows that $\Sigma \cap M$ is a totally geodesic submanifold of $M$.  Let $\xi$ be a unit normal vector field of $M$ defined around   $q\in \Sigma \cap M$ and $v \in T_q(\Sigma \cap M)$. Since $\Sigma$ is a totally geodesic submanifold of $\bar{M}$,  $(\bar{\nabla} _v  \xi)_q  \in T_q\Sigma = 
	T_q(\Sigma \cap M) \oplus \nu _qM$. Consequently, the projection  of $(\bar{\nabla} _v  \xi)_q$ to $T_qM$ belongs to $T_q(\Sigma \cap M)$, and this shows that $\mathcal{S}_{\xi _q}T_q(\Sigma \cap M)\subset T_q(\Sigma \cap M)$.
\end{proof}
As a consequence of Lemma~\ref{lemma:riemannianintersectiontg} and the previous discussion, we obtain the following result.

\begin{corollary}\label{cor:geodsphereinter} Let $\bar{M}$ be a Riemannian manifold,  $p\in \bar{M}$, and  $\s{S}_t(p)$ the geodesic sphere of radius $t\in(0,\inj(p))$. Let  $ V\subset T_p\bar{M}$ be a vector subspace  and   consider the set  $V_r: = \{v\in  V: \Vert v\Vert <r\}$, where $r\in(t,\inj(p))$.  Assume that 
	$\Sigma = \overline{\exp}_p ( V _r)$  is a totally geodesic submanifold of $\bar M$.  Then, $\s{S}_t(p)\cap \Sigma$ is a totally geodesic submanifold of  $\s{S}_t(p)$ whose tangent space is invariant under the shape operator $\mathcal{S}_{\xi}^t$ of $\s{S}_t(p)$, where $\xi$ is the normal unit vector field of $\s{S}_t(p)$. 
	\end {corollary}

	The following lemma shows that there is a certain class of totally geodesic submanifolds of the total space of a given Riemannian submersion that remain totally geodesic when we rescale the metric on the vertical distribution, see~ \cite[Lemma 3.12]{dearricott} for a proof.
	
	\begin{lemma}\label{lemma:submersiontg} Let  $\pi\colon M  \to B$ be a Riemannian submersion  and denote by  $\mathcal{H}$ and $\mathcal{V}$ the horizontal and vertical distributions, respectively. Consider the family of Riemannian metrics $\langle \cdot, \cdot \rangle^\tau $ of $M$ obtained by rescaling the metric tensor  $\langle \cdot, \cdot \rangle $ of $M$ by a factor $\tau>0$ in the vertical distribution. Moreover, let $\Sigma$ be a totally geodesic submanifold of $M$, with respect to the original metric, such that
		\[T_q\Sigma = W_1(q) \oplus  W_2(q) , \quad \text{for all $q\in \Sigma$}, \]	
		where $W_1(q)\subset \mathcal H_q$ and $W_2(q)\subset \mathcal{V} _q$.  Then, $\Sigma$ is a totally geodesic submanifold of $M$ with respect to any metric $\langle \cdot, \cdot \rangle^\tau $.
	\end{lemma}

	\subsection{Homogeneous spaces}
	\label{subsect:homprelim}
	For a nice introduction to the theory of homogeneous spaces, one can consult \cite{arvanitoyeorgos}, \cite[Chapter 4]{berestovskii}, or \cite[Chapter X]{kobayashi}.
	A Riemannian manifold $M$ is \textit{homogeneous} if there exists some subgroup $\s{G}$ of $\mathrm{Isom}(M)$ such that $\s{G}$ acts transitively on $M$. We fix a point $o\in M$, so  $M$ is diffeomorphic to $\s{G}/\s{K}$, where $\s{K}=\s G_o$ is the isotropy at $o$, by the map $\Phi\colon \s{G}/\s{K}\rightarrow M$ defined by $g\s{K}\rightarrow g(o)$. We pull back the metric of $M$ by $\Phi$ to $\s{G}/\s{K}$, turning $\Phi$ into an isometry. Furthermore, the metric  $\langle\cdot,\cdot\rangle$ induced in $\s{G}/\s{K}$  is $\s{G}$-invariant.

	For any $X\in\g{g}$, where $\g{g}$ is the Lie algebra of $\s G$, we can associate a Killing vector field $X^*$ given by $X_p^{*}=\frac{d}{dt}_{|t=0} (\Exp(t X)\cdot p)$, for every $p\in M$, where $\Exp$ denotes the Lie exponential map of $\s{G}$. Homogeneous spaces can be characterized in terms of Killing vector fields as follows. A Riemannian manifold $M$ is homogeneous if and only if for every $p\in M$ and every $v\in T_p M$, there is a Killing vector field $X\in\Gamma(TM)$ such that $X_p=v$. Moreover, Riemannian homogeneous spaces $\s G/\s K$ always have a reductive decomposition. A \textit{reductive decomposition} is a splitting $\g{g}=\g{k}\oplus\g{p}$, where $\g{k}$ is the Lie algebra of $\s{K}$ and $\g{p}$ is an $\Ad(\s{K})$-invariant subspace of $\g{g}$. Thus, we have the bracket relations $[\g{k},\g{p}]\subset\g{p}$ and $[\g{k},\g{k}]\subset\g{k}$. If we consider the linearization at $o\in M$ of the isotropy action of $\s{K}$ on $M$, we get the \textit{isotropy representation} of the homogeneous space $M$, which is defined as $k\in \s{K}\mapsto k_{*o}\in \s{GL}(T_o M)$, where $k_{*o}$ denotes the differential of $k$ at $o\in M$. This is equivalent to the adjoint representation of $\s{G}$ restricted to $\s{K}$ on $\g{p}$, since $\g{p}$ and $T_o M$ can be identified by the map which sends $X\in\g{p}$ to $X^*_o\in T_o M$.

	Let us denote by $X_{\g{k}}$ and $X_{\g{p}}$ the projection of $X\in\g{g}$ onto $\g{k}$ and $\g{p}$, respectively. We define a symmetric bilinear map $U\colon\g{p}\times\g{p}\rightarrow \g{p}$ by
	\[2 \langle U(X, Y),Z \rangle=\langle [Z,X]_{\g{p}},Y\rangle + \langle X, [Z,Y]_{\g{p}}\rangle, \quad \text{where $X,Y,Z\in\g{p}$,} \]
	and $\langle\cdot,\cdot \rangle$ denotes the inner product on $\g{p}$ induced by the Riemannian metric on $M$.
	The reductive decomposition $\g{g}=\g{k}\oplus\g{p}$ is \textit{naturally reductive} if $U$ is identically zero. In particular, if $U\equiv0$, every geodesic $\gamma$ of $M$ passing through $o\in M$ is given by  $t\mapsto\Exp(tX)$ with $X\in\g{p}$, where $\Exp$ denotes the Lie exponential map of $\s{G}$. A homogeneous space where every geodesic is an orbit of a $1$-parameter subgroup of the isometry group is said to be a \textit{geodesic orbit space}, or, for short, g.\ o.\ space. Thus, naturally reductive spaces are g.\ o. Moreover, we say that the reductive  decomposition $\g{g}=\g{k}\oplus\g{p}$ is \textit{normal homogeneous} if there exists some $\Ad(\s{G})$-invariant inner product $q$ on $\g{g}$ such that $\langle\cdot,\cdot\rangle=q_{\rvert\g{p}\times\g{p}}$ and $\g{p}=\g{k}^\perp$, where $\g{k}^{\perp}$ denotes the orthogonal complement of $\g{k}$ in $\g{g}$ with respect to $q$. For further generalizations of this concept, one can consult \cite[Chapter 6]{berestovskii}. It turns out that every normal homogeneous reductive decomposition is naturally reductive.  Indeed, we have the chain of strict inclusions
	\[
	\left\{ \parbox[c]{0.24\linewidth}{
		\centering{
			Normal\\
			homogeneous spaces}
	} \right\}
	\subsetneq\left\{ \parbox[c]{0.24\linewidth}{
		\centering{
			Naturally reductive\\ homogeneous spaces}
	} \right\}\subsetneq
	\left\{ \parbox[c]{0.17\linewidth}{
		\centering{
			Geodesic orbit \\ spaces}
	} \right\}.
	\]

	Let us consider the canonical connection $\nabla^c$ associated with a reductive decomposition $\g{g}=\g{k}\oplus\g{p}$, which is the unique $\s{G}$-invariant affine connection on $M$ such that
	\begin{equation*}
		\label{eq:nablac}
		(\nabla^c_{X^*}Y^*)_o=(-[X,Y]_{\g{p}})^*_o, \quad\text{where $X,Y\in\g{p}$.}
	\end{equation*}
	We can express the Levi-Civita connection of $M$ at $o$ as
	\begin{equation}
		\label{eq:nabla}
		(\nabla_{X^*} Y^*)_o=\left(-\frac{1}{2}[X,Y]_{\g{p}} + U(X,Y) \right)^*_o, \quad \text{where $X,Y\in \g{p}$.}
	\end{equation}
	The difference tensor $D$ at $o\in M$ is defined as $D=(\nabla-\nabla^c)_o$, and using the identification of $T_o M$ and $\g{p}$, we have
	\begin{equation}
		\label{eq:dtensor}
		D_X Y=\frac{1}{2}[X,Y]_{\g{p}}+ U(X,Y), \quad\text{for every $X,Y\in\g{p}$.}  
	\end{equation}

	Using the formula for the Levi-Civita connection and the identification of $\g{p}$ with $T_o M$, we can compute the curvature tensor of $M$, which is given by
	\begin{equation}
		\begin{aligned}
			\label{eq:curvature31}
			R_o(X,Y)Z&=\frac{1}{2}[Z,[X,Y]_{\g{p}}]_{\g{p}}-[[X,Y]_{\g{k}},Z]]_{\g{p}}- U(Z,[X,Y]_{\g{p}}) + \frac{1}{4}[[Z,Y]_{\g{p}}, X]_{\g{p}}\\
			&-\frac{1}{2}U([Z,Y]_{\g{p}},X)-\frac{1}{2}[U(Z,Y),X]_{\g{p}}+ U(U(Z,Y),X)-\frac{1}{4}[[Z,X]_{\g{p}},Y]_{\g{p}}\\ &+\frac{1}{2}U([Z,X]_{\g{p}},Y)+\frac{1}{2}[U(Z,X),Y]_{\g{p}} - U(U(Z,X),Y),
		\end{aligned}
	\end{equation}
	where $X,Y,Z\in\g{p}$.
	We end up this section with a remark that will be useful to compute the covariant derivatives of the curvature tensor at the base point $o\in M=\s{G}/\s{K}$.
	\begin{remark}
		\label{rem:covderivative}
		By \cite[Proposition 2.7, Chapter X]{kobayashi} the curvature tensor $R$ of a Riemannian homogeneous space $M=\s{G}/\s{K}$ satisfies $\nabla^c R=0$, since it is $\s G$-invariant. Let $\g{g}=\g{k}\oplus\g{p}$ be a reductive decomposition for $M=\s{G}/\s{K}$. Then, using the definition of the difference tensor and the identification of $\g{p}$ with $T_o M$ we have
		\begin{align*}
			(\nabla_V R)(X,Y,Z)&=((\nabla_V - \nabla^c_V) R)(X,Y,Z)\\
			&= D_V R(X,Y)Z- R(D_V X,Y)Z-R(X,D_V Y)Z-R(X,Y)D_VZ,
		\end{align*}
		where $X,Y,Z,V\in\g{p}$.
	\end{remark}

	\subsection{Rank one symmetric spaces and their totally geodesic submanifolds}
	\label{subsect:rkonetg}
	In this section we discuss some facts  about totally geodesic submanifolds in symmetric spaces of rank one that will be useful later.
	
	Let $\bar{M}=\bar{\s{G}}/\bar{\s{K}}$ be a connected Riemannian symmetric space, where $\bar{\s{G}}$ is up to some finite quotient equal to $\mathrm{Isom}^0(\bar{M})$, the connected component of the identity of the isometry group of $\bar{M}$, and $\bar{\s{K}}$ is the isotropy at some fixed point $p\in \bar{M}$. Let $s_q$ be the geodesic reflection of $\bar{M}$ at $q\in \bar{M}$. It turns out that a submanifold $\Sigma$ of $\bar{M}$ is a complete totally geodesic submanifold of $\bar{M}$ if and only if $s_q \Sigma=\Sigma$ for every $q\in\Sigma$. Recall that the rank of a symmetric space is defined as the dimension of a maximal flat totally geodesic submanifold.

	The complete Riemannian manifolds with non-zero constant sectional curvature are: real hyperbolic spaces,  round spheres and real projective spaces. These are symmetric spaces of  rank one. Furthermore, the remaining symmetric spaces of rank one are hyperbolic and projective spaces over $\C$, $\H$ and $\O$. The totally geodesic submanifolds in symmetric spaces of rank one were classified by Wolf in~\cite{wolfrankone}, see Figure~\ref{fig:tgrk1}.
	
	Let $\kappa\in(0,\infty)$. We  denote by $\s{S}^n_{\kappa}$, $\R \s P^n({\kappa})$, and $\R \s H^n({\kappa})$, the round sphere of sectional curvature $\kappa$, the real projective space of sectional curvature $\kappa$ and the real hyperbolic space with sectional curvature $-\kappa$, respectively. In the case that $\kappa=1$ or $\kappa=-1$, we  write $\s{S}^n$, $\R \s P^n$, or $\R \s H^n$, respectively.

	\begin{table}[]
		
		\begin{minipage}{.5\textwidth}
			\centering
			\begin{tabular}{lll}
				\hline
				\multicolumn{1}{|l}{$\mathbb{C}\s P^n$} &                                                          & \multicolumn{1}{l|}{}                  \\ \hline
				\multicolumn{1}{|l|}{}               & \multicolumn{1}{l|}{$\mathbb{C}\s P^k$}                     & \multicolumn{1}{c|}{$2\leq k\leq n-1$} \\
				\multicolumn{1}{|l|}{}               & \multicolumn{1}{l|}{$\mathbb{R}\s P^k$}                     & \multicolumn{1}{c|}{$2\leq k\leq n$\phantom{......}} \\
				\multicolumn{1}{|l|}{}               & \multicolumn{1}{l|}{$\s{S}_4^k$}                & \multicolumn{1}{l|}{$1\leq k\leq 2$\phantom{...} } \\ \hline
				\multicolumn{1}{|l}{$\mathbb{H}\s P^n$} &                                                          & \multicolumn{1}{l|}{}                  \\ \hline
				\multicolumn{1}{|l|}{}               & \multicolumn{1}{l|}{$\mathbb{H}\s P^k$}                     & \multicolumn{1}{c|}{$2\leq k\leq n-1$} \\
				\multicolumn{1}{|l|}{}               & \multicolumn{1}{l|}{$\mathbb{C}\s P^k$}                     & \multicolumn{1}{c|}{$2\leq k\leq n$\phantom{......}} \\
				\multicolumn{1}{|l|}{}               & \multicolumn{1}{l|}{$\mathbb{R}\s P^k$}                & \multicolumn{1}{c|}{$2\leq k\leq n$\phantom{......}} \\
				\multicolumn{1}{|l|}{}               & \multicolumn{1}{l|}{$\s{S}^k_4$} & \multicolumn{1}{c|}{$1\leq k\leq 4$\phantom{......}}   \\ \hline
				\multicolumn{1}{|l}{$\mathbb{O}\s P^2$} &                                                          & \multicolumn{1}{l|}{}                  \\ \hline
				\multicolumn{1}{|l|}{}               & \multicolumn{1}{l|}{$\mathbb{H}\s P^2$}                     & \multicolumn{1}{l|}{}  \\
				\multicolumn{1}{|l|}{}               & \multicolumn{1}{l|}{$\mathbb{C}\s P^2$}                     & \multicolumn{1}{l|}{}  \\
				\multicolumn{1}{|l|}{}               & \multicolumn{1}{l|}{$\mathbb{R}\s P^2$}                & \multicolumn{1}{l|}{}                  \\
				\multicolumn{1}{|l|}{}               & \multicolumn{1}{l|}{$\s{S}_4^k$}                     & \multicolumn{1}{c|}{\hspace{-0.6cm}$1\leq k \leq 8$} \\ \hline
				&                                                          &                                       
			\end{tabular}
		\end{minipage}%
		\begin{minipage}{0.5\textwidth}
			
			\begin{tabular}{lll}
				\hline
				\multicolumn{1}{|l}{$\mathbb{C}\s H^n$} &                                                          & \multicolumn{1}{l|}{}                  \\ \hline
				\multicolumn{1}{|l|}{}               & \multicolumn{1}{l|}{$\mathbb{C}\s H^k$}                     & \multicolumn{1}{c|}{$2\leq k\leq n-1$} \\
				\multicolumn{1}{|l|}{}               & \multicolumn{1}{l|}{$\mathbb{R}\s H^k$}                     & \multicolumn{1}{c|}{$2\leq k\leq n$\phantom{......}} \\
				\multicolumn{1}{|l|}{}               & \multicolumn{1}{l|}{$\R \s H^k(4)$}                & \multicolumn{1}{l|}{$1\leq k\leq 2$\phantom{...} } \\ \hline
				\multicolumn{1}{|l}{$\mathbb{H}\s H^n$} &                                                          & \multicolumn{1}{l|}{}                  \\ \hline
				\multicolumn{1}{|l|}{}               & \multicolumn{1}{l|}{$\mathbb{H}\s H^k$}                     & \multicolumn{1}{c|}{$2\leq k\leq n-1$} \\
				\multicolumn{1}{|l|}{}               & \multicolumn{1}{l|}{$\mathbb{C}\s H^k$}                     & \multicolumn{1}{c|}{$2\leq k\leq n$\phantom{......}} \\
				\multicolumn{1}{|l|}{}               & \multicolumn{1}{l|}{$\mathbb{R}\s H^k$}                & \multicolumn{1}{c|}{$2\leq k\leq n$\phantom{......}} \\
				\multicolumn{1}{|l|}{}               & \multicolumn{1}{l|}{$\R \s H^k(4)$} & \multicolumn{1}{c|}{$1\leq k\leq 4$\phantom{......}}   \\ \hline
				\multicolumn{1}{|l}{$\mathbb{O}\s H^2$} &                                                          & \multicolumn{1}{l|}{}                  \\ \hline
				\multicolumn{1}{|l|}{}               & \multicolumn{1}{l|}{$\mathbb{H}\s H^2$}                     & \multicolumn{1}{l|}{}  \\
				\multicolumn{1}{|l|}{}               & \multicolumn{1}{l|}{$\mathbb{C}\s H^2$}                     & \multicolumn{1}{l|}{}  \\
				\multicolumn{1}{|l|}{}               & \multicolumn{1}{l|}{$\mathbb{R}\s H^2$}                & \multicolumn{1}{l|}{}                  \\
				\multicolumn{1}{|l|}{}               & \multicolumn{1}{l|}{$\R \s H^k(4)$}                     & \multicolumn{1}{c|}{\hspace{-0.6cm}$1\leq k \leq 8$} \\ \hline
				&                                                          &                                       
			\end{tabular}
		\end{minipage}
		\caption{Totally geodesic submanifolds in symmetric spaces of rank one and non-constant sectional curvature, up to congruence.}
		\label{fig:tgrk1}
	\end{table}
	
	Now we assume that $\bar{M}= \bar{\s{G}}/\bar{\s{K}}$ has rank one and that the  minimal absolute value of its sectional curvature is equal to $1$. One has that $\bar{\s{K}}$, via the isotropy representation, acts transitively on  $\s{S}(T_p\bar{M})$, the unit sphere of $T_p\bar{M}$. Moreover, $\bar{\s{K}}$ acts transitively on the geodesic sphere $\s{S}_t(p)$ of radius $t$ centered at $p\in\bar M$, for each $t\in(0,\inj(p))$.   This implies that geodesic spheres are $\bar{\s{K}}$-homogeneous spaces. Recall that the Jacobi operator $\bar{R}_v$  associated with $v\in T_p\bar M$ is defined by $\bar{R}_v(X)=\bar{R}(X,v,v)$ for every $X\in T_p \bar M$.  In this case, the Jacobi operators $\{ \bar{R}_v: v\in T_p\bar{M}, \, \Vert v\Vert =t\}$ are all conjugate by elements of $\bar{\s{K}}$.
	
	Let us assume that $\bar M$ does not have constant sectional curvature.
	It is well known that, for any non-zero $v\in T_p\bar M$, the Jacobi operator $\bar{R}_{v}$ restricted to $T_p \bar{M}\ominus \R v$ has only two different eigenvalues $\lambda _1$ and $\lambda _2$, where $\ominus$ denotes the orthogonal complement. Let $v\in T_p \bar{M}$ be of unit length. Then,  $\lambda _1=\pm 4,  \ \lambda _2 = \pm 1,$
	where the signs are positive if $\bar{M}$ is of  compact type, and negative if $\bar{M}$ is of  non-compact type.  Observe that  $\bar{R}_v$ has exactly one more eigenvalue that is equal to zero and has associated eigenspace $\mathbb R v$. Let $\mathbb{V}_i(v)$ be the eigenspace of $\bar R_v$ associated with $\lambda_i$ for each $i\in\{1,2\}$. When  $\bar{M}$ is of compact type,  ${\inj}(\bar{M})=\pi/2$. However, if $\bar{M}$ is of  non-compact type, then it follows by the Cartan-Hadamard Theorem that  $\inj(\bar{M}) = +\infty$.

	Fix $p\in\bar M$. For each $i\in\{1,2\}$ we define  $\mathbb V_i^t(v)$  to be the parallel transport of $\mathbb V_i(v)$ along the unit speed geodesic $\gamma _v\colon[0,\inj(\bar M))\rightarrow \bar M$ with initial conditions $p\in\bar{M}$ and $v\in T_p\bar{M}$. The equation 
	$\bar{\nabla} \bar{R} =0$ implies that $\mathbb V_i^t(v)$ is the eigenspace of $\bar{R}_{ \dot\gamma_v(t)}$ associated with 
	$\lambda _i$, for each $i\in\{1, 2\}$.  Moreover, $\mathbb V_i^t(v)\perp \dot\gamma_v(t)$ and then we have the following orthogonal decomposition 
	\begin{equation*}\label{AA3}
		T_{\gamma _v(t)}\s{S}_t(p) = \mathbb V_1^t(v) \oplus \mathbb V_2^t(v).
	\end{equation*}
	
	Let $\bar{\s{K}}_v$ be the isotropy of $\bar{\s{K}}$ at $v$. Since $\bar{\s{K}}$ preserves the curvature tensor $\bar{R}$ of $\bar M$ at $p$, then $\bar{\s{K}}_v$ commutes with the Jacobi operator $\bar{R}_v$ and therefore leaves $\mathbb V_i(v)$ invariant. Notice that  $\bar{\s{K}}_v$ coincides with the isotropy $\bar{\s{K}}_{\gamma _v(t)}$ of the action of $\bar{\s{K}}$ on the geodesic sphere $\s{S}_t(p)$ at $\gamma _v(t)$ for every $t\in(0,\inj(\bar{M}))$. Moreover,  
	$\bar{\s{K}}_{\gamma _v(t)}$ leaves invariant $\mathbb V_i^t(v)$ since $\bar{\s{K}}_{\gamma _v(t)}$ and $\bar{R}_{ \dot\gamma_v(t)}$ commute. Hence, for any fixed $t\in(0,\inj(\bar{M}))$, the subspace 
	$\mathbb V_i^t(v)$ extends to a $\bar{\s{K}}$-invariant distribution $\mathcal D_i^t$ on the geodesic sphere $\s{S}_t(p)$. Thus, 
	\begin{equation*}\label {AA4}
		T\s{S}_t(p) = \mathcal D^t_1\oplus \mathcal D^t_2.
	\end{equation*}
	The distributions $\mathcal{D}_1$ and $\mathcal{D}_2$ coincide with the vertical and horizontal distributions, respectively, defined by the Hopf fibrations.

	\begin{remark}
		\label{rem:shapeoperatoreigenspaces}
		Using Equation~\eqref{eq:jacobisphere} and solving  the Jacobi Equation~(see e.g.\ \cite[p.~393]{BCO}) explictly, one can prove that $\mathbb{V}^t_1(v)$ and $\mathbb{V}^t_2(v)$ are the eigenspaces of the shape operator $\mathcal{S}^t$ with respect to the outer normal unit direction of $\s{S}_t(p)$ in $\bar{M}$ with corresponding eigenvalues $\beta^t_1$ and $\beta^t_2$, respectively. These eigenvalues are given by the following expressions (see e.g.~\cite[p.~16]{bettiol}):
		\begin {equation}\label {eq:b1}
		\beta _1^t = \left\{
		\begin{array}{ll}
			- 2\cot (2t)  & \mbox{if } \bar M  \mbox{ is of  compact type,} \\
			- 2\coth (2t)  &  \mbox{if } \bar M  \mbox{ is of  non-compact type.} 
		\end{array}
		\right.
	\end{equation}
	
	\begin {equation}\label {eq:b2}
	\beta _2^t = \left\{
	\begin{array}{ll}
		-\cot ( t)  & \mbox{if } \bar M  \mbox{ is of  compact type,} \\
		- \coth (t)  &  \mbox{if } \bar M  \mbox{ is of  non-compact type.} 
	\end{array}
	\right.
\end{equation}
Notice that  $\beta _1^t \neq  \beta _2^t$ for every $t\in(0,\inj(\bar{M}))$. Moreover, since the connected group $\bar{\s{K}}$ acts isometrically on  $\s{S}_t(p)$, the shape operator  $\mathcal{S}^t$ is a $\bar{\s{K}}$-invariant tensor. Then, 
$\mathcal D ^t_i$ is an $\mathcal{S}^t$-invariant distribution of $\s{S}_t(p)$ and $
\mathcal{S}^t_{\vert \mathcal D_i^t} = \beta _i ^t \Id_{\vert \mathcal D_i^t}$, for each $i\in\{1,2\}$.
\end{remark}

\begin{remark}\label{fact}  If $\bar{M}$ is of compact type, it is well known that $\mathbb{V}_1(v) \oplus \R v\subset T_p\bar M$ is the tangent space of a totally geodesic  sphere of $\bar{M}$ isometric  to $\s{S}_4^{m+1}$, which we denote by $\s{S}^{m+1}(v)$, where $m=\rank(\mathcal{D}^t_1)$. This is a so-called Helgason sphere, see \cite{helgasonsphere} and \cite{ohnita} for more information on Helgason spheres.

This implies that $\mathcal D ^t_1$ is an integrable
distribution of $\s{S}_t(p)$ whose integral submanifold through $\gamma_w(t)=\overline{\exp} (tw)$, where $w\in T_p\bar{M}$ is of unit length,  is the intersection of $\s{S}^{m+1}(w)$ with $\s{S}_t(p)$. Moreover, by Corollary \ref {cor:geodsphereinter}, $\mathcal D ^t_1$ is an autoparallel distribution of $\s{S}_t(p)$.
The same is true when $\bar{M}$ is of non-compact type. But in this case $\mathbb{V}_1(v) \oplus \mathbb Rv$  is the tangent space at $p$ of a totally geodesic hyperbolic space of $\bar M$ that is isometric to $\R \s H^{m+1}(4)$.
\end{remark}
\vspace{.15cm}
We finish this section with a remark that will be useful in Section~\ref{sect:wellpos}.

\begin{remark}\label{remark:CC1} Let $\bar{M}= \bar{\s{G}}/\bar{\s{K}}$ be a simply connected compact symmetric space with non-constant sectional curvature and minimal sectional curvature $1$. Recall that  $\inj(\bar{M})$ is $\pi/2$. Let $p\in \bar{M}$ and let $M$ be a complete totally geodesic submanifold of $\bar{M}$ that contains $p$. 
Then, $M$ is a rank one symmetric space of compact type  with $\inj(M) = \pi/2$ or a closed geodesic of length $\pi$.

Let  $B^E_{\pi/2} (0)$  be the Euclidean  open ball of $T_p\bar{M}$ of radius $\pi/2$ centered at the origin. Hence, $\overline{\exp}_p\colon B^E_{\pi/2} (0)\to B_{\pi/2} (p) $ is a diffeomorphism, where  $B_{\pi/2} (p)$ is the open ball of $\bar{M}$ with center $p$ and radius $\pi/2$. Now, since $\inj(M) = \pi/2$,  the points $x$ of $M$ that are in the complement of $\overline{\exp}_p (B^E_{\pi/2} (0)\cap T_pM) $ must be at a distance of $\pi/2$ from $p$  in $\bar{M}$. Indeed, if $\gamma\colon[0,\pi/2] \to M$ is a  minimizing unit speed geodesic in $M$ from $p$ to $x$, we have that $\gamma$ restricted to $[0,\pi/2)$ is a  minimizing geodesic of $\bar M$, since $M$ is totally geodesic. Thus, the distance in $\bar{M}$ from $p$ to $x$ is $\pi/2$. 
This implies  that 
$M \cap B_{\pi/2} (p) =  \overline{\exp}_p (B^E_{\pi/2} (0)\cap T_pM)$ and then
\[M\cap \s{S}_t(p)=\overline{\exp}_p (B^E_{\pi/2} (0)\cap T_pM)\cap\s{S}_t(p), \quad \text{where $t\in(0,\pi/2)$.}    \]

\end{remark}
\subsection[The geometry of Hopf-Berger spheres]{The geometry of Hopf-Berger spheres}
\label{subsect:geomhopf}
In this section we summarize some facts concerning the geometry of Hopf-Berger spheres.

It is well-known that any homogeneous metric on the even-dimensional sphere $\s{S}^{2n}$ is isometric to a round metric, while homogeneous metrics on odd-dimensional spheres $\s{S}^{2n+1}$ of dimension bigger than three are homothetic to a metric lying in one of the following families, see~\cite{Z2}:
\begin{enumerate}
\item[(1)] a $1$-parameter family of $\s{U}_{n+1}$-invariant metrics on $\s{S}^{2n+1}$,
\item[(2)] a $1$-parameter family of $\s{Sp}_1\s{Sp}_{n+1}$-invariant metrics on $\s{S}^{4n+3}$,
\item[(3)] a $1$-parameter family of $\s{Spin}_9$-invariant metrics on $\s{S}^{15}$,
\item[(4)] a $3$-parameter family of $\s{Sp}_{n+1}$-invariant metrics on $\s{S}^{4n+3}$.
\end{enumerate}
It turns out that metrics in (2) lie in (4). Moreover, metrics in (1), (2) and (3) can be obtained by rescaling the round metric of the total space of a Hopf fibration in the direction of the fibers. In what follows, we will outline the construction of Hopf-Berger spheres.

Let $\{\s{S}_{\mathbb{F},\tau}^n\}_{\tau>0}$ denote the family of complex, quaternionic or octonionic Hopf-Berger spheres of dimension $n$, according  to the value of~ $\mathbb{F}\in\{\C,\H,\O\}$. These were proved to be geodesic orbit spaces in~\cite{tamarugo}. This means that every geodesic of $\s{S}_{\mathbb{F},\tau}^n$ is the orbit of a $1$-parameter subgroup of $\Isom(\s{S}_{\mathbb{F},\tau}^n)$. More generally, Hopf-Berger spheres were proved to be weakly-symmetric spaces in~\cite[Theorem 2]{BV}. A space $M$ is said \textit{weakly symmetric} if given two points $p,q\in M$ one can find an isometry $\varphi$ of $M$ mapping $p$ to $q$ and $q$ to $p$. Weakly symmetric spaces are in particular geodesic orbit spaces, see~\cite{BKV} or \cite[Theorem~5.3.6]{berestovskii}.

Furthermore, according to~\cite{Z1}, $\s{S}_{\mathbb{F},\tau}^n$ admits, for each $\tau>0$, a naturally reductive decomposition when  $\mathbb{F}\in\{\C,\H\}$. However,  $\s{S}_{\mathbb{\O},\tau}^{15}$ does not admit a naturally reductive decomposition for most $\tau\in(0,+\infty)$, see~\cite{Z2}.

On the one hand, when $\tau=1$, Hopf-Berger spheres are round. Through this article, we  denote  round spheres with constant sectional curvature equal to $\kappa>0$  by  $\s{S}^n_{\kappa}$.
On the other hand, every Hopf-Berger sphere with $\tau\neq 1$ is homothetic to a geodesic sphere of a rank one symmetric space $\bar{M}$, see~\cite{Z2}. We  consider the metrics of these symmetric spaces rescaled in such a way that the minimal absolute value of their sectional curvatures is equal to $1$. Notice that two geodesic spheres of $\bar{M}$ of the same radius are congruent. Furthermore, it can be proved that every geodesic sphere $\s{S}(t)$ of radius $t$ in a symmetric space of rank one $\bar{M}$ can be obtained by performing a homothety (cf.~\cite{bettiol}) of ratio $\alpha$ to $\s{S}^n_{\F,\tau}$, with $\F\in\{\C,\H,\O\}$ and $\tau\in(0,+\infty)$, where the values of $t,\alpha$ and $\tau$ are related by:
\begin{equation}
\label{eq:radius_tau_homothety}
\begin{aligned}
	&t=\arccos(\sqrt{\tau}),&	&\alpha=\sqrt{1-\tau}, &  &\text{for } 0< \tau< 1,&  &\text{when $\bar{M}$ is  compact},\\
	&t=\arccosh(\sqrt{\tau}),&  &\alpha=\sqrt{\tau-1}, & &\text{for } 1< \tau< +\infty,& &\text{when $\bar{M}$ is not compact.}
\end{aligned}
\end{equation}

The complex and quaternionic Hopf fibrations are invariant under the action of $\s{U}_{n+1}$ and $\s{Sp}_{n+1}\s{Sp}_{1}$, respectively. The complex and quaternionic Hopf-Berger spheres can be expressed as $\s{S}^{2n+1}_{\C,\tau}=\s{U}_{n+1}/\s{U}_n$ and  $\s{S}^{4n+3}_{\H,\tau}=\s{Sp}_{n+1}/\s{Sp}_n$. Now we will describe the reductive decomposition for $\s{S}^n_{\F,\tau}=\s{G}/\s{K}$, for $\F\in\{\C,\H\}$, where the Lie groups $\s{G}$ and $\s{K}$ are as above.  Let $\g{g}$ be the Lie algebra of $\s{G}$ for $\F\in\{\C,\H\}$. We denote by $\mathrm{Im}(x)$ the imaginary part of an element $x\in\F$ and by $A^*$ the conjugate transpose of a matrix $A$ with entries in $\F\in\{\C,\H\}$. Then, we consider the following subspaces of $\g{g}$:
\[ \g{k}=\left(
\begin{array}{c|c}
Z & 0 \\
\hline
0 & 0
\end{array}
\right), \quad \g{p}_1=\left(
\begin{array}{c|c}
0 & 0 \\
\hline
0 & \mathrm{Im}(x)
\end{array}
\right), \quad \g{p}_2=\left(
\begin{array}{c|c}
0 & v \\
\hline
-v^* & 0
\end{array}
\right),   \]
where $x\in\F$, $v\in\F^n$ and $Z$ belongs to $\g{u}_n$ or $\g{sp}_n$, when $\F=\C$ or  $\F=\H$, respectively.
Observe that $\g{k}$ is the Lie algebra of $\s{K}$ and $[\g{k},\g{p}_1]=0$. 
On the one hand,  $\g{p}_1$ is spanned by $\{X_l\}^{\dim_{\R} \F-1}_{l=1}$, where each $X_l$ is the matrix filled with zeroes except the last entry, which is equal to one of the imaginary units $i,j$ or $k$, for $X_1, X_2$ or $X_3$, respectively. On the other hand, $\g{p}_2$ is identified with $\F^n$, and thus, we can consider the set $\{Y_j\}^n_{j=1}$ of $\g{p}_2$ where $Y_j$ is identified with the $j$-th element of the canonical basis of $\F^n$ for $\F\in\{\C,\H\}$.  Furthermore, the adjoint representation of $\s{G}$ restricted to $\s{K}$ acts on $\g{p}_2$ as the standard representation of  $\s{U}_n$ or $\s{Sp}_n$, depending on whether $\F$ is equal to $\C$ or $\H$. 

Now let $\F=\O$. The octonionic Hopf fibration is invariant under the action of $\s{Spin}_9$ and $\s{S}^{15}_{\O,\tau}=\s{Spin}_9/\s{Spin}_7$. For positive integers $i,j$, such that $i<j<9$, let $E_{ij}\in\g{spin}_9$ be the matrix filled with zeroes except for the entries $(i,j)$ and $(j,i)$ which are equal to $1$ and $-1$, respectively. We define the following elements in $\g{g}=\g{spin}_9$.
\begin{align*}
X_1&= E_{15} + E_{26} + E_{37} + E_{48},&
X_2&= E_{17} + E_{28} - E_{35} - E_{46},\\
X_3&= E_{13} - E_{24} - E_{57} + E_{68},&
X_4&= E_{16} - E_{25} - E_{38} + E_{47},\\
X_5&= E_{18} - E_{27} + E_{36} - E_{45},&
X_6&=E_{12} + E_{34} - E_{56} - E_{78},\\
X_7&=E_{14} + E_{23} - E_{58} - E_{67},&
Y_1&=2 E_{19},\\
Y_2&=2 E_{29},&
Y_3&=2 E_{39},\\
Y_4&=2 E_{49},&
Y_5&=2 E_{59},\\
Y_6&=2 E_{69},&
Y_7&=2 E_{79},\\
Y_8&=2 E_{89}.
\end{align*}

Then, we consider $\g{p}_1$ and $\g{p}_2$ as the subspaces of $\g{g}$ spanned by $\{X_i\}_{i=1}^7$ and $\{Y_i\}_{i=1}^8$, respectively. We define $\g{k}$ as the orthogonal complement of $\g{p}_1\oplus\g{p}_2$ with respect to the Killing form of $\g{g}=\g{spin}_9$. It turns out that $\g{k}\cong \g{spin}_7$ is the Lie algebra of the isotropy $\s{K}=\s{Spin}_7$ for the action of $\s{G}=\s{Spin}_9$ in $\s{S}^{15}_{\O,\tau}$ (cf.\ \cite[p. 476]{verdiani-ziller}). Moreover, $\g{k}$ acts on $\g{p}_1$ as the standard representation of $\g{so}_7= \g{spin}_7$ and on $\g{p}_2$ as the irreducible spin representation of dimension eight.

To sum up, for each Hopf-Berger sphere $\s{S}^n_{\F,\tau}$, we can check that $[\g{k},\g{p}_1\oplus\g{p}_2]\subset\g{p}_1\oplus\g{p}_2$. Thus, $\g{g}=\g{k}\oplus\g{p}$, with $\g{p}:=\g{p}_1\oplus\g{p}_2$, 
defines a reductive decomposition for $\s{S}^n_{\F,\tau}=\s G/ \s K$.

Notice that for each $X_i\in\g{p}_1$, the map $J\colon\g{p}_2\rightarrow \g{p}_2$ given by $J_{X_i} Y:=[Y,X_i]$ defines a complex structure in $\g{p}_2$. In the following, we simply write $J_i$ to denote $J_{X_i}$, for each $i\in\{1,\ldots, \dim \F-1\}$. Let $m:=\dim\g{p}_2/\dim\F$. We consider the unique $\s K$-invariant inner product $\langle\cdot,\cdot\rangle$ in $\g{p}$ such that $\{X_i/\sqrt{\tau},Y_j,J_i Y_j\}_{i,j}^{\dim\F-1,m}$ and $\{X_i/(2\sqrt{\tau}),Y_j\}_{i,j}^{\dim\F-1,8}$ are orthonormal bases for $\g{p}$ when $\F\in\{\C,\H\}$ and $\F=\O$, respectively. Thus, the tangent space of $\s{S}^n_{\F,\tau}$ at a base point $o$ is identified with $\g{p}$ as inner product spaces. Furthermore, $\g{p}_1$ and $\g{p}_2$ can be identified with the vertical and horizontal spaces  at $o\in\s{S}^n_{\F,\tau}$ defined by  the corresponding Hopf fibration, respectively.

It can be checked that $J_i$ is a skew-symmetric endomorphism of $\g{p}_2$ with respect to the inner product $\langle\cdot,\cdot\rangle$, hence $J_i$ is an orthogonal complex structure for each subindex $i\in\{1,\ldots,\dim \F-1\}$. On the one hand, if $\F=\H$, we have the relations $J_j J_i= J_k$, where $(i,j,k)$ is a cyclic permutation of $(1,2,3)$. On the other hand, if $\F=\O$, the relations among the complex structures $\{J_i\}_{i=1}^7$ are given by the labeling of the Fano plane indicated in Figure~\ref{fig:fanoplane}. The Fano plane has 7 points and 7 lines. Each line $I$ contains exactly three points, and each of these triples has a cyclic ordering shown by the arrows in such a way that if $(i,j,k)$ are cyclically ordered,  we have $J_j J_{i\rvert\H_I Y_1}=J_{k\rvert\H_I Y_1}$, where $\H_I Y_1:=\spann\{Y_1,J_l Y_1\}^3_{l=1}$.
\begin{figure}[h!]
\centering
\includegraphics[width=50mm]{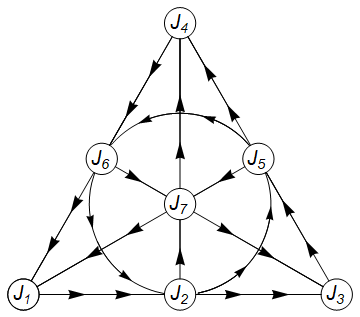}
\caption{Fano plane with the appropiate labeling. \label{fig:fanoplane}}
\end{figure}

\begin{remark}
\label{rem:isom}
The isometry group of $\s{S}^n_{\F,\tau}$, $\tau\neq1$, can be identified with the  isotropy of the symmetric space of rank one $\bar M$ where $\s{S}^n_{\F,\tau}$ is embedded as a geodesic sphere, see~\cite[\S 8.5]{arvthesis}. Therefore,
\[\Isom(\s{S}^{2n+1}_{\C,\tau})\cong\s{U}_{n+1}\rtimes\mathbb{Z}_2, \phantom{a} \Isom(\s{S}^{4n+3}_{\C,\tau})\cong \s{Sp}_{n+1}\s{Sp}_1, \phantom{a} \Isom(\s{S}^{15}_{\O,\tau})\cong\s{Spin}_9,\]
where $\s{Sp}_{n+1}\s{Sp}_1:= \s{Sp}_{n+1}\times_{\mathbb{Z}_2} \s{Sp}_1$. 
\end{remark}

\begin{remark}
\label{rem:isotropyrep}
Let $ v_1\in \g{p}_1$ and $v_2\in \g{p}_2$ be non-zero elements of lengths $r$ and $s$, respectively. Moreover, denote by  $\s{S}_1(r)$ the sphere of radius $r$ of $\g{p}_1$,  and by 
$\s{S}_2(s)$ the sphere of radius $s$ of $\g{p}_2$. If $\F=\O$, then $\s{K}$ acts polarly on $\g{p}$ and its orbits are given by 
$$\s{K}\cdot (v_1+v_2) = \s{S}_1(r)\times \s{S}_2(s).$$
When $\F=\H$, we can extend the group $\s{G}=\s{Sp}_{n+1}$ acting transitively on $\s{S}^{4n+3}_{\H,\tau}$ to the Lie group $\check{\s{G}}=\s{Sp}_{n+1}\s{Sp}_1$ which also acts transitively and isometrically on $\s{S}^{4n+3}_{\H,\tau}$. The isotropy of the action of $\check{\s{G}}$ at $o\in\s{S}^{4n+3}_{\H,\tau}$ is  $\check{\s{K}}=\s{Sp}_{n}\s{Sp}_1$. Then $\check{\s{K}}$ acts polarly on $\g{p}$ and its orbits are given by 
$$\check{\s{K}}\cdot (v_1+v_2) = \s{S}_1(r)\times \s{S}_2(s).$$

Now, let $\F=\C$ and consider $\check{\s{K}}$, the disconnected Lie group generated by $\s{K}$ and $\sigma$, where $\sigma$ is the  isometry lying in the isotropy of $\s{S}^{2n+1}_{\C,\tau}$ induced by the standard conjugation in $\C^{n+1}$. Then, $\check{\s{K}}$ and $\s{K}$ act polarly in $\g{p}$ and their orbits are given by
\[\s{K}\cdot (v_1+v_2) = \{v_1\}\times \s{S}_2(s), \quad \check{\s{K}}\cdot (v_1+v_2) = \s{S}_1(r)\times \s{S}_2(s).  \]

We define the \textit{slope} of a non-vertical vector $v\in\g{p}$ as the quotient $\tfrac{||v_1||}{||v_2||}$, where $v_1$ and $v_2$ are the vertical and horizontal projections of $v$, respectively. By convention, we say that a vertical vector has slope $\infty$. The previous discussion implies that two vectors in $\g{p}$ with the same length lie in the same orbit of the full isotropy representation of $\s{S}^n_{\F,\tau}$ if and only if they have the same slope.
\end{remark}

\subsubsection{The curvature tensor of Hopf-Berger spheres}
\label{subsubsect:curvaturetensorhb}
Now we compute the curvature tensor $R$ and the difference tensor $D$ of $\s{S}^n_{\F,\tau}$ when $\F\in\{\C,\H\}$. Notice that $D$  restricted to $\g{p}_1$ is  zero and satisfies $D_X Y=-D_Y X$ for every  $X,Y\in\g{p}_2$.
Let $Y,Z,W\in\g{p}_2$ spanning a totally real subspace of $\g{p}_2$ and let $(i,j,k)$ be a cyclic permutation of $(1,2,3)$. Then, we have
\begin{equation}
\label{eq:dtensorcomplex}
\begin{aligned}
	D_{Y} X_i&=\tau J_i Y,& D_{X_i} Y&=(\tau-1) J_i Y,& D_Y J_iY&=-X_i,\\
	D_{J_i Y} J_j Y&=X_{k},& D_Y Z&=0. 
\end{aligned}
\end{equation}
By using Equation~(\ref{eq:curvature31}), 
we have for a cyclic permutation $(i,j,k)$  of $(1,2,3)$,
\begin{equation}
\begin{aligned}
	\label{eq:rtensorcomplexquaternionic}
	R(X_i,X_j)Y&=2\tau(1-\tau)J_k Y,&   R(X_i,X_j)X_i&=-X_j, \\ R(X_i,Y)X_i&=-\tau^2 Y,&	R(X_i,Y)J_jY&=(1-\tau)X_k, \\
	R(X_i,Y)X_j&=\tau(1-\tau)J_k Y, & R(X_i,Y)Y&=\tau X_i,\\
	R(Y,J_i Y)Y&=(-4+3\tau) J_i Y ,& R(Y,J_i Y)Z&=2(-1+\tau)J_i Z,\\
	R(Y,J_i Y)X_j&=2(1-\tau)X_k,& 	R(Y,J_i Y)J_i Y&=(4-3\tau) Y,\\ 
	R(Y,Z)Y&=-Z,& R(Y,Z)J_i Y&=(-1+\tau)J_i Z,
\end{aligned}
\end{equation}
\begin{equation}
\begin{aligned}
	\label{eq:rtensorcomplexquaternionic0}
	R(X_i,X_j)X_k&=R(X_i,Y)J_i Y=R(X_i,Y)Z= R(Y,J_i Y)X_i=R(Y,J_i Y)J_j Y\\
	&=R(Y,Z)X_i=R(Y,Z)W=0.
\end{aligned}
\end{equation}
Let us consider $R_X\colon\g{p}\rightarrow\g{p}$, the Jacobi operator associated with a vector $X\in\g{p}$, which is given by $R_X(Y)=R(Y,X)X$.

Let $X\in\g{p}_1$ be of unit length. By Equations~(\ref{eq:rtensorcomplexquaternionic}) and (\ref{eq:rtensorcomplexquaternionic0}), we deduce that $R_X$ leaves  $\g{p}_1$ and $\g{p}_2$ invariant. Moreover, $R_{X{|\g{p}_2}}=\tau \Id_{\g{p}_2}$, and the eigenvalues of $R_{X{|\g{p}_1}}$ are $1/\tau$ and $0$, with multiplicity $\dim \F-2$ and $1$, respectively. Let $Y\in\g{p}_2$ be a unit vector. By Equations~(\ref{eq:rtensorcomplexquaternionic}) and (\ref{eq:rtensorcomplexquaternionic0}), we deduce that $R_Y$ leaves  $\g{p}_1$ and $\g{p}_2$ invariant. Moreover, $R_{Y{|\g{p}_1}}=\tau \Id_{\g{p}_1}$, and the eigenvalues of $R_{Y{|\g{p}_2}}$ are $4-3\tau$, $1$ and $0$ of multiplicities $\dim \F-1$, $n-2\dim \F +1$ and $1$, respectively.

Now we compute the curvature tensor $R$ and the difference tensor $D$ of $\s{S}^{15}_{\O,\tau}$. We have again that $D$  restricted to $\g{p}_1$ is  zero and satisfies $D_X Y=-D_Y X$ for every  $X,Y\in\g{p}_2$.
Let $(i,j,k)$ be an ordered triple contained in a line of the Fano plane, see Figure~\ref{fig:fanoplane}.
Then, we have
\begin{equation}
\label{eq:dtensoroctonions}
\begin{aligned}
	D_{Y_1} X_i&=2\tau J_i Y_1, & D_{X_i} Y_1&=(2\tau-1) J_i Y_1,& D_{Y_1} J_iY_1&=-X_i/2,\\ 
	D_{J_i Y_1} J_j Y_1&=-X_{k}/2. 
\end{aligned}
\end{equation}
Moreover, by using Equation~(\ref{eq:curvature31}), we get
\begin{equation}
\begin{aligned}
	\label{eq:rtensoroctonions}
	R(X_i,X_j)Y_1&=8(\tau-\tau^2)J_k Y_1,& R(X_i,X_j)X_i&=-4X_j,\\ R(X_i,Y_1)Y_1&=\tau X_i,& 	R(X_i,Y_1)X_j&=4(\tau-\tau^2)J_k Y_1,\\
	R(X_i,Y_1)X_i&=-4\tau^2 Y_1,& R(X_i,Y_1)J_jY_1&=(1-\tau)X_k,\\
	R(Y_1,J_i Y_1)Y_1&=(-4+3\tau) J_i Y_1,& R(Y_1,J_i Y_1)X_j&=2(1-\tau)X_k, \\
	R(Y_1,J_i Y_1)J_iY_1&=(4-3\tau) Y_1.
\end{aligned}
\end{equation}
\begin{equation}
\begin{aligned}
	\label{eq:rtensoroctonions0}
	R(X_i,X_j)X_k&=R(X_i,Y_1)J_i Y_1= R(Y_1,J_i Y_1)X_i=R(Y_1,J_i Y_1)J_j Y_1=0.
\end{aligned}
\end{equation}
Let $X\in\g{p}_1$ be of unit length. By Equations~(\ref{eq:rtensoroctonions}) and (\ref{eq:rtensoroctonions0}), we deduce that $R_X$ leaves $\g{p}_1$ and $\g{p}_2$ invariant. Moreover, $R_{X{|\g{p}_2}}=\tau \Id_{\g{p}_2}$ and the eigenvalues of $R_{X{|\g{p}_1}}$ are $1/\tau$ and $0$, with multiplicity $6$ and $1$, respectively. Let $Y\in\g{p}_2$ be a unit vector. By Equations~(\ref{eq:rtensoroctonions}) and (\ref{eq:rtensoroctonions0}), we have that, $R_Y$ leaves $\g{p}_1$ and $\g{p}_2$ invariant. Moreover, $R_{Y{|\g{p}_1}}=\tau \Id_{\g{p}_1}$ and the eigenvalues of $R_{Y{|\g{p}_2}}$ are $4-3\tau$ and $0$ with multiplicities $7$ and $1$, respectively.

In Table~\ref{table:jacobi} we summarize the  pairs $(\lambda,m_{\lambda})$ of eigenvalues and multiplicities for the Jacobi operator.

\begin{table}[htb]
\centering
\begin{tabular}{l|l|l|}
	& $R_X$ & $R_Y$  
	\strt \\ \hline
	\multirow{2}{5ex}{$\g{p}_1$}
	& $(0,1)$ & $(\tau,\dim(\F)-1)$  
	\strt \\ 
	& $(1/\tau,\dim(\F)-2)$  &   
	\strt \\ \cline{1-3}
	\multirow{3}{5ex}{$\g{p}_2$}
	& $(\tau,n-\dim(\F)+1)$ & $(0,1)$ 
	\strt \\ 
	&  & $(4-3\tau,\dim(\F)-1)$ 
	\strt \\ 
	&  & $(1,n-2\dim(\F)+1)$ 
\end{tabular}
\caption{Eigenvalues and multiplicities of the Jacobi operator of $\s{S}^n_{\F,\tau}$ associated with unit vectors $X$ and $Y$ in $\g{p}_1$ and $\g{p}_2$, respectively.}
\label{table:jacobi}
\end{table}

\section{Totally geodesic submanifolds and homogeneity}
\label{sect:tgandhom}
Let $M=\s G/\s K$ be a Riemannian homogeneous space. Since every $\s{G}$-invariant tensor is parallel under the canonical connection $\nabla^c$, we have $\nabla^c R=0$ and $\nabla^c D=0$, where $R$ and $D$ denote the curvature tensor and the difference tensor of $M=\s G/\s K$, respectively. Thus, by Equations (\ref{eq:nabla}) and (\ref{eq:dtensor}), the covariant derivatives of the curvature tensor can be expressed in terms of $D$ and $R$ (see~Remark~\ref{rem:covderivative}). Notice that homogeneous spaces are real analytic, see~\cite[Lemma~1.1]{Bohm1}. Hence, every subspace $\g{p}_{\Sigma}$ of $\g{p}$  invariant under $D$ and $R$ is invariant under every covariant derivative of $R$, which implies by Theorem~\ref{th:fundamentaltg} that $\g{p}_{\Sigma}$ is the tangent space of some totally geodesic submanifold of $M$ and $\Sigma=\exp_o \g{p}_{\Sigma}$ is a complete totally geodesic submanifold of $M$, where $\exp_o$ denotes the Riemannian exponential map of $M$ at $o$. Hence, we obtain the following result.
\begin{lemma}
\label{lemma:DRinvariant}
Let $M=\s{G}/\s{K}$ be a Riemannian homogeneous space with base point $o\in M$ and reductive decomposition $\g{g}=\g{k}\oplus\g{p}$. Let $\g{p}_{\Sigma}$ be a subspace of $\g{p}$ invariant under $D$ and $R$. Then, there   is a complete totally geodesic submanifold $\Sigma$ of $M$ such that $T_o\Sigma=\g{p}_\Sigma$.
\end{lemma}
We remark that Lemma \ref{lemma:DRinvariant} provides a sufficient condition to obtain totally geodesic submanifolds in homogeneous spaces in terms of a simple linear algebraic property. However, there are many totally geodesic submanifolds whose tangent space is not invariant under the difference tensor $D$.

Now observe that the integral lines of a Killing vector field $X$ of constant length are geodesics. In fact, by making use of the Killing equation, one has that the gradient of the real function $\Vert X\Vert^2$ is $\frac12 \nabla_XX$. Then $\nabla _XX =0$ and every integral line of $X$ is a geodesic. 

Recall that a geodesic orbit  space (g.o.~in short) is a Riemannian manifold $M$ such that every geodesic is the orbit of a one parameter subgroup of $\Isom(M)$ or, equivalently, an integral line of a Killing vector field of $M$. Observe that a totally geodesic submanifold $\Sigma$ of $M$ is also a g.o.\ space. 
This follows from the fact that the projection to $T\Sigma$ of the restriction to $\Sigma$ of a Killing vector field of 
$M$ is a Killing vector field of $\Sigma$.

\begin{lemma}\label{lemma: g.o.tg}
Let $M=\s G/\s K$ be  a compact  g.o.\ space and let $\Sigma$ be a totally geodesic  submanifold of $M$ of dimension at least $2$. Then, either $\Sigma$ is compact or $\Sigma$ has a totally geodesic and flat submanifold $F$ of dimension at least $2$.  
\end{lemma}
\begin{proof}

Let us assume that   $\Sigma$ is not compact. Then, the projection of the restriction to $\Sigma$, of a Killing vector field of $M$ to $T\Sigma$ is a bounded Killing vector field since $M$ is compact. Since there is a projected Killing vector field in any direction tangent to $\Sigma$, we conclude that $\mathfrak b\cdot q= T_q\Sigma$, where $\mathfrak{b}$ denotes the ideal, in the full isometry algebra of $\Isom(\Sigma)$, of bounded intrinsic Killing vector fields of $\Sigma$.
Let $\s B$ be the (transitive) Lie subgroup of isometries of $\Sigma$ associated with $\mathfrak b$. By making use of \cite [Theorem 1.3] {boundedkilling}, one has that 
$\s B=\s H \s A$  (almost direct product)  where $\s H$, if non-trivial, is compact semisimple and 
$\s A$  is contained in the center of 
$\Isom(\Sigma)$ and thus it is abelian. Observe that $\dim \s A\geq 1$ since $\Sigma$ is not compact. We claim that $\s A$ is closed  in  $\Isom(\Sigma)$.  Let  $\bar{\g{a}}$ be the Lie algebra of the closure of $\s{A}$ in $\Isom(\Sigma)$ and $q\in\Sigma$. Then, since $[\g{b},\g{a}]=0$, we have$$||X^*_q||=|| g_*X^*_q||=||(\Ad(g)X)^*_{g(q)}||=||X^*_{g(q)}||,\quad\text{for every $g\in \s B$ and $X\in\bar{\g{a}}$.}$$  Hence, $X^*$ is a Killing vector field of  constant length, and thus, bounded since $\s{B}$ acts transitively on~$\Sigma$. 
Then, every orbit of $\s A$ is a totally geodesic and flat (closed) submanifold of $\Sigma$.  By  \cite[Lemma~5.3]{DO-nullity2}, see also Remark~5.4 therein, $\dim \s A \cdot q = \dim \s A$, for all $q\in M$. Hence, if $\dim \s A \geq 2$, we are done.

Let us assume that $\dim \s A =1$. Since $\s B =  \s H\s A$ is transitive on $\Sigma$ and $\Sigma$ is not compact, the orbits of the normal subgrop  $\s H$ of $\s B$ induces a $\s B$-invariant foliation $\mathfrak F$ of codimension one on $\Sigma$. We claim that  $\mathfrak F$ is a totally geodesic  foliation. Let us prove that $\s H\cdot q$ is a totally geodesic submanifold of $\Sigma$ for each $q\in\Sigma$. Notice that $X\in\g{b}$ induces a Killing vector field of $\Sigma$ that projects to a vector field $\bar{X}$ in  the quotient $\s H\setminus \Sigma$ given by the orbits of the action of $\s H$ on $\Sigma$.  Then, if $X_q\in T_q (\s H\cdot q)$, we have
\[ \pi_{*h(q)} X_{h(q)}=\bar{X}_{\pi(h(q))}=\bar{X}_{\pi(q)}=\pi_{*q}X_q=0  \quad \text{for every $h\in\s H$},  \]
where $\pi\colon \Sigma\rightarrow \s{H}\setminus\Sigma$ denotes the  quotient map.  Then the restriction of $X$ to $\s H\cdot q$ is always tangent to $\s H\cdot q$. Hence, $\Exp(t X\cdot q)\in \s H\cdot q$ for every $t\in \R$.
This implies that $\s H\cdot q$ is a totally geodesic submanifold of $\Sigma$ since $\Sigma$ is g.o.\ with respect to the presentation $\s B/\s B_q$.

Let $\mathcal F$ be the autoparallel distribution associated to $\mathfrak F$.   Observe that 
$\s H$ acts polarly on $\Sigma$, since it is a cohomogeneity one action. Hence, the perpendicular $1$-dimensional distribution $\mathcal F^\perp$ is autoparallel. Then $\mathcal F$ and $\mathcal F^ \perp$ are parallel distributions. In fact, any two orthogonally complementary autoparallel distributions must be parallel. Then $\Sigma$ locally splits and it admits a $2$-dimensional totally geodesic submanifold (e.g.\ a product of geodesics). 
\end{proof}

\begin{remark}\label{remark:section}
Let $M=\s G/\s K$ be a compact Riemannian homogeneous manifold and let $\Sigma$ be a totally geodesic and flat (immersed) submanifold that contains the base point 
$o= e\cdot \s{K}$. Let us consider the isotropy action of $\s{K}$ on $T_oM$.

Then $\s K\cdot w$ is perpendicular to $T_o\Sigma$ at $w$, for all 
$w\in T_o\Sigma$. In fact, let $Y$ be in the Lie algebra $\mathfrak k$ of $\s{K}$ and let us consider the Killing vector field $Y^*$ induced by $Y$. Since $M$ is compact, $Y^*$ is bounded. Then the projection $\bar Y$ to $T\Sigma$ of $Y^*_{\vert \Sigma}$ is a bounded Killing vector field of $\Sigma$ which vanishes at 
$o$. Then $\bar Y =0$ and so $ Y^*_{\vert \Sigma}$ is always perpendicular to $\Sigma$. Then $\nabla _{T_q\Sigma}Y^* \perp T_q\Sigma$ for all $q\in \Sigma$ and in particular for $q=o$. We have that  $ (\mathrm{Exp}(t Y))_{*o}(w) 
= \mathrm{e}^{t(\nabla  Y^*)_o}(w)$ (see formula (2.2.1) in \cite{DO-nullity}). Hence, differentiating at $t=0$, we obtain that 
$Y.w =\nabla _w Y^*$, where $Y.w:=\tfrac{d}{dt}_{\rvert t=0}  (\mathrm{Exp}(t Y))_{*o}(w)$. Consequently, $Y.T_o\Sigma \perp T_o\Sigma$ for all $Y$ in the Lie algebra of~$\s K$.
\end{remark}

Let $M=\s G/\s K$ be a homogeneous Riemannian manifold and let $\Sigma$ be an embedded  (connected)   complete totally geodesic submanifold that contains the base point $o=e\s K$. Let 
$\s{G}_\Sigma:= \{g\in \s{G}:g\Sigma=\Sigma\}$
be the  Lie sugroup of $\s{G}$ that leaves  $\Sigma$
invariant. Observe that 
$\s{G}'_\Sigma:= \{g\in \s{G}:g_{\vert \Sigma}=\mathrm{Id}_{\vert \Sigma}\}$
is a normal subgroup of $\s{G}_\Sigma$ and $\s{G}_\Sigma/\s{G}'_\Sigma$ acts effectively on $\Sigma$. In the case that $\s{G}$ is compact  there exists a normal subgroup $\s{N}$ of $\s{G}_\Sigma$ such that $\s{G}_\Sigma = \s{G}'_\Sigma\cdot \s{N}$ (almost direct product) and $\s{N}$ acts almost effectively on $\Sigma$. 

The totally geodesic submanifold $\Sigma$ is called 
\textit{extrinsically homogeneous with respect to the presentation group $\s G$ of $M$}, if $\s{G}_\Sigma$ acts transitively on $\Sigma$.  Moreover, $\Sigma$ is \textit{extrinsically homogeneous} if it is extrinsically homogeneous with respect to $\s{G}=\Isom(M)$.

We say that two  complete totally geodesic submanifolds $\Sigma_1$ and $\Sigma _2$ passing through $o$ are \textit{isotropy-congruent at $o$} if there exists $k\in \s K$ such that 
$k\Sigma_1 = \Sigma _2$ or, equivalently, if 
$k_{*o}T_o\Sigma_1= T_o\Sigma _2$.

\begin{lemma}\label{lemma:finclas} Let $M=\s{G}/\s{K}$ be a homogeneous Riemannian manifold and $\Sigma$ be a complete totally geodesic submanifold passing through $o$. If there exist finitely many  isotropy-congruence classes at $o$ of totally geodesic submanifolds isometric to $\Sigma$, then every  complete embedded totally geodesic submanifold of $M$ isometric to $\Sigma$ is extrinsically homogeneous with respect to $\s{G}$. 
\end{lemma}
\begin{proof} Let $\Sigma$ be a totally geodesic submanifold of $M$ passing through $o$. Assume that the closed subgroup $\s{G}_\Sigma:= \{g\in \s{G}:g\Sigma=\Sigma\}$ does not act  transitively on $\Sigma$, and let us prove that there are infinitely many isotropy-congruence classes at $o$ of totally geodesic submanifolds isometric to $\Sigma$. Then $\s{G}_\Sigma\cdot p$ is a closed subset of $\Sigma$ for every $p\in \Sigma$. Moreover, the sets $\s{G}_{\Sigma}\cdot p$ do coincide or have empty intersection. By Sard's Theorem, each orbit of $\s{G}_{\Sigma}$ on $\Sigma$ has measure zero in $\Sigma$. Since a countable union of measure zero subsets  of $\Sigma$ has measure zero in $\Sigma$, there is a sequence $\{q_n\}_n\in \mathbb N$ in $\Sigma$ such that $q_i$ do not belong to $\s{G}_\Sigma \cdot q_{j}$ for each $i\neq j$. Now take $g_n\in \s{G}$ in such a way that $g_nq_n =o$, $n\in \mathbb N$, and define $\Sigma_n:= g_n\Sigma$. 
Let $i$ and $j$ be distinct elements in $\mathbb{N}$ and  $k\in \s{K}$ such that $k\Sigma_{i}= \Sigma_{j}$. Then $g':= g_j^{-1}kg_i\in \s{G}_\Sigma$ and $g'q_{i}=q_{j}$. This yields a contradiction showing that $\Sigma _{i}$ is not congruent to $\Sigma_{j}$ for every $i\neq j$, since $q_i$ and $q_j$ do not belong to the same $\s{G}_{\Sigma}$-orbit. Therefore  there are infinitely many congruence classes at $o$ of totally geodesic submanifolds of $M$ isometric to $\Sigma$.
\end{proof}

\begin{remark}
Notice that two totally geodesic submanifolds $\Sigma_1$ and $\Sigma_2$ of $M$ passing through $o\in M$ can be congruent but not isotropy-congruent at $o\in M$. Indeed,  we will see that there are not well-positioned totally geodesic $4$-dimensional  spheres, which are congruent in $\s{S}^{15}_{\O,\tau}$ but not isotropy-congruent at $o\in\s{S}^{15}_{\O,\tau}$, since they are not extrinsically homogeneous, see~Theorem~\ref{th:4inhomo}. 
\end{remark}

\section{Well positioned totally geodesic submanifolds of Hopf-Berger spheres}
\label{sect:wellpos}
The goal of this section is to give a characterization of well-positioned totally geodesic submanifolds in Hopf-Berger spheres.

Motivated by Lemma~\ref{lemma:submersiontg}, we  say that a  totally geodesic submanifold $\Sigma$ of a Hopf-Berger sphere $\s{S}^n_{\F,\tau}$ is \textit{well-positioned at $p\in\Sigma$} if  $T_p \Sigma= (T_p\Sigma\cap \mathcal{H}_p)\oplus(T_p\Sigma\cap \mathcal{V}_p)$, where $\mathcal{H}$ and $\mathcal{V}$ denote the horizontal and vertical distributions of the corresponding Hopf fibration, respectively. Moreover, we say that the totally geodesic submanifold $\Sigma$ is \textit{well-positioned} if the previous property holds for every $p\in\Sigma$.

The following lemma implies that if a totally geodesic submanifold has a vertical or horizontal tangent vector, then it is well-positioned at the corresponding point.
\begin{lemma}
\label{lemma:splitting}
Let $\Sigma$ be a totally geodesic submanifold of $\s{S}^n_{\F,\tau}$, with $\tau\neq 1$. Let $\g{p}_{\Sigma}\subset\g{p}$ be identified with the tangent space of $\Sigma$ at $o\in\s{S}^n_{\F,\tau}$.
Then, if $\g{p}_{\Sigma}\cap\g{p}_i\neq 0$ for some $i\in\{1,2\}$, we have  $\g{p}_{\Sigma}=(\g{p}_{\Sigma}\cap\g{p}_1)\oplus(\g{p}_{\Sigma}\cap\g{p}_2)$ and $\Sigma$ is well-positioned at $o$.
\end{lemma}
\begin{proof}
Let us identify $\g{p}_\Sigma\subset\g{p}$ with the tangent space at $o\in\s{S}^n_{\F,\tau}$ of a totally geodesic submanifold $\Sigma$ in $\s{S}^n_{\F,\tau}$. Also,  if $X\in\g{p}_\Sigma\cap\g{p}_1$ is a non-zero vector, we have  $R_X\g{p}_{\Sigma}\subset\g{p}_{\Sigma}$. Then, the sets of eigenvalues of $R_{X\rvert\g{p}_1}$ and $R_{X\rvert\g{p}_2}$ have non-trivial intersection if and only if $\tau=1$, see~Table~\ref{table:jacobi}.  If $Y\in\g{p}_\Sigma\cap\g{p}_2$ is a non-zero vector, we have $R_Y\g{p}_{\Sigma}\subset\g{p}_{\Sigma}$. Then, the sets of eigenvalues of $R_{Y\rvert\g{p}_1}$ and $R_{Y\rvert\g{p}_2}$ have non-trivial intersection if and only if $\tau=1$, see~Table~\ref{table:jacobi}.
Hence, if there is some non-zero vector $Z\in\g{p}_{\Sigma}\cap\g{p}_i$, for some $i\in\{1,2\}$, we have $\g{p}_{\Sigma}=(\g{p}_{\Sigma}\cap\g{p}_1)\oplus(\g{p}_{\Sigma}\cap\g{p}_2)$ when $\tau\neq1$. 
\end{proof}

\begin{lemma} \label{lemma:wellatp}
Let $\Sigma$ be a complete totally geodesic submanifold of $\s{S}^{n}_{\F,\tau}$, with $\tau\neq 1$.	
If there exists $p\in \Sigma$ such that $\Sigma$ is well-positioned at $p$, then $\Sigma$ is well-positioned. 
\end{lemma}

\begin{proof} For the sake of notational simplicity  and since the following arguments are local, we will assume that $\Sigma$ is embedded. 

Firstly, let us suppose that
$T_p\Sigma\cap \mathcal{V}_p=\{0\}$. Let $q\in \Sigma$ be arbitrary  and let $\gamma$ be a geodesic of $\Sigma$ with $\gamma (0)= p$, $\gamma (1)= q$. Observe that $\gamma$ is a horizontal geodesic since Hopf fibrations are Riemannian submersions and $\gamma'(0)\in\mathcal{H}_p$. Hence, $\gamma '(1)\in \mathcal{H}_q$. Then, by Lemma \ref{lemma:splitting}, 
$T_q \Sigma= (T_q\Sigma\cap \mathcal{H}_q)\oplus(T_q\Sigma\cap \mathcal{V}_q)$, and thus, $\Sigma$ is well-positioned at $q$. 

Now let us assume that 
$T_p\Sigma\cap \mathcal{V}_p\neq \{0\}$. If 
$T_p\Sigma\subset \mathcal{V}_p$, then $\Sigma$ is locally contained around $p$ in the fiber $F(p)$, which is the totally geodesic integral submanifold of $\mathcal{V}$ by $p$. Then   $\Sigma$ is well-positioned around $p$. Thus, we may assume that $T_p\Sigma\cap \mathcal{H}_p\neq \{0\}$. Let $\Sigma'$ be the totally geodesic submanifold obtained by intersecting  $\Sigma$ and $F(p)$ around $p$. 
Let $\nu (\Sigma ')$ be the normal bundle of $\Sigma '$ as a submanifold of $\Sigma$. Observe that $\Sigma$ is well-positioned at $q$ for every element $q\in\Sigma'$. Then  $\nu_q(\Sigma')$ is contained in $\mathcal{H}_q$.  One has that there exists $\delta >0$ such that the  normal exponential map from $\nu_\delta (\Sigma ')$ into $\Sigma$ is a diffeomorhism into its image $U\subset \Sigma$, where $\nu _\delta (\Sigma ')$ are the normal vectors with length less that $\delta$. 
This implies that any 
element $x\in U\setminus\Sigma '$ is of the form $x =\gamma (1)$, with $\gamma$ a geodesic of $\Sigma$ with initial velocity in $\nu_q$ for some $q$. Moreover, $\gamma'(1)\in T_x\Sigma\cap \mathcal{H}_x$. Then, by Lemma \ref{lemma:splitting}, $\Sigma$ is well-positioned at $x$. If $x\in \Sigma '$, then $T_x\Sigma $ has a non-trivial vertical vector and hence, by Lemma \ref{lemma:splitting}, $\Sigma$ is well-positioned at $x$.
We have proved that the points where $\Sigma$ is well-positioned form an open subset of $\Sigma$. 
Moreover, by continuity and Lemma~\ref{lemma:splitting}, we have that the set of points of $\Sigma$ where $\Sigma$ is not well-positioned defines an open subset of $\Sigma$. Since $\Sigma$ is connected, we conclude that $\Sigma$ is well-positioned at every point, and thus, $\Sigma$ is well-positioned. 
\end{proof}

\begin{remark}\label {rem:pull-back} Let $\bar{M}$ be a symmetric space of rank one with minimal absolute value of its sectional curvature equal to one, and  consider  for each $t\in(0,\inj(p))$ the natural map
$h^t\colon \s{S}(T_p \bar{M})\to \s{S}_t(p)$ given by 
$h^t (q):= \overline{\exp} _p(tq) = \gamma _q(t)$, 
where $\s{S}_t(p)$ is endowed with the Riemannian metric induced by $\bar{M}$. We compute the pull-back metric $\langle \cdot,\cdot \rangle^t$  induced by $h^t$ on  $\s{S}(T_p \bar{M})$ at a point $q\in\s{S}(T_p \bar{M})$. 

Let $w\in  T_q\s{S}(T_p \bar{M})$
and let  $c(s)$ be  a curve in $\s{S}(T_p \bar{M})$ with initial conditions $c(0)=q$ and $\dot c(0) = w$.   Then $h^t(c(s)) = 
\gamma _{c(s)}(t)$ is a variation of radial geodesics starting at $p$. Then 
$$h^t_{*} (\dot c(0))= \frac {\partial \,}{\partial s}_{\rvert s=0}\gamma _{c(s)}(t) = J_w(t),$$
where $ J_w(t)$ is the Jacobi field along $\gamma _q(t)$ with initial conditions $J(0)=0$ and 
$J'_w(0)= w$. Hence,
$$\langle w, w\rangle^t = \langle  J_w(t), J_w(t)\rangle.$$
There are two complementary distributions $\mathcal F _1$ and $\mathcal F_2= \mathcal F_1^\perp$ (with respect to $\langle\cdot,\cdot\rangle^1$) on $\s{S}(T_p \bar{M})$ defined by  the eigenspaces associated with the eigenvalues $\lambda _1=\pm 4$ and $\lambda _2 = \pm 1$ of the Jacobi operator 
$\bar{R}_{\cdot,q}q$ of $\bar{M}$ at $p$ (the plus sign is for the compact case and the minus sign for the non-compact case). One has that $(h^t)_*(\mathcal F_i) = \mathcal D^t_i$, for each $i\in\{1,2\}$ (see Section~\ref{subsect:rkonetg} for the definition of $ \mathcal D^t_i$). Let $\widetilde{w}_i$ be the parallel vector field along $\gamma_q$ with initial condition $\widetilde{w}_i(0)=w_i$. If $w_i\in \mathcal F_i(q)$, then  $J_{w_i}(t)=  \frac {1}{\sqrt { \lambda _i}}\sin (\sqrt {\lambda _i}\,  t)\tilde w _i(t)$ if $\bar{M}$ is of compact type, or 
$J_{w_i}(t) = \frac {1}{\sqrt { -\lambda _i}} \sinh  (\sqrt {-\lambda _i}\,  t)\tilde w _i(t)$ if $\bar{M}$ is of non-compact type. 
In addition to that, $\langle \mathcal F_1 , \mathcal F_2  \rangle^t = 0$  and 
\begin{equation*}
	\begin{aligned}
		\langle w_1  , w_1   \rangle^t & = \tfrac{1}{4}\sin^2 (2t)\Vert w_1\Vert^2,&
		\langle w_2,w_2   \rangle^t&= \sin^2 ( t)\Vert w_2\Vert^2, \hspace{0.3cm} \text{when $\bar{M}$ is  compact, or}\\
		\langle w_1  , w_1   \rangle^t &= \tfrac{1}{4}\sinh^2 (2t)\Vert w_1\Vert^2,&
		\langle w_2  , w_2   \rangle^t&= \sinh^2 ( t)\Vert w_2\Vert^2, \text{
			when $\bar{M}$ is non-compact.}
	\end{aligned}
\end{equation*}
Recall that $\s{S}^n_{\F,1}$ is a round sphere of constant sectional curvature equal to $1$. Moreover, $\s{S}^n_{\F,\tau}$ is homothetic to the geodesic sphere $\s{S}_t(p)$ via a homothety of ratio $\alpha=\sin(t)$ or $\alpha=\sinh(t)$, depending on whether $\bar{M}$ is of compact or non-compact type, respectively. Notice that this proves that  the radius $t$ of the geodesic sphere and $\tau$ are related by $t=\arccos(\sqrt{\tau})$ in the compact setting, as we pointed out in Equation~(\ref{eq:radius_tau_homothety}), since $\tau=\tfrac{\sin^2(2t)}{4\sin^2(t)}=\cos^2(t)$. The relation $t=\arccosh(\sqrt{\tau})$ in the non-compact setting is obtained analogously.
\end{remark}

\begin{theorem}
\label{th:wellpos}
Let $\Sigma$ be a totally geodesic  submanifold of $\s{S}^n_{\F,\tau}$, $\tau\neq 1$, and let $\bar{M}$ be the corresponding rank one symmetric space such that  $\s{S}^n_{\F,\tau}$ arises as a geodesic sphere of $\bar M$ centered at $p\in \bar{M}$. Then, the following statements are equivalent:
\begin{enumerate}[i)]
	\item $\Sigma$ is a  well-positioned totally geodesic submanifold of $\s{S}^n_{\F,\tau}$.
	\item $T_q\Sigma$ has a vertical or a horizontal non-zero vector for every $q\in \Sigma$.
	\item $T_q\Sigma$ is invariant under the shape operator of $\s{S}^n_{\F,\tau}$ in $\bar{M}$ for every $q\in \Sigma$.
	\item $\Sigma$ is the intersection of a totally geodesic submanifold $M$ of $\bar{M}$ containing $p\in\bar M$ with $\s{S}^n_{\F,\tau}$ regarded as a geodesic sphere of $\bar{M}$.
\end{enumerate}
\end{theorem}
\begin{proof}
First of all, let us identify $\s{S}^n_{\F,\tau}$ with the geodesic sphere $\s{S}_t(p)$ of the appropiate symmetric space $\bar{M}$ of rank one centered at $p\in\bar{M}$ and  $t=\arccos(\sqrt{\tau})$, in the compact case; or $t=\arccosh(\sqrt{\tau})$,  in the non-compact case, see Equation~(\ref{eq:radius_tau_homothety}).

Now, observe that \textit{i)} and~\textit{ii)} are equivalent by Lemma~\ref{lemma:splitting}. Moreover, let us recall that \textit{i)} and~\textit{iii)} are equivalent since the only eigenvalue of the shape operator $\mathcal{S}^t$ of $\s{S}_t(p)$ restricted to the horizontal distribution is different to the only eigenvalue of the shape operator  restricted to the vertical distribution, for every $t\in(0,\inj(p))$, see~Equations~(\ref{eq:b1}), (\ref{eq:b2}) and the paragraph below them.

Now, by Corollary~\ref{cor:geodsphereinter} and Remark~\ref{remark:CC1}, the intersection of a totally geodesic submanifold of $\bar M$ that contains $p$ with  $\s{S}_t(p)$ is totally geodesic in $\s{S}_t(p)$ and its tangent space is invariant under $\mathcal{S}^t$. Hence, \textit{iv)} implies \textit{iii)}.

Let us prove that \textit{iii)}  implies \textit{iv)}.
Let $0<s , t <\inj(p)$ and consider the map $f_s: \s{S}_t(p)\to \s{S}_{s}(p)$ given by 
$f_s(\overline{\exp} _p (tv) ) = \overline{\exp} _p (sv)$,  for all $v\in T_pM$ of unit length. Then the pull-back by $f_s$ of the Riemannian metric of 
$\s{S}_{s}(p)$, if $M$ is of  compact type, is given by modifying the metric on $\s{S}_t(p)$ by the factor 
$\frac{\sin^2 (2s)}{\sin^2 (2t)}$ on the distribution $\mathcal D_1^t$ and by the factor 
$\tfrac {\sin^2 (  s)}{\sin^2 (t)}$ on the distribution $\mathcal D_2^t$ (in the non-compact case the trigonometric functions have to be replaced by the corresponding hyperbolic functions).  By Remark~\ref{rem:pull-back}, if we   rescale this pull-back metric by a factor of $\frac{\sin^2 (t) }{\sin^2 (s)}$, we are under the assumptions of  Lemma~\ref {lemma:submersiontg}.  This shows that $\Sigma$ is also totally geodesic with respect to the pull-back metric for every $s\in(0,\inj(p))$. Or equivalently, $f_s(\Sigma)$ is a totally geodesic submanifold of $\s{S}_{s}(p)$. Furthermore, using that $T_q\Sigma$ is invariant under the shape operator for every $q$ and that the integral curves of the outer unit normal vector field to the geodesic spheres are geodesics,  it is standard to show that $\widehat{\Sigma}:=\bigcup _{0<s < \inj(p)}f_s(\Sigma)$ is a totally geodesic submanifold of the open ball $B_{\inj(p)} (p)$ which contains a piece of the radial geodesic $\gamma_v(t)
=\overline{\exp}_p(tv)$.  Since a totally geodesic submanifold of a symmetric space extends to a complete totally geodesic submanifold, the complete extension $\widetilde {\Sigma}$ of $\widehat{\Sigma}$ contains $p$. By making use of Remark~\ref{remark:CC1}, we have $\Sigma=\widetilde{\Sigma}\cap \s{S}_t(p)$, and  it follows that \textit{iii)}  implies~\textit{iv)}.
\end{proof}
Notice that if $\dim(\Sigma)\ge\dim\F$, then $T_p\Sigma$ has a  horizontal non-zero vector for every $p\in\Sigma$. Thus, by combining Lemma~\ref{lemma:splitting} and Theorem~\ref{th:wellpos} we have the following:
\begin{corollary}
\label{cor:tghighdimension}
Let $\Sigma$ be a totally geodesic submanifold of $\s{S}^n_{\F,\tau}$. If $\dim\Sigma\ge\dim\F$, then $\Sigma$ is well-positioned.
\end{corollary}
\begin{remark}
\label{rem:intersect}
Firstly, observe that the intersection between a geodesic sphere and a totally geodesic submanifold passing through the origin is transverse. Hence, the well-positioned totally geodesic submanifolds of Hopf-Berger spheres are embedded.

Secondly, notice that the intersection of a complete totally geodesic submanifold $M$ of $\bar{M}$ passing through $p$ with a geodesic sphere $\s{S}_t(p)$ of $\bar{M}$ is a geodesic sphere of ${M}$ of radius $t>0$. Indeed, let us denote by $\s{S}'_t(p)$ the geodesic sphere of $M$ of radius $t>0$ centered at $p$. Let us denote by $d_M$ and $d_{\bar{M}}$ the distances on $M$ and $\bar M$ induced by the respective Riemannian metrics. Then,
\[\s{S}_t(p)\cap M=\{q\in M: d_{\bar{M}}(p,q)=t\}=\{q\in M: d_{M}(p,q)=t\}=\s{S}'_t(p),\]
where we have used that $M$ is a complete totally geodesic submanifold of $\bar M$.
\end{remark}
\begin{remark}
\label{rem:congwellpos}
Notice that  the isotropy of $\bar M$ at $p$ is  equal to $\Isom(\s S^n_{\F,\tau})$, see  Remark~\ref{rem:isom}. Moreover, for each $i\in\{1,2\}$, let $\s{S}_i(t)$ be a geodesic sphere of a rank one symmetric space $M_i$ of radius $t\in(0,\inj(M_i))$. Then,  $\s{S}_1(t)$ and $\s{S}_2(t)$ are isometric if and only if $M_1$ and $M_2$ are isometric.   Thus, by Remark~\ref{rem:intersect}, two well-positioned totally geodesic submanifolds $\Sigma_1$ and $\Sigma_2$ of $\s{S}_t(p)$ passing through $o\in\s{S}_t(p)$   are isotropy-congruent at $o\in\s{S}_t(p)$ if and only if they are the intersection of $\s{S}_t(p)$ with  totally geodesic submanifolds $M_1$ and $M_2$ of $\bar M$ which are isotropy-congruent at $p$. Moreover, observe that $M_1$ and $M_2$ are isotropy-congruent at $p$ if and only if they are congruent in $\bar M$ and contain $p\in\bar M$. Hence,   by using the classification of totally geodesic submanifolds in symmetric spaces of rank one~(see Figure~1) we obtain the classification in Table~\ref{table:hopfbergertgwellpos}.  Consequently, two  isometric well-positioned totally geodesic submanifolds of $\s{S}^n_{\F,\tau}$ passing through $o\in\s{S}^n_{\F,\tau}$ must be isotropy-congruent at $o$. This implies, by Lemma~\ref{lemma:finclas}, that every well-positioned totally geodesic submanifold of $\s{S}^n_{\F,\tau}$ is extrinsically homogeneous.
\end{remark}

\subsection{Examples of well-positioned totally geodesic submanifolds in Hopf-Berger spheres}
\label{subsect:examplestg}
A good source of examples of well-positioned totally geodesic submanifolds of Hopf-Berger spheres is provided by fixed-point components of isometries as it shows the following lemma.
\begin{lemma}
\label{lemma:fixedpoint}
Let $\Omega$ be a subset of $\Isom(\s{S}^n_{\F,\tau})$. Then, the set of fixed points of $\Omega$ is a well-positioned totally geodesic submanifold of $\s{S}^n_{\F,\tau}$.
\end{lemma}	
\begin{proof}
Let us regard the Hopf-Berger spheres as geodesic spheres $\s{S}_t(p)$ of an appropiate rank one symmetric space $\bar M$. Recall that every isometry of $\s{S}_t(p)$ can be identified with a certain isometry of $\bar M$ lying in the isotropy at $p\in\bar M$, see~Remark~\ref{rem:isom}.
Let $\Omega$ be a subset of isometries of $\bar M$ that fix $p$. Consider the totally geodesic submanifold $\bar \Sigma$ of $\bar M$ given by the connected component by $p$ of the fixed set of $\Omega$.  Let $\Sigma \subset \s{S}_t(p)$ be  the  set of fixed points of $\Omega_{\vert \s{S}_t(p)}$, where $\Omega_{\vert \s{S}_t(p)}$ denotes the set of elements in $\Omega$ restricted to $\s{S}_t(p)$.  	
Then $\Sigma$ is a totally geodesic submanifold of $ \s{S}_t(p)$ and $\Sigma = \bar{\Sigma} \cap \s{S}_t(p)$.
Thus, Theorem 
\ref{th:wellpos} implies that $\Sigma$ is a well positioned totally geodesic submanifold of $\s{S}_t(p)$.

Also notice that by  \ref{remark:CC1},  $\Sigma$ is connected. In fact, $\Sigma$ can be regarded, using normal coordinates at $p$, as the intersection of 
$T_p\bar \Sigma$ with the Euclidean sphere centered at $0$ and radius~$t$. \qedhere

\end{proof}

It will be convenient for the subsequent discussions to posses explicit descriptions of the tangent space of some well-positioned totally geodesic submanifolds of $\s S^n_{\F,\tau}$. Consider the following subspaces $(1)_{\F}$ and $(2)_{\F}$ of $\g{p}$, where $\F\in\{\C,\H,\O\}$.

\begin{enumerate}
\item[$(1)_{\C}$] $\g{p}_{\Sigma}=V$, where $V$ is a totally real subspace of $\g{p}_2$.
\item[$(2)_{\C}$]  $\g{p}_{\Sigma}=\g{p}_1\oplus V$, where $V\subset\g{p}_2$ is totally complex. 
\item[$(1)_{\H}$] $\g{p}_{\Sigma}=\R X \oplus V$, where $X\in\g{p}_1$ and $V$ is a totally complex subspace of $\g{p}_2$ with respect to $J_X$, see   in Subsection~\ref{subsect:geomhopf} for the definition of the map $J$, and also   \cite[\S2.4 \textup{(2)}]{DDR}.
\item[$(2)_{\H}$] $\g{p}_{\Sigma}=\g{p}_1 \oplus V$, where $V$ is invariant under $J_X$ for every $X\in\g{p}_1$.
\item[$(1)_{\O}$] $\g{p}_{\Sigma}=\spann\{X_i\}_{i\in I} \oplus \H_I Y_1$, where $I$ is a line of the Fano plane (see Figure~\ref{fig:fanoplane}).
\item[$(2)_{\O}$]
$\g{p}_{\Sigma}=\g{p}_1$.
\end{enumerate}
It is clear from Equations~(\ref{eq:dtensorcomplex}-\ref{eq:rtensoroctonions0}) that these subspaces are invariant under the difference and curvature tensor. Hence, by Lemma~\ref{lemma:DRinvariant}, they induce totally geodesic submanifolds in~$\s{S}_{\mathbb{F},\tau}^{n}$. Additionally, it can be checked that these subspaces are invariant under $U$, see Subsection~\ref{subsect:homprelim} for the definition of $U$.

On the one hand, if we consider Example~$(1)_{\F}$, for each $\F\in\{\C,\H,\O\}$, we have that $\g{k}_\Sigma:=[\g{p}_{\Sigma},\g{p}_{\Sigma}]_\g{k}$ is isomorphic to $\g{so}_k$, $\g{u}_k$, or $\g{sp}_1\oplus\g{sp}_1$, where $k$ is the real or complex dimension of~$V$, respectively. The Lie algebra generated by $\g{p}_{\Sigma}$ is $\g{g}_{\Sigma}=\g{k}_\Sigma\oplus\g{p}_{\Sigma}$, which is isomorphic to $\g{so}_{k+1}$, $\g{u}_{k+1}$ or $\g{sp}_2\oplus\g{sp}_1$, respectively.  On the other hand, if we consider Example~$(2)_{\F}$, we have that $\g{k}_\Sigma$ is isomorphic to $\g{u}_k$, $\g{sp}_k$, or  $\g{so}_7$, where $k$ is the complex or quaternionic dimension of $V$, respectively, and the Lie algebra generated by $\g{p}_{\Sigma}$ is $\g{g}_{\Sigma}=\g{k}_\Sigma\oplus\g{p}_{\Sigma}$, which is isomorphic to $\g{u}_{k+1}$, $\g{sp}_{k+1}$, or $\g{so}_{8}$. Let $\s{G}_{\Sigma}$ be the connected subgroup of $\s{G}$ with Lie algebra $\g{g}_{\Sigma}$.
Observe that $\g{p}_{\Sigma}$ is identified with $T_o(\s G_\Sigma\cdot o)$, and the connected component of the isotropy of $\s G_\Sigma$ at $o\in \s S^n_{\F,\tau}$ is $\s K_\Sigma$, which is the connected Lie subgroup of $\s G$ with Lie algebra $\g{k}_\Sigma$. Now,  using Equations~(\ref{eq:nabla}) and (\ref{eq:dtensor}), and the fact that $\g{p}_\Sigma$ is invariant under $D$ and $U$, we have that the second fundamental form of $\Sigma=\s G_\Sigma\cdot o$ vanishes at $o$. Thus, since $\Sigma$ is an orbit, $\Sigma$ is a totally geodesic submanifold of $\s S^n_{\F,\tau}$, and  $\Sigma=\exp_o(\g{p}_{\Sigma})$ by the uniqueness of complete totally geodesic submanifolds.  Hence,  these totally geodesic submanifolds are extrinsically homogeneous submanifolds.

Moreover, by making use of the Gauss equation, we can compute the curvature of these examples using Equations~(\ref{eq:dtensorcomplex}), (\ref{eq:rtensorcomplexquaternionic}), (\ref{eq:rtensorcomplexquaternionic0}) and  (\ref{eq:dtensoroctonions}), (\ref{eq:rtensoroctonions}), (\ref{eq:rtensoroctonions0}). It turns out that Example~$(1)_{\F}$ is isometric, for each $\F\in\{\C,\H,\O\}$, to $\s{S}^k_1$, $\s{S}^{2k+1}_{\C,\tau}$, or $\s{S}^7_{\H,\tau}$, respectively. Furthermore, Example~$(2)_{\F}$ is isometric, for each $\F\in\{\C,\H,\O\}$, to $\s{S}_{\mathbb{C},\tau}^{2k+1}$, $\s{S}_{\mathbb{H},\tau}^{4k+3}$, or $\s{S}^7_{1/\tau}$, respectively.

\begin{remark}
\label{rem:tginclusions}
Notice that for $\tau\neq 1$ and $n\ge1$, we have the   totally geodesic inclusions
\begin{equation*}
	\s{S}^{n}_{1}\subset\phantom{.}\s{S}^{2n+1}_{\C,\tau}\phantom{}\subset\s{S}^{4n+3}_{\H,\tau},\qquad
	\s{S}^3_{\C,\tau}\subset\phantom{.}\s{S}^7_{\H,\tau}\phantom{tt}\subset\s{S}^{15}_{\O,\tau},
\end{equation*}
for the totally geodesic submanifolds constructed in Subsection~\ref{subsect:examplestg}.
\end{remark}

\begin{lemma}
\label{lemma:2octo}
Let $V$ be a curvature invariant $2$-plane in $\s{S}^{15}_{\O,\tau}$ with $\tau\neq1$. Then, $V$ is contained in the tangent space of a well-positioned totally geodesic $\s{S}^7_{\H,\tau}$ in $\s{S}^{15}_{\O,\tau}$.
\end{lemma}
\begin{proof}
First of all, using the isotropy of $ \s{S}^{15}_{\O,\tau}$ (see~\cite{verdiani-ziller}), we can assume that the $2$-plane $V$ is spanned by $u,v\in\g{p}$, where
\begin{align*}
	u&=a_1 X_1/(2\sqrt{\tau})+a_2 X_2/(2\sqrt{\tau})+a_3 X_3/(2\sqrt{\tau})+ a_4 Y_1,\\
	v&=b_1 X_1/(2\sqrt{\tau})+b_2 X_2/(2\sqrt{\tau})+b_3 X_3/(2\sqrt{\tau})+ b_4 X_4/(2\sqrt{\tau})+ b_5 Y_1 +b_6 J_1Y_1,
\end{align*}
where we are using the notation of Subsection~\ref{subsect:geomhopf}.
Let us proceed by contradiction.
Then, we can assume that $b_4\neq 0$, since otherwise, $V$ is contained in the tangent space of some well-positioned totally geodesic submanifold $\s{S}^7_{\H,\tau}$ of $\s{S}^{15}_{\O,\tau}$, see~example $(1)_{\O}$ above. Now,
\begin{equation}
	\label{eq:Rinvarocto}
	\begin{aligned}
		0&=\frac{1}{2\sqrt{\tau}}\langle R(u,v)v,X_6\rangle=-3a_4 b_4 b_6(-1+\tau),\\
		0&=\langle R(u,v)v,Y_6\rangle=3a_1 b_4 b_6(-1+\tau),
	\end{aligned}
\end{equation}
since $V$ is curvature invariant. Thus we either have that $b_6=0$ or $a_1=a_4=0$.

Let us assume that $b_6=0$. Thus,
\begin{equation*}
	\begin{aligned}
		0&=\langle R(u,v)v,Y_8\rangle=3a_3 b_4 b_5(-1+\tau),\\
		0&=\langle R(u,v)v,Y_4\rangle=3a_2 b_4 b_5(-1+\tau),\\
		0&=\langle R(u,v)v,Y_2\rangle=-3a_1 b_4 b_5(-1+\tau),
	\end{aligned}
\end{equation*}
since $V$ is curvature invariant. Hence, $a_1=a_2=a_3=0$ or $b_5=0$. However, in the first case, $V$ is contained in the tangent space of some  well-positioned totally geodesic submanifold $\s{S}^7_{\H,\tau}$ of $\s{S}^{15}_{\O,\tau}$. Thus, $b_5=0$. Observe that we can assume that $a_4\neq 0$, since otherwise $V$ is contained in the tangent space of some totally geodesic submanifold $\s{S}^7_{\H,\tau}$ of $\s{S}^{15}_{\O,\tau}$. Then,
\begin{equation*}
	\begin{aligned}
		0&=\langle R(v,u)u,Y_8\rangle=-3a_3 a_4 b_4(-1+\tau),\\
		0&=\langle R(v,u)u,Y_4\rangle=-3a_2 a_4 b_4(-1+\tau),\\
		0&=\langle R(v,u)u,Y_2\rangle=3a_1 a_4 b_4(-1+\tau),
	\end{aligned}
\end{equation*}
since $V$ is curvature invariant. Hence, $a_1=a_2=a_3=0$ and $V$  is contained in the tangent space of some well-positioned totally geodesic submanifold $\s{S}^7_{\H,\tau}$ of $\s{S}^{15}_{\O,\tau}$.

Let us assume that $b_6\neq0$. Then, by Equation (\ref{eq:Rinvarocto}), $a_1=a_4=0$. Moreover,
\begin{equation*}
	\begin{aligned}
		0&=\langle R(u,v)v,Y_8\rangle=3 b_4(a_3b_5+a_2b_6)(-1+\tau),\\
		0&=\langle R(u,v)v,Y_4\rangle =3b_4(a_2b_5-a_3b_6)(-1+\tau),
	\end{aligned}
\end{equation*}
since $V$ is curvature invariant. Hence, $a_2=a_3=0$, implying that $V$  is contained in the tangent space of some  well-positioned totally geodesic submanifold $\s{S}^7_{\H,\tau}$ of $\s{S}^{15}_{\O,\tau}$.\qedhere
\end{proof}

\section{Not well-positioned totally geodesic submanifolds}
The purpose of this section is to study not well-positioned totally geodesic submanifolds. We start by studying the $1$-dimensional case, i.e.\ geodesics. Subsequently, we prove Theorem~\ref{th:notwellpos}, which gives a characterization of not well-positioned totally geodesic submanifolds in Hopf-Berger spheres. As a consequence of these and other previous results we are able to give the proof of Theorem~\ref{th:a}. Finally, we prove Theorem~\ref{th:notwellsphere}, which provides a description of  not well-positioned maximal totally geodesic spheres in Hopf-Berger spheres. 
\label{sect:notwellpos}
\subsection{Geodesics of Hopf-Berger spheres}
In what follows we will focus on $1$-dimensional totally geodesic submanifolds of Hopf-Berger spheres, i.e.\ geodesics.
\label{subsect:geodesics}
\begin{lemma}\label{lemma:noflat} Every flat totally geodesic submanifold of a Hopf-Berger sphere
$ \s{S}^n_{\F,\tau}$ is a geodesic.
\end{lemma}

\begin{proof}
Let us recall that $\s K$ acts polarly on $T_o\s{S}^n_{\F,\tau}$ (see Remark~\ref{rem:isotropyrep}). Moreover, the principal $\s K$-orbits are products of spheres and they have codimension two.
Let $\Sigma '$ be a flat totally geodesic submanifold of $\s{S}^n_{\F,\tau}$ of dimension greater than or equal two and let $\Sigma$ be a totally geodesic (and hence flat) surface of $\Sigma '$. We may assume that $o\in \Sigma$. Then, by Remark \ref{remark:section},  
$T_o\Sigma$ is a section for the polar representation of $\s K$ in $T_o\s{S}^n_{\F,\tau}$,  since $\s{K}$ acts with cohomogeneity two. Hence, $T_o\Sigma$
is spanned by a horizontal and a vertical vector, and thus $\Sigma$ is well-positioned at $o$. By Lemma~\ref{lemma:wellatp},  $\Sigma$ is well-positioned. Then, by Theorem \ref{th:wellpos} and Remark~\ref{rem:intersect}, $\Sigma$ is a geodesic sphere of a totally geodesic submanifold of dimension three  of the rank one symmetric space containing $\s{S}^n_{\F,\tau}$ as a geodesic sphere. Thus, $\Sigma$ is a round sphere of dimension two, and therefore we obtain a contradiction with the fact that $\Sigma$ is flat.
\end{proof}

\begin{lemma}\label{lemma:geodesicin3dim}
Let $\tau\neq 1$ and $\gamma$ be a geodesic in $\s{S}^n_{\F,\tau}=\s{G}/\s{K}$. Then, there exists a totally geodesic submanifold $\Sigma$ of $\s{S}^n_{\F,\tau}$ satisfying that:

\begin{enumerate}[i)]
	\item $\gamma$ is contained in $\Sigma$.
	\item $\Sigma$ is well-positioned.
	\item	$\Sigma$ is isometric to $\s{S}^3_{\C,\tau}$.
	\item  $\Sigma$ is an orbit of a subgroup of $\Isom(\s{S}^{n}_{\F,\tau})$ locally isomorphic to $\s{U}_2$.
\end{enumerate}

\end{lemma}
\begin{proof}	
Let us assume that $\gamma$ is a (non-constant) geodesic of $\s{S}^n_{\F,\tau}$ with initial velocity $v_1+v_2\in T_o\s{S}^n_{\F,\tau}$, where $v_1\in \g{p}_1$ and $v_2\in\g{p}_2$.  If $v_1$ or $v_2$ is equal to zero, let us take any (not well-positioned) geodesic $\widetilde{\gamma}$ whose velocity has horizontal or vertical projection equal to $v_2$ or $v_1$, respectively. Hence, if  $\widetilde{\gamma}$ is contained in a well-positioned totally geodesic submanifold $\Sigma$, then $\gamma$ is also contained in $\Sigma$. Hence, we may assume that $v_1$ and $v_2$ are both not zero. Then, the initial velocity of $\gamma$ is contained in the subspace spanned by $\{v_1, v_2, [v_1,v_2]\}$. This subspace is clearly the tangent space of a well-positioned totally geodesic submanifold $\Sigma$ isometric to $\s{S}^3_{\C,\tau}$, see~Subsection~\ref{subsect:examplestg} $(2)_{\C}$ or $(2)_{\H}$, and Remark~\ref{rem:tginclusions}. Moreover, by the  discussion in Subsection~\ref{subsect:examplestg}, we have that $\Sigma$ is an orbit of a subgroup of $\Isom(\s{S}^n_{\F,\tau})$ locally isomorphic to $\s{U}_2$.\qedhere

\end{proof}
We define the \textit{slope} of a geodesic $\gamma$ in $\s{S}^n_{\F,\tau}$ to be the quotient between the lengths of the vertical and horizontal projections of the velocity of $\gamma$. This quantity is well defined for every geodesic since Hopf-Berger spheres are geodesic orbit spaces, see~\cite{tamarugo}. Moreover, using the full isotropy representation of $\s{S}^n_{\F,\tau}$ described in Remark~\ref{rem:isotropyrep}, it is easy to see that two geodesics are congruent in $\s{S}^n_{\F,\tau}$ if and only if they have the same slope.

\begin{lemma}
\label{lemma:slopegeod}
Let $\s{S}^n_{\F,\tau}$, with $\tau\neq 1$, and let $\gamma$ be a closed geodesic in $\s{S}^n_{\F,\tau}$. Then, the set of possible slopes for $\gamma$ is countable. 
\end{lemma}
\begin{proof}
First of all, by Lemma~\ref{lemma:geodesicin3dim}, we can assume that $\gamma$ is a closed geodesic of $\s{S}^3_{\C,\tau}$ with $\tau\neq1$.  Let us define a naturally reductive decomposition for $\s{S}^3_{\C,\tau}=\s{U}_2/\s{U}_1$. Keeping the notation in Subsection~\ref{subsect:geomhopf} applied to $\s S^3_{\F,\tau}$, we define  $\g{p}^\tau_1$ as the one-dimensional subspace of $\g{g}=\g{u}_2$ spanned by the unit vector
\[X^\tau:=\frac{1}{\sqrt{\tau}}\left(X_1+(1-2\tau)Z\right), \quad \text{where} \hspace{0.2cm} Z=\left(
\begin{array}{c|c}
	i & 0 \\
	\hline
	0 & 0
\end{array}
\right)\in \g{u}_1\subset\g{u}_2. \]  Hence, it can be checked that $\g{g}=\g{k}\oplus\g{p}^\tau$,  where $\g{p}^\tau:=\g{p}^\tau_1\oplus\g{p}_2$, is a naturally reductive decomposition of $\s{S}^3_{\C,\tau}$, for each $\tau\in(0,\infty)$. In particular, every geodesic of $\s{S}^3_{\C,\tau}$ is of the form $\Exp(s X)\cdot o$, for some $X\in\g{p}^\tau$. For our purposes, we can assume that $o=(0,1)\in \s{S}^3_{\C,\tau}\subset \C^2$. Let $X:=\alpha_1 Y_1 + \alpha_2 J_1 Y_1 +\alpha_3 X^\tau$, where $\alpha_i\in\R$ for every $i\in\{1,2,3\}$ satisfy $\sum_{i=1}^3\alpha^2_i=1$, and consider the geodesic $\gamma_X(s)=\Exp(s X)\cdot o$. We define the quantities
\[P:=\sqrt{\alpha_1^2+\alpha_2^2+\tau\alpha_3^2},\qquad
Q:=\frac{e^{-\frac{i s}{\sqrt{\tau}}\left( P \sqrt{\tau} + (\tau-1)\alpha_3\right)}}{2 P}.\]

A direct computation shows that
\[\gamma_X(s)=Q((-1+e^{2 i P s})(-i \alpha_1 + \alpha_2) ,(1+e^{2 i P s})P+(-1+e^{2 i P s})\sqrt{\tau}\alpha_3).  \]

If $X$ is vertical, $\alpha_3=\pm1$, $\alpha_1=\alpha_2=0$, and  we have  $\gamma_X(s)=(0,e^{\pm\frac{is}{\sqrt{\tau}}})$.  If $X$ is horizontal, $\alpha_3=0$, $\alpha_1^2+\alpha_2^2=1$ and we have
\[\gamma_X(s)=\frac{e^{{-i}{s}}}{2}((-1+e^{2is})(\alpha_2-i \alpha_1),1+e^{2is}).\] 

Now, let us assume that $X$ has non-trivial projection onto $\g{p}^\tau_1$ and $\g{p}_2$. Then, if  $\gamma_X$ is closed, 
\[-1+e^{2 i P s}=0,\quad \textrm{and} \quad 2PQ=1, \quad \text{for some $s\in\R$, $s\neq 0$}.  \]

Hence,
$\alpha_3=\frac{P \sqrt{\tau} (j+2 k )}{ j(1-\tau)}=\frac{\sqrt{(1-\alpha_3^2)+\tau\alpha_3^2} \sqrt{\tau} (j+2k )}{ j(1-\tau)}$, where $j\in\mathbb{Z}\setminus\{0\}$ and $k\in\mathbb{Z}$. Since the slope of $\gamma_X$ is given by $\sqrt{\tfrac{\alpha^2_3}{1-\alpha^2_3}}$, we deduce that the set of possible slopes for $\gamma_X$ is countable.\qedhere
\end{proof}

\subsection {A characterization of not well-positioned totally geodesic submanifolds}
\label{subsect:charactnotwellpos}
Let us start by proving that all the geodesics of a not well-positioned totally geodesic submanifold are closed, have the same slope, and thus, the same length.

\begin{lemma}\label{lemma:notwellpos(-1)}
Let $f\colon\Sigma\rightarrow \s{S}^n_{\F,\tau}$ be a not well-positioned  totally geodesic  immersion of a complete Riemannian manifold $\Sigma$ with dimension $d\ge 2$. Then, the following statements hold:
\begin{enumerate}[i)]
	
	\item $\Sigma$ is compact.
	
	\item Every geodesic of $\Sigma$ is closed.
	
	\item The slope of $ f_{*}(v)$ does not depend on $v\in T\Sigma\setminus\{0\}$.
\end{enumerate}
\end{lemma}
\begin{proof}
Firstly,  \textit{i)} is a consequence of Lemma \ref{lemma: g.o.tg} and Lemma \ref{lemma:noflat}.  

Let $\gamma_v$ be the geodesic of $\Sigma$ with initial conditions $q\in\Sigma$ and $v\in T_q\Sigma$. Also, consider the geodesic $\tilde \gamma = f\circ \gamma_v$ of $\s{S}^n_{\F,\tau}$.  By Proposition~\ref{lemma:geodesicin3dim}, there exists an embeded and compact totally geodesic submanifold  $N$ of $\s{S}^n_{\F,\tau}$, which is isometric to $\s{S}^3_{\C,\tau}$,  such that $\tilde {\gamma}(\mathbb R)\subset N$. Since $\Sigma$ and   $N$ are compact it is not difficult to show that 
$f(F) = \tilde F$, where $F$ and $\tilde F$ are the closure of $\gamma _v(\mathbb R)$ and 
$\tilde {\gamma}(\mathbb R)$, respectively. Observe that 
$\tilde F\subset N$. 

Since $\Sigma$ and $N$ are compact g.o.\ spaces, then 
$F$ and $\tilde F$ are the orbits of abelian compact Lie subgroups $\s{A}$, $\tilde{ \s{A}}$ of the full isometry groups of 
$\Sigma$ and $N$, respectively. In particular, 
$F$ and $\tilde F$ are compact submanifolds of $\Sigma$ and $N$, respectively. Moreover, since $f$ is an immersion, $f\colon F\to \tilde F$ is a covering map. Asumme that $\gamma _v$ is not closed. Then $\dim (\tilde F)= \dim (F)\geq 2$. Thus, the subspace $f_{*q} (T_qF) = T_{f(q)}\tilde F \subset T_{f(q)}N$ intersects non-trivially the horizontal $2$-dimensional subspace of $N$, since $\dim (N) = 3$. Consequently, by  Lemma~\ref{lemma:wellatp} and  Lemma~\ref{lemma:splitting}, $f(\Sigma)$ is well positioned, yielding a contradiction. This proves \textit{ii)}.

Observe that $f\circ \gamma _v$ is a closed geodesic for all $v\in T\Sigma$. Consequently, by  Lemma~ \ref{lemma:slopegeod}, we can prove using a continuity argument that the slope of $f_{*q}(v)$ does not depend on $v\neq 0$.
\end{proof}
\begin{lemma}
\label{lemma:notwellposrk1}
Let $\Sigma$ be a not well-positioned totally geodesic submanifold of $\s{S}^n_{\F,\tau}=\s{G}/\s{K}$ with dimension $d\ge 2$. Then, $\Sigma$ is isometric to a compact rank one symmetric space.
\end{lemma}
\begin{proof}
Let us assume that $o\in \Sigma$ and let $\s{S}(T_o\Sigma)$ be the unit sphere of $T_o\Sigma$. Every vector in $T_o\Sigma$ has the same slope by Lemma~\ref{lemma:notwellpos(-1)}. Let $v\in \s{S}(T_o\Sigma)$, then $\s{K}\cdot v$ contains $\s{S}(T_o\Sigma)$, see Remark~\ref{rem:isotropyrep}. Let $w\in T_v(\s{S}(T_o\Sigma))$. Then, there exists an element $X\in\g{k}\subset \g{so}(T_o\s{S}^n_{\F,\tau} )$ such that $Xv=w$.  Let ${\g{k}}_{\Sigma}$ be the Lie algebra of the full isotropy group of $\Sigma$ at $o$, and $X_{\Sigma}$ be the projection of the restriction of $X$ to $T_o\Sigma$. Then $X_{\Sigma}\in\g{k}_{\Sigma}$. Moreover, $X_{\Sigma} v= w$. This implies that the isotropy of $\Sigma$ acts transitively on the unit sphere of its tangent space at $o$. Therefore, since $\Sigma$ is a compact homogeneous space, $\Sigma$ is a compact $2$-point homogeneous space. Consequently, $\Sigma$ is isometric to a compact rank one symmetric space, see~\cite{Szabo}.\qedhere

\end{proof}

\begin{remark}
Notice that if $\Sigma$ is embedded, the proof of Lemma~\ref{lemma:notwellposrk1} follows by a celebrated result in the theory of manifolds with all the geodesics closed. This result asserts that a compact homogeneous space with all its geodesics closed and of the same length is isometric to a compact rank one symmetric space, see \cite[Theorem 7.55, p.~196]{Bessegeod}. Indeed, by Lemma~\ref{lemma:notwellpos(-1)}, $\Sigma$ is a compact manifold with all its geodesics closed and of the same length, since they all have the same slope. Moreover, since $\Sigma$ is a  totally geodesic submanifold of a homogeneous space is intrinsically homogeneous. Then, $\Sigma$ is isometric to a compact rank one symmetric space.
\end{remark}

Recall that every Hopf-Berger sphere can be regarded, up to a homothety, as a geodesic sphere of an appropriate rank one symmetric space $\bar M$. Let $\s{S}_{t}(p)$ be the geodesic sphere of  $\bar M$ with center $p$ and radius $0<t< \inj(p)$. Observe that $\inj(p)=\pi/2$ if $\bar M$ is of  compact type and $\infty$ if $\bar M$ is of  non-compact type. Let  $\II^t$ be the second fundamental form of 
$\s{S}_{t}(p)$ and $\mathcal{S}^t$ the  shape operator of $\s{S}_{t}(p)\subset \bar M$. By Remark \ref{rem:shapeoperatoreigenspaces}, we have
\begin{equation}
\label{eq: 432} \mathcal{S}^t_{\vert \mathcal D_i^t}= \beta _i^t \Id_{\vert \mathcal D_i^t}, \ \ \ \ i=1,2, 
\end{equation} 
where $\beta _1^t$ and  $\beta _2^t$ are scalars defined by equations (\ref{eq:b1}) and (\ref{eq:b2}), respectively. From these expressions one obtains that 
$\beta _1^t\neq \beta _2^t$.

We define a  symmetric bilinear form $\upalpha$ on $\s{S}_t(p)$ given by
\begin{equation}
\label{eq:alphadef}
\upalpha_x(v,w)=\langle\II_x^t(v,w),\xi_x\rangle, \quad\text{$v,w\in T_x\s{S}_t(p)$ and $x\in\s{S}_t(p)$},
\end{equation}
where $\xi$ is the outer unit normal field to the geodesic sphere.
Observe that $\upalpha$ is positive definite if and only if $t\in(0,\tfrac{\pi}{4})$ or $t>1$, see Equations~\eqref{eq:b1} and~\eqref{eq:b2}. By identifying $\s{S}_t(p)$ with the appropiate Hopf-Berger sphere $\s{S}^n_{\F,\tau}$, we can equip $\s{S}^n_{\F,\tau}$ with $\upalpha$. Consequently,  $\upalpha$ is non-degenerate for $\tau\neq \frac{1}{2}$. Moreover, $\upalpha$ defines a  Riemannian metric if and only if $\tau>\tfrac{1}{2}$, and a pseudo-Riemannian metric with index equal to $\dim\F -1$ for $\tau<\tfrac{1}{2}$, where $\F\in\{\C,\H,\O\}$.

We say that a vector $v\in T_p\s{S}^n_{\F,\tau}$ is \textit{$\upalpha$-isotropic} if $\upalpha(v,v)=0$, and a subspace $V$ of $T_p\s{S}^n_{\F,\tau}$ is \textit{$\upalpha$-isotropic} if every vector $v\in V$ is $\upalpha$-isotropic.

\begin{theorem}
\label{th:notwellpos}
Let $\Sigma$ be a complete immersed submanifold of $\s{S}^n_{\F,\tau}$, $\tau\neq1$, with dimension $d\ge2$, and let $\bar{M}$ be the rank one symmetric space containing $\s{S}^n_{\F,\tau}$ as a geodesic sphere. If $\tau\neq\tfrac{1}{2}$, then the following statements are equivalent:
\begin{enumerate}
	\item[i)] $\Sigma$ is a not well-positioned totally geodesic submanifold of $\s{S}^n_{\F,\tau}$.
	\item[ii)] $T_p\Sigma$ is $\upalpha$-isotropic and invariant under the curvature tensor of $\s{S}^n_{\F,\tau}$ for every $p\in \Sigma$.
\end{enumerate}
Moreover, under any of the above conditions $\Sigma$ is a totally geodesic submanifold of $\bar M$, and if $\tau=\tfrac{1}{2}$, every totally geodesic submanifold of $\s{S}^n_{\F,\tau}$ is well-positioned.
\end{theorem}
\begin{proof}
Let us prove that a not well-positioned totally geodesic submanifold ${\Sigma}$  of $\s{S}^n_{\F,\tau}$ with dimension $d\ge 2$ is $\upalpha$-isotropic. Firstly, by Lemma~\ref{lemma:notwellposrk1}, we have that ${\Sigma}$ is isometric to a compact rank one symmetric space. Thus, by the classification of totally geodesic submanifolds in rank one symmetric spaces, $\Sigma$ contains a totally geodesic surface $\widetilde{\Sigma}$, which is a not well-positioned totally geodesic submanifold of $\s{S}^n_{\F,\tau}$, see Figure~\ref{fig:tgrk1}. By Corollary~\ref{cor:tghighdimension}, we have that $\F=\H$ or $\F=\O$. Moreover, by Lemma~\ref{lemma:2octo},  we can assume without loss of generality that $\widetilde{\Sigma}$ is a not well-positioned totally geodesic submanifold of $\s{S}^{4m+3}_{\H,\tau}$.

Let  $\g{p}_{\widetilde{\Sigma}}$ be the tangent space of $\widetilde{\Sigma}$ at $o\in \widetilde{\Sigma}$ identified with a subspace of the reductive complement~$\g{p}$ of $\s{S}^{4m+3}_{\H,\tau}$, see Subsection~\ref{subsect:geomhopf}.
Using the isotropy of $\s{S}^{4m+3}_{\H,\tau}$ (see~\cite{verdiani-ziller}), we can assume that $m=2$, and $\g{p}_{\widetilde{\Sigma}}$ is spanned by the basis $\{u,v\}$ given by
\begin{equation*}
	\begin{aligned}
		u&=a_1 X_1/\sqrt{\tau}+a_2 X_2/\sqrt{\tau}+ a_3 X_3/\sqrt{\tau} +a_4 Y_1,\\
		v&=b_1 X_1/\sqrt{\tau}+b_2 X_2/\sqrt{\tau} + b_3 X_3/\sqrt{\tau} +b_4 Y_1+ b_5 J_1 Y_1+b_6 Y_2,
	\end{aligned}
\end{equation*}
where $a_i, b_j\in\R$ for $i\in\{1,\ldots,4\}$ and $j\in\{1,\ldots,6\}$.

By Lemma~\ref{lemma:splitting}, $\g{p}_{\widetilde{\Sigma}}$ does not contain horizontal or vertical vectors, since $\widetilde{\Sigma}$ is not well-positioned. Hence, $a_4\neq 0$. Moreover, we can assume without loss of generality that $b_4=0$. Otherwise, the set $\{u,v-b_4/a_4 u\}$ is also a basis for $\g{p}_{\widetilde{\Sigma}}$ where the second vector projects trivially over $Y_1$. 

Since $\g{p}_{\widetilde{\Sigma}}$ is curvature invariant, we have
\begin{equation}
	\label{eq:tgsurface1}
	\begin{aligned}
		0&=\langle R(u,v)v, J_1 Y_2\rangle=3(a_3b_2-a_2b_3+a_4b_5)(-1+\tau)b_6,\\
		0&=\langle R(u,v)v, J_2Y_2\rangle=3(a_3b_1-a_1b_3)(1-\tau)b_6,\\
		0&=\langle R(u,v)v, J_3Y_2\rangle=3(a_2b_1-a_1b_2)(-1+\tau)b_6.
	\end{aligned}
\end{equation}
Every totally geodesic submanifold of a homogeneous space is intrinsically homogeneous and every homogeneous space of dimension two has constant sectional curvature. Then, every covariant derivative of the curvature tensor $R$ restricted to $\g{p}_{\widetilde{\Sigma}}$ vanishes. Thus,
\[0=\langle (\nabla_u R)(v,u)u, Y_1\rangle=-4a_1a_4^2b_5\sqrt{\tau}(-1+\tau). \]
This implies that $a_1=0$ or $b_5=0$. In the latter case, $b_6\neq 0$, since otherwise there would be a vertical vector in $\g{p}_{\widetilde{\Sigma}}$. Thus,  by Equation~(\ref{eq:tgsurface1}) 
\[a_3b_2-a_2b_3=a_3b_1 -a_1b_3=a_2b_1-a_1b_2=0.   \]
Hence,  $(a_1,a_2,a_3)$ and $(b_1,b_2,b_3)$ are proportional and there would be a horizontal vector in~$\g{p}_{\widetilde{\Sigma}}$. Now assume that $b_5\neq0$, and thus $a_1=0$. Then,
\[
0=\langle R(u,v)v,J_2 Y_1\rangle=3a_2b_1b_5(1-\tau),\quad
0=\langle R(u,v)v, J_3 Y_1\rangle=3a_3b_1b_5(1-\tau).\]
Thus, $b_1=0$ or $a_2=a_3=0$, but the latter cannot happen because it would imply the existence of a horizontal vector. Hence, $b_1=0$. Moreover, we have
\[0=\langle(\nabla_u R)(v,u)v, X_1\rangle=4 a_4(a_2 b_2 + a_3 b_3) b_5 (-1+\tau)(-1+2\tau).   \]
This yields $\tau=1/2$ or $a_2 b_2 + a_3 b_3=0$. Let us assume that $\tau=1/2$. Then,
\[
0=\langle(\nabla_u R)(v,u)u, J_2 Y_1\rangle=-\sqrt{2} a^3_4 b_2,\quad
0=\langle(\nabla_u R)(v,u)u, J_3 Y_1\rangle=-\sqrt{2} a^3_4 b_3.\]
Hence, $b_2=b_3=0$ and we obtain a contradiction due to the existence of a horizontal vector in $\g{p}_{\widetilde{\Sigma}}$. This proves that there are no not well-positioned totally geodesic surfaces in $\s{S}^n_{\F,1/2}$.

Now, let us assume that $a_2 b_2 + a_3 b_3=0$. Then, there exist $r,s\in\R$ and $\varphi\in[0,2\pi]$ such that
\[ a_2=r\cos(\varphi), \quad a_3=r\sin(\varphi), \quad b_2=-s\sin(\varphi), \quad b_3=s\cos(\varphi).    \]
Thus, 
\begin{align*}
	0&=\langle (\nabla_u R)(v,u)u,X_1\rangle=4 a_4 b_5(-1+\tau)(a^2_4\tau +r^2(-1+2\tau)),\\
	0&=\langle (\nabla_v R)(u,v)v,X_1\rangle=4 a_4 b_5(-1+\tau)(s^2(-1+2\tau) + \tau b^2_5+\tau b^2_6).
\end{align*}
Hence, 
\[
a_4=\pm r\sqrt{\frac{1-2\tau}{\tau}},\qquad b_5=\pm\sqrt{\frac{s^2(1-2\tau) - \tau b^2_6}{\tau}}.
\]
To sum up, we have
\begin{align*}
	u&=r \cos(\varphi)\frac{X_2}{\sqrt{\tau}} + r\sin(\varphi)\frac{X_3}{\sqrt{\tau}} \pm \sqrt{ \frac{r^2-2r^2\tau}{\tau}}Y_1,\\
	v&= -s \sin(\varphi)\frac{X_2}{\sqrt{\tau}} + s \cos(\varphi)\frac{X_3}{\sqrt{\tau}}\pm  \sqrt{\frac{s^2(1-2\tau) - \tau b^2_6}{\tau}}J_1Y_1+ b_6 Y_2.
\end{align*}
Now a simple computation yields that $\g{p}_{{\widetilde{\Sigma}}}$ is $\upalpha$-isotropic. Therefore, by Lemma~\ref{lemma:notwellpos(-1)}, ${\Sigma}$ is $\upalpha$-isotropic. This completes the proof that \textit{i)} implies \textit{ii)} and since there are no not well-positioned totally geodesic surfaces in $\s{S}^n_{\F,1/2}$, we also have proved that there are no not well-positioned totally geodesic submanifolds in $\s{S}^n_{\F,1/2}$. 
\color{black}

Conversely, by the Gauss equation associated with the embedding of $\s{S}^n_{\F,\tau}$ in $\bar M$, if $T_o\Sigma$ is $\upalpha$-isotropic and invariant under the curvature tensor of $\s{S}^n_{\F,\tau}$, then $T_o\Sigma$ is invariant under the curvature tensor of~$\bar M$.  Thus, $\bar{\Sigma}:=\overline{\exp}_o(T_o\Sigma)$ is a totally geodesic submanifold of $\bar M$. Observe that the velocity of every geodesic $\gamma$ of $\Sigma$ is $\upalpha$-isotropic for every $t\in\R$, since $\Sigma$ is a g.o.\ space. This implies that $\gamma$ is a geodesic of $\bar \Sigma$, and the Riemannian exponential map $\exp_o$ of $\s{S}^n_{\F,\tau}$, and the Riemannian exponential map  $\overline{\exp}_o$ of $\bar{M}$ agree on $T_o\Sigma$. Thus, $\Sigma=\exp_o(T_o\Sigma)=\overline{\exp}_o(T_o\Sigma)=\bar{\Sigma}$. Consequently, $\Sigma$ is a totally geodesic submanifold of~$\bar M$.\qedhere

\end{proof}
By Remark~\ref{rem:intersect}, well positioned totally geodesic submanifolds of $\s{S}^n_{\F,\tau}$ are embedded. Moreover, by Theorem~\ref{th:notwellpos}, every not well-positioned totally geodesic submanifold is also embedded. Then, we have the following result.
\begin{corollary}\label{cor:embedded55} Let $\Sigma$ be a complete totally geodesic submanifold of $\s{S}^n_{\F,\tau}$. Then, $\Sigma$ is embedded in~$\s{S}^n_{\F,\tau}$.
\end{corollary}

Now we are in a position to prove the first main result of this article.
\begin{proof}[Proof of Theorem~\ref{th:a}]
Notice that $\upalpha$-isotropic subspaces exist if and only if $\tau\leq\tfrac{1}{2}$, see~Equations~\eqref{eq:b1}, \eqref{eq:b2}, and \eqref{eq:alphadef}. Then, by Theorem~\ref{th:notwellpos}, if $\tau\ge \tfrac{1}{2}$ we have that every totally geodesic submanifold of $\s{S}^n_{\F,\tau}$ is well-positioned. Also, by Corollary~\ref{cor:tghighdimension}, if $\dim\Sigma\ge\dim \F$, then $\Sigma$ is well-positioned.

Let us identify $\s S^n_{\F,\tau}$ with a geodesic sphere $\s{S}(p)$ centered at some point $p$ of a rank one symmetric space $\bar M$. The intersection of a geodesic sphere $\s{S}(p)$ of a symmetric space of rank one  $\bar{M}$ with a complete totally geodesic submanifold $M$ passing through $p\in \bar{M}$ is a geodesic sphere of~$M$ (see~Remark \ref{rem:intersect}).  Hence, the intersection of $\s{S}(p)$ with a complete totally geodesic submanifold $M$ of $\bar{M}$ of dimension $d'\ge3$ containing $p$ has dimension $d'-1$. Thus, it follows by Theorem~\ref{th:wellpos} that \textit{i)} and \textit{ii)} are equivalent.

Finally, by Remark~\ref{rem:congwellpos}, two isometric well-positioned totally geodesic submanifold must be congruent, and this concludes the proof. 
\end{proof}

\subsection{Not well-positioned maximal totally geodesic spheres}
\label{subsect:notwellposspheres}
This subsection is devoted to the study of  not well-positioned maximal totally geodesic spheres. Thus, by Theorem~\ref{th:a}, we may assume that $\tau <1/2$, and   $\F=\H$ or $\F=\O$.

Recall that we have the following presentations: $\s{S}^{4n+3}_{\H,\tau}=\s{G}/\s{K}=\s{Sp}_{n+1}/\s{Sp}_n$ and $\s{S}^{15}_{\O,\tau}=\s{G}/\s{K}=\s{Spin}_9/\s{Spin}_7$.
Let $o\in \s{S}^n_{\F,\tau}$ be fixed and let 
$F(o)$ be the fiber of the corresponding Hopf fibration by $o$. Let us consider the reductive decomposition of $\s{S}^{n}_{\F,\tau}$ appearing in Subsection~\ref{subsect:geomhopf}, which is given by $\g{g}=\g{k}\oplus\g{p}$, with $\g{p}=\g{p}_1\oplus\g{p}_2$, where $\g{p}_1$ is identified with the vertical distribution $\mathcal{V}$ at $o$, i.e.\ the tangent space of the fiber $F(o)$ at $o\in \s{S}^n_{\F,\tau}$, and $\g{p}_2$ is identified with the horizontal  distribution $\mathcal{H}$ at $o\in\s{S}^n_{\F,\tau}$.

Let $X\in\g{p}_2$ be a unit vector. In order to simplify our computations we will assume that $X=Y_1$, defined as in Subsection~\ref{subsect:geomhopf}. Using Equation~\eqref{eq:nabla}, we have for each $i\in\{1,\ldots,\dim \F-1\}$,
\begin{equation}\label{eq:nablarotation}
(\nabla  _{X^*_i}Y_1^{*})_o =-\tau\frac{\dim \F}{4}(J_i X)^*_o,\qquad (\nabla_{Y_1}^* Y_1^*)_o=0, \qquad (\nabla_{(J_i Y_1)^*} Y_1^*)_o= \frac{4}{\dim\F} (X_i)^*_o.
\end{equation}
Thus, Equation~\eqref{eq:nablarotation} implies that the linear transformation $e^{s(\nabla Y_1^*)_o}$ of $\g{p}$ is the rotation of angle $s\sqrt{\tau}$ in each $2$-plane spanned by $\{(J_i Y_1)^*, X^*_i\}$, and fixes the orthogonal complement of all these $2$-planes. Hence, in particular, $e^{s(\nabla Y_1^*)_o}$ fixes $Y_1$. Thus, since the slope of the subspace $e^{s(\nabla Y_1^*)_o}\mathcal{V}_o$, which is spanned by $\{ e^{s(\nabla Y_1^*)_o} (X_i)^*_o  \}^{\dim \F-1}_{i=1}$, is $0$ for $s=0$, and is $\infty$ for $s=\frac{\pi}{2\sqrt{\tau}}$, by continuity, the subspace $e^{s'(\nabla Y_1^*)_o}\mathcal{V}_o$ is  $\upalpha$-isotropic for some $s'>0$.

Let us denote by $\phi_s$ the flow associated with $Y_1^*$ in $\s{S}^n_{\F,\tau}$. Then, the integral curve of $Y_1^*$ given by $\phi_s(o)=\gamma(s)$ is a geodesic of $\s{S}^n_{\F,\tau}$ by Equation~\eqref{eq:nablarotation}. Moreover, let us consider the Equation~(2.2.1) of~ \cite{DO-nullity}, which relates the parallel transport $\s{P}_\gamma$ along a geodesic $\gamma$ with the pushforward $\phi_{s_{*}}$ of the flow of $Y_1^*$ in the following manner
\begin{equation}
\label{eq:magicformula}
\s{P}^{-1}_{\gamma(s)}\circ\phi_s{_{*o}}=e^{s\nabla (Y_1^*)_o}. 
\end{equation}   

In the following, we  regard, up to a homothetical factor, the Hopf-Berger sphere $\s S^n_{\F,\tau}$ of factor $\tau <1/2$ as a geodesic sphere $\s{S}_{t}(p)$ of radius $t\in(\pi/4,\pi/2)$ of a  
symmetric space $\bar {M}$ of sectional curvatures between $1$ and $4$, see Equation~\eqref{eq:radius_tau_homothety}. 
The shape operator $\mathcal{S}^t$ of  the  geodesic sphere 
$\s{S}_{t}(p)$ with respect to the outer unit normal vector field $\xi$ leaves the vertical and horizontal distributions invariant and has only two eigenvalues of different sign: the vertical one $\beta_1^t = -2\cot(2t)$ and the horizontal one $\beta_2^t=-\cot(t)$, see Remark \ref{rem:shapeoperatoreigenspaces}. 
Let $\bar \nabla$ and $\bar R$ be the Levi-Civita connection and the curvature tensor of $\bar M$. We have
$\langle \bar{\nabla}_{(Y_{1})^*_o}(Y_{1})^*_o, \xi _o\rangle = 
-\langle \mathcal{S}^t(Y_{1})^*_o),(Y_{1})^*_o\rangle = \cot (t)= \sqrt{\frac{\tau}{1-\tau}}$. Then,
\begin{equation}
\label{eq:roty_1xi}
\bar{\nabla}_{Y_1^*}(Y_1^*)_o = \sqrt{\frac{\tau}{1-\tau}}\, \xi _o,\quad\bar{\nabla}_{\xi}(Y_1^*)_o = -\sqrt{\frac{\tau}{1-\tau}}\, (Y_1^*)_o.
\end{equation}

The tangent space of a Helgason sphere $\s{S}_4^{\dim\F}$ that contains the radial geodesic from the center $p$ of $\s{S}_t(p)$ to $o\in\s{S}_t(p)$ is $T_o\s{S}_4^{\dim\F}=\mathcal{V}_o\oplus\R\xi_o$. Hence, $T_o\s{S}_4^{\dim\F}$ is invariant under $\bar R$, since $\s{S}_4^{\dim\F}$ is a totally geodesic submanifold of $\bar M$. 

Recall that the isometries of $\s{S}^n_{\F,\tau}$ can be extended to isometries in the isotropy of $\bar M$ at $p\in\bar M$, see~Remark~\ref{rem:isom}. Hence, for simplicity, we will also denote by $(Y_1)^*$ and $\phi_s$, the Killing vector field induced by $Y_1$ in $\bar M$ and its associated flow, respectively.  Then, taking into account that the flow $\phi_s$ is given by isometries, we deduce that
$\mathrm ({\phi}_s)_{*o} (T_o\s S_4^{\dim \F})$ and $({\phi}_s)_{*o}(\mathcal{V}_o)=\mathcal{V}_{\gamma(s)}\subset T_{\gamma(s)}\s{S}_t(p)$ are totally geodesic subspaces of $T_{\gamma(s)}\bar M$. Since $({\phi}_s)_{*o} (\mathcal{V}_o)$ is perpendicular to $\gamma'(s)$, taking into account that $\gamma (s)$ is a geodesic of $\s{S}_t(p)$ and that $\gamma '(s)$ is a horizontal eigenvector of $\mathcal{S}^t$ for all~$s$,
one has that 
$$\bar {\s P}_{\gamma(s)} (v) = 
{\s P}_{\gamma(s)} (v)\quad\text{for all $v\in  \mathcal{V}_o$},$$
where  $\bar{\s P}_{\gamma(s)}$  denotes the parallel transport in  $\bar{M}$
along $\gamma (s) = \phi_s(o)$. Then, by Equation~\eqref{eq:magicformula}, 
\begin{equation}\label{eq: 0321}
\mathrm{e}^{s(\bar {\nabla} Y_1^{*})_o}
\mathcal{V}_o = \bar{\s{P}}^{-1}_{\gamma(s)}((\phi_s)_{*o} (\mathcal{V}_o)) ={\s{P}}^{-1}_{\gamma(s)}((\phi_s)_{*o} (\mathcal{V}_o)) =  \mathrm{e}^{s( {\nabla} Y_1^{*})_o}
\mathcal{V}_o.
\end{equation}
Therefore, $\mathrm{e}^{s( {\nabla} Y_1^{*})_o}
\mathcal{V}_o$ is the tangent space of a totally geodesic hypersphere of a Helgason sphere of $\bar M$, and thus invariant under $\bar R$. This implies by Gauss equation that $\mathrm{e}^{s( {\nabla} Y_1^{*})_o}
\mathcal{V}_o$ is a $R$-invariant subspace of $\g{p}$.   Consequently, taking into account that the Levi-Civita connection is preserved by homotheties, and the paragraph below Equation~\eqref{eq:nablarotation}, there is some $s'>0$ such that  $\mathrm{e}^{s'( {\nabla} Y_1^{*})_o}
\mathcal{V}_o$ is an $\upalpha$-isotropic, $R$-invariant subspace of $\g{p}$. Therefore, by Theorem~\ref{th:notwellpos}, $\mathrm{e}^{s'( {\nabla} Y_1^{*})_o}
\mathcal{V}_o$ is the tangent space of a not well-positioned totally geodesic submanifold $\widehat{\Sigma}$ of $\s{S}^n_{\F,\tau}$.
Furthermore notice that by Equation~\eqref{eq:roty_1xi},
\[ \mathrm{e}^{s'(\bar {\nabla} Y_1^{*})_o}
T_o\s{S}_4^{\dim\F}= \mathrm{e}^{s'(\bar {\nabla} Y_1^{*})_o}
\mathcal{V}_o\oplus \R(\cos(s'\sqrt{\frac{\tau}{1-\tau}})(Y_1^*)_o+\sin(s'\sqrt{\frac{\tau}{1-\tau}})\xi_o).   \]
Thus, this subspace is tangent to a Helgason sphere of $\bar M$ and contains the tangent space of $\widehat{\Sigma}$ at $o$ which is equal to $\mathrm{e}^{s'(\bar {\nabla} Y_1^{*})_o}
\mathcal{V}_o$. Consequently, $\widehat{\Sigma}$ is a totally geodesic hypersurface of the aforementioned Helgason sphere of $\bar M$, and by using Equation~\eqref{eq:nablarotation}, and the fact that $\mathrm{e}^{s'(\bar {\nabla} Y_1^{*})_o}
\mathcal{V}_o$ is $\upalpha$-isotropic,  we can compute that
\begin{equation}
\label{eq:tangentwidehat}
T_o\widehat{\Sigma}=\spann\left\{\frac{4X_i}{\dim\F\sqrt{1-\tau}} - \sqrt{\frac{2 \tau -1}{\tau -1}}J_i Y_1\right\}^{\dim\F-1}_{i=1}.
\end{equation}
Furthermore, $\mathrm{e}^{s'(\bar {\nabla} Y_1^{*})_o}
T_o\s{S}_4^{\dim\F}$ projects non trivially over $\xi_o$. Thus, the Helgason sphere whose tangent space is equal to $\mathrm{e}^{s'(\bar {\nabla} Y_1^{*})_o}
T_o\s{S}_4^{\dim\F}$ and  the geodesic sphere $\s{S}_t(p)$ intersect transversally at $o$. 
\begin{remark}
\label{rem:reductiondimension}
Observe that if $\F=\H$, by looking at Equation~\eqref{eq:tangentwidehat}, it is clear that $\widehat{\Sigma}$ is totally geodesic of  a well-positioned totally geodesic submanifold $\s{S}^{7}_{\H,\tau}$ of $\s{S}^{4n+3}_{\H,\tau}$, see Subsection~\ref{subsect:examplestg}. Furthermore, the totally geodesic submanifold $\widehat{\Sigma}$ is a maximal totally geodesic submanifold of $\s{S}^7_{\H,\tau}$ and $\s{S}^{15}_{\O,\tau}$. Indeed, by Corollary~\ref{cor:tghighdimension}, $\widehat{\Sigma}$ has the highest dimension that a not well-positioned totally geodesic submanifold can attain, and by Theorem~\ref{th:wellpos}, well-positioned totally geodesic submanifolds of $\s{S}^7_{\H,\tau}$ or $\s{S}^{15}_{\O,\tau}$ have dimension less or equal than $3$ or $7$, respectively; see~Table~\ref{table:hopfbergertgwellpos}.
\end{remark}
\begin{theorem}
\label{th:notwellsphere}
Let $\Sigma$ be a complete not well-positioned  totally geodesic sphere of $\s{S}^n_{\F,\tau}$, $\tau\neq1$.	Then, the following statements hold:
\begin{enumerate}
	\item[i)]
	There is a unique maximal not well-positioned totally geodesic submanifold $\widehat\Sigma$ of $\s{S}^{n}_{\F,\tau}$ isometric to $\s{S}^{\dim\F-1}_{4-4\tau}$  with $T_o\Sigma\subset T_o\widehat\Sigma$. 
	
	\item[ii)] Every two not well-positioned totally geodesic spheres in $\s{S}^{n}_{\F,\tau}$ of dimension~$\dim\F-1$ are congruent in $\s{S}^{n}_{\F,\tau}$ by some element in the isotropy of $\s{S}^n_{\F,\tau}$.
\end{enumerate}
\end{theorem}
\begin{proof}
Let us regard, up to a homothetical factor, the Hopf-Berger sphere $\s S^n_{\F,\tau}$  as a geodesic sphere $\s{S}_{t}(p)$ of a  rank one
symmetric space $\bar {M}$ with maximal sectional curvature equal to $4$. Now taking into account that $\s{S}^n_{\F,\tau}$ and $\s{S}_t(p)$ are related by a homothety of ratio $\alpha=\sqrt{1-\tau}$, see~Equation~\eqref{eq:radius_tau_homothety}; we deduce that $\widehat{\Sigma}$ (as constructed above) is isometric to $\s{S}^{\dim\F-1}_{4-4\tau}$.

Let $v\in T_o\Sigma$ be a unit vector. Then, there exists a unique Helgason sphere $\s{S}^{\dim \F}_4$  of $\bar M$ with $v\in T_o\s{S}_4^{\dim\F}$.  We may  asumme that $v\in T_o\widehat{\Sigma}$ by  using the isotropy of $\s{S}^n_{\F,\tau}$, and therefore $\widehat{\Sigma}\subset\s{S}_4^{\dim\F}$. If $\Sigma$ is not contained in $\widehat{\Sigma}$, we have that $\dim(T_o\Sigma + T_o\widehat{\Sigma})\ge \dim \F$. Moreover, this implies that $\dim((T_o\Sigma\ominus\R v) +(T_o\widehat{\Sigma}\ominus\R v))\ge \dim \F-1$. Since $(T_o\Sigma\ominus\R v) +(T_o\widehat{\Sigma}\ominus\R v)$ is contained in the eigenspace associated with the eigenvalue~$4$ of the Jacobi operator $\bar R_v$ (see Subsection~\ref{subsect:rkonetg}), we deduce that $\dim(T_o\Sigma +T_o\widehat{\Sigma})=\dim \F$, and $T_o\Sigma +T_o\widehat{\Sigma}$ must be equal to the tangent space of $\s{S}_4^{\dim\F}$. However, this yields a contradiction since this Helgason sphere $\s{S}_4^{\dim\F}$ intersects the geodesic sphere $\s{S}_t(p)$ transversally at some point, as we pointed out in the paragraph above Equation~\eqref{eq:tangentwidehat}. This proves \textit{i)}.

Finally, notice that \textit{i)} implies \textit{ii)} since the isotropy at $o$ of $\s{S}^n_{\F,\tau}$ acts transitively on the set of vectors with the same length and slope.\qedhere	
\end{proof}

\section{Totally geodesic submanifolds in quaternionic Hopf-Berger spheres}
\label{sect:quaternionic}
The purpose of this section is to achieve the classification of totally geodesic submanifolds in $\s{S}^{4n+3}_{\H,\tau}$.
\subsection{Not well-positioned totally geodesic projective spaces}
\label{subsect:tgproj}
In what follows, we classify not well-positioned totally geodesic submanifolds in $\s{S}^{4n+3}_{\H,\tau}$ isometric to projective spaces.

First of all, notice that a not well-positioned totally geodesic projective space $\Sigma$ in $\s{S}^{4n+3}_{\H,\tau}$ has constant sectional curvature equal to $(1-\tau)$, since  not well-positioned totally geodesic spheres have sectional curvature equal to $4(1-\tau)$, see Theorem~\ref{th:notwellpos} and Theorem~\ref{th:notwellsphere}.
Let $\g{p}_{\Sigma}\subset \g{p}$ be the tangent space of $\Sigma$ at $o$. Then, we can define
\[ T_X(Y):=R(X,Y,X) +(1-\tau)\langle X,X\rangle Y, \quad\text{for $X,Y\in\g{p}$.} \]
Since $\Sigma$ has constant sectional curvature equal to $(1-\tau)$, we have $T_{X}(Y)=0$ for every $X,Y\in\g{p}_\Sigma$.

As a consequence of the classification of totally geodesic submanifolds in rank one symmetric space~(see~Figure~\ref{fig:tgrk1}), Corollary~\ref{cor:tghighdimension}, and Theorem~\ref{th:notwellpos},  every not well-positioned totally geodesic projective space of $\s{S}^{4n+3}_{\H,\tau}$ has dimension two or three. Let us first start with the $2$-dimensional case, i.e.\ projective planes.
\begin{lemma}
\label{lemma:notwellposproj2}
A not well-positioned totally geodesic surface of $\s{S}^{4n+3}_{\H,\tau}$ passing through $o\in\s{S}_{\H,\tau}^{4n+3}$ is isometric to a projective plane $\R \s P^2_{1-\tau}$ if and only if it is congruent  to $\Sigma=\exp_o(\g{p}_\Sigma)$ by some element of the full isotropy of $\s{S}^{4n+3}_{\H,\tau}$, where $\g{p}_\Sigma$ is spanned by one of the following orthogonal bases:

\begin{align*}
	&\textup{(a)}& 	\mathcal{A}^2_{\tau}&:=\left\{\frac{X_2}{\sqrt{\tau}} +\sqrt{ \frac{c_2}{\tau}}Y_1, \frac{X_3}{\sqrt{\tau}}+ \sqrt{\frac{\tau}{c_2}}J_1Y_1+\sqrt{\frac{c_1 c_3}{\tau c_2}} Y_2\right\}, \enspace \tau<\frac{1}{3},\enspace n\ge 2, \enspace \text{or}\\
	&\textup{(b)}& 	\mathcal{A}^2_{1/3}&:=\left\{\frac{X_2}{\sqrt{\tau}} +\sqrt{ \frac{c_2}{\tau}}Y_1, \frac{X_3}{\sqrt{\tau}}+ \sqrt{\frac{\tau}{c_2}}J_1Y_1\right\}, \enspace \tau=1/3,
\end{align*}
where $c_m:=1-m\tau$  for each $m\in\{1,2,3\}$.
\end{lemma}
\begin{proof}
Observe that the subspaces spanned by $\mathcal{A}^2_\tau$, $\tau<1/3$ and  $\mathcal{A}^2_{1/3}$ are $\upalpha$-isotropic, curvature invariant, and their sectional curvatures are equal to $(1-\tau)$, and $2/3$, respectively. Then, by Theorem~\ref{th:notwellpos}, the exponential image of both subspaces are totally geodesic projective planes.

Conversely, notice that using the isotropy representation and the fact that $\Sigma$ is not well-positioned, we can assume without loss of generality that $\g{p}_{\Sigma}$ is spanned by two vectors $u$ and $v$ given by
\begin{align*}
	u&=a_1 X_1/\sqrt{\tau} + a_2 X_2/\sqrt{\tau} + a_3 X_3/\sqrt{\tau}+ a_4 Y_1,\\
	v&=b_1 X_1/\sqrt{\tau} + b_2 X_2/\sqrt{\tau} + b_3 X_3/\sqrt{\tau} + b_4 J_1 Y_1 +b_5 Y_2.
\end{align*}
Recall that $\Sigma$ is isometric to a real projective plane, and therefore $\Sigma$ is locally symmetric. Thus,
\[0=\langle (\nabla_u R)(v,u,u),Y_1\rangle=4(1-\tau)\sqrt{\tau}a_1a_4^2b_4.  \]
Taking into account that $\Sigma$ is not well-positioned, we have that $\tau<1/2$ by Theorem~\ref{th:notwellpos}, and $a_4\neq0$ by Lemma~\ref{lemma:splitting}. Thus,  $a_1=0$ or $b_4=0$. Let us assume that $b_4=0$. Then, we have for each $i\in\{1,2,3\}$,
\[0=\langle T_u(v), J_i Y_1\rangle=3(-1)^i(1-\tau)a_4(a_{i+2}b_{i+1}-a_{i+1}b_{i+2}),\enspace\mbox{ (indices modulo 3)}.  \]
This implies the existence of a horizontal vector that yields a contradiction by Lemma~\ref{lemma:splitting}. Let us assume that $a_1=0$. Then,
\[0=\langle T_u(v),J_2 Y_1\rangle=3(1-\tau)a_3a_4b_1,\qquad 0=\langle T_u(v),J_3 Y_1\rangle=3(\tau-1)a_2a_4b_1.   \]
Hence, $b_1=0$, since otherwise $\g{p}_{\Sigma}$ would contain a non-zero horizontal vector. Moreover,
\[0=\langle T_u(v), Y_1\rangle=a_4 \tau  (a_2 b_2+a_3 b_3).  \]
This implies that $a_2=r\cos(\varphi)$, $a_3=r\sin(\varphi)$, $b_2=-s\sin(\varphi)$, $b_3=s\cos(\varphi)$, for some $r,s>0$ and $\varphi\in[0,2\pi)$. However, we can assume without loss of generality, that $r=s=1$. By using the isotropy, we can further assume that $\varphi=0$. Hence,
\[0=\langle T_u(v), Y_2\rangle=b_5 (-a_4^2 \tau -2 \tau +1).  \]
This implies that $a_4=\pm\frac{\sqrt{c_2 }}{\sqrt{\tau }}$. Then, we have
\[0=\langle T_u(v),X_3\rangle=(\tau -1) ( \sqrt{\tau }\mp b_4 \sqrt{1-2 \tau }),  \]
which implies that $b_4=\pm \frac{\sqrt{\tau }}{\sqrt{c_2 }}$. By using the isotropy, we can assume that both $a_4$ and $b_4$ are positive. Finally, since $\g{p}_{\Sigma}$ is $\upalpha$-isotropic, we have
\[0=\upalpha(v,v)=\frac{b_5^2 (2 \tau -1) \tau +3 \tau ^2-4 \tau +1}{(2 \tau -1)\sqrt{(1-\tau) \tau } }. \]
This implies that $b_5=\pm\sqrt{\frac{c_1 c_3}{\tau  c_2}}$. Observe that $b_5\in\R$ when $\tau<1/3$. Using the isotropy we can assume that $b_5>0$, and then we obtain the subspace $\g{p}_\Sigma$ spanned by $\mathcal{A}^2_\tau$. On the one hand, observe that $b_5=0$ when $\tau=1/3$, obtaining the subspace spanned by $\mathcal{A}^2_{1/3}$. On the other hand, if $\tau<1/3$, we have $b_5\neq 0$, and then $\Sigma$ is a totally geodesic submanifold of $\s{S}^{11}_{\H,\tau}$. This proves the desired result.
\end{proof}

\begin{lemma}
\label{lemma:notwellposproj3}
A not well-positioned totally geodesic submanifold of $\s{S}^{4n+3}_{\H,\tau}$ passing through $o\in\s{S}^{4n+3}_{\H,\tau}$ is isometric to a projective space $\R \s P^3_{1-\tau}$ if and only if it is congruent to $\Sigma=\exp_o(\g{p}_\Sigma)$  by some element of the full isotropy of $\s{S}^{4n+3}_{\H,\tau}$, where $\g{p}_\Sigma$ is spanned by one of the following orthogonal bases:
\normalsize

\begin{align*}
	&\textup{(a)}&	\mathcal{A}^3_{\tau}&:=\mathcal{A}^2_{\tau}\cup\left\{ \frac{X_1}{\sqrt{\tau}}-\sqrt{\frac{\tau}{c_2}} J_3 Y_1 + \frac{\tau \sqrt{c_1} }{\sqrt{\tau  c_2 c_3}} J_2 Y_2 - \sqrt{\frac{c_1 c_4}{c_3 \tau}} Y_3 \right\}, \, \tau<1/4, \,  n\ge 3, \, or \\
	&\textup{(b)}&	\mathcal{A}^3_{1/4}&:=\mathcal{A}^2_{1/4}\cup\left\{ \frac{X_1}{\sqrt{\tau}}-\sqrt{\frac{\tau}{c_2}} J_3 Y_1 + \frac{\tau \sqrt{c_1} }{\sqrt{\tau  c_2 c_3}} J_2 Y_2 \right\}, \enspace \tau=1/4, \enspace n\ge 2,
\end{align*}
where $c_m:=1-m\tau$  for each $m\in\{1,2,3,4\}$.
\end{lemma}
\begin{proof}
Observe that the subspaces spanned by $\mathcal{A}^3_\tau$, $\tau<1/4$, and  $\mathcal{A}^2_{1/4}$ are $\upalpha$-isotropic, curvature invariant,  and their sectional curvatures are equal to $(1-\tau)$ and $3/4$, respectively. Then, by Theorem~\ref{th:notwellpos}, the exponential image of both subspaces are three dimensional totally geodesic projective spaces.

Conversely,  using the isotropy, the fact that $\Sigma$ is not well-positioned and Lemma~\ref{lemma:notwellposproj2}, we can assume without loss of generality that $\g{p}_{\Sigma}$ is spanned by three vectors $u$, $v$ and $w$ given by
\begin{align*}
	u=&\frac{X_2}{\sqrt{\tau}} +\sqrt{ \frac{c_2}{\tau}}Y_1,\qquad
	v=\frac{X_3}{\sqrt{\tau}}+ \sqrt{\frac{\tau}{c_2}}J_1Y_1+\sqrt{\frac{c_1 c_3}{\tau c_2}} Y_2,\\	w=&\frac{X_1}{\sqrt{\tau}}+a_4 Y_1+ a_5 J_1 Y_1
	+ a_6 J_2 Y_1 +a_7 J_3 Y_1+	\\&  a_8  Y_2+ a_9 J_1 Y_2 +a_{10} J_2 Y_2 +a_{11}J_3 Y_2 + a_{12} Y_3.
\end{align*}
Recall that $\Sigma$ has constant sectional curvature and therefore $\Sigma$ is locally symmetric. Thus,
\begin{equation*}
	0=\sqrt{\tau}\langle (\nabla_u R)(w,u,u),Y_1\rangle=4 a_6 c_1 c_2, \qquad 	0=\sqrt{\tau}\langle (\nabla_u R)(w,u,u),J_2Y_1\rangle=4 a_4 c_1 c_2. 
\end{equation*}
This implies that $a_4=a_6=0$. Now assume that $\tau= 1/3$. Then,
\[0=-3\sqrt{3}\langle (\nabla_v R)(w,v,v), J_2 Y_1\rangle=8 a_5.  \]
Thus, $a_5=0$. However,
\begin{equation*}
	0=\langle T_u(w), J_3 Y_1\rangle=-2(1+a_7),\qquad
	0=\langle T_v(w), J_3 Y_1\rangle=-2(1-a_7),
\end{equation*}
yielding a contradiction with the fact that $\tau=1/3$. Hence, we can assume that $\tau\neq 1/3$. Then,
\[0=\langle (\nabla_v R)(w,v,v),Y_2\rangle=\frac{4 a_{11} c_1^2 c_3}{\sqrt{\tau } c_2}.  \]
Consequently, $a_{11}=0$. Moreover,
\begin{equation*}
	0=\langle T_u(w),X_1\rangle=-3 c_1 (a_7 \sqrt{c_2}+\sqrt{\tau }),\qquad
	0=\langle T_u(w), X_3\rangle=3 a_5 \sqrt{c_2}c_1\tau,
\end{equation*}
which implies that $a_5=0$ and $a_7=-\sqrt{\tau}/\sqrt{c_2}$. Now,
\begin{equation*}
	0=\langle T_v(w), X_3\rangle=\tau a_8\sqrt{\frac{c_1 c_3}{c_2}},\qquad
	0=\langle T_v(w), X_2\rangle=3 a_9 (\tau -1) \sqrt{\frac{c_1 c_3}{c_2}}.
\end{equation*}
Therefore, we deduce that $a_8=a_9=0$. Furthermore,
\[0=\langle T_v(w),X_1\rangle=\frac{3 c_1 \sqrt{\tau } \left(a_{10} c_2 \sqrt{\frac{c_1 c_3}{c_2 \tau }}-c_1\right)}{c_2}.  \]
Hence, we deduce that $a_{10}= \frac{\sqrt{c_1\tau}}{\sqrt{c_2 c_3}}$. Finally, taking into account that $\g{p}_{\Sigma}$ is $\upalpha$-isotropic, we have
\[0=\upalpha(w,w)=\frac{a_{12}^2 (1-3\tau) \tau -4 \tau ^2+5 \tau -1}{c_3\sqrt{(1-\tau ) \tau }}.  \]
Consequently, we have $a_{12}=\pm \sqrt{\frac{c_1 c_4}{\tau c_3}}$. Using the isotropy we can assume that $a_{12}>0$, and then we obtain the subspace $\g{p}_\Sigma$ spanned by $\mathcal{A}^3_\tau$. On the one hand, observe that $a_{12}=0$ when $\tau=1/4$, obtaining the subspace spanned by $\mathcal{A}^3_{1/4}$. On the other hand, if $\tau<1/4$, we have $a_{12}\neq 0$, and then $\Sigma$ is a totally geodesic submanifold of $\s{S}^{15}_{\H,\tau}$. This proves the desired result.\qedhere
\end{proof}

\begin{proposition}
\label{prop:congquaternionic}
Every two isometric totally geodesic submanifolds of $\s{S}^{4n+3}_{\H,\tau}$ passing by $o$ with dimension $d\ge 2$ are isotropy-congruent at $o$.
\end{proposition}
\begin{proof}
Firstly, the result follows for well-positioned totally geodesic submanifolds by virtue of Remark~\ref{rem:congwellpos}. By Lemma~\ref{lemma:notwellposproj2} and Lemma~\ref{lemma:notwellposproj3}, two isometric  not well-positioned totally geodesic submanifolds of $\s{S}^{4n+3}_{\H,\tau}$  diffeomorphic to projective spaces passing through $o$ are isotropy-congruent at $o$. 
By Theorem~\ref{th:notwellsphere},  there is exactly one  isotropy-congruence class of not well-positioned  totally geodesic $3$-dimensional spheres.

Finally, we tackle the case of $2$-dimensional not well-positioned spheres. Firstly, we claim that  not well-positioned $3$-dimensional totally geodesic spheres are orbits of a subgroup of $\s{S}^{4n+3}_{\H,\tau}$  isomorphic to $\s{SO}_4$. By Lemma~\ref{lemma:finclas} and Theorem~\ref{th:notwellsphere}, every not well-positioned totally geodesic $3$-dimensional sphere $\widehat{\Sigma}$ is congruent to the orbit of some subgroup $\s{H}$ of $\Isom(\s{S}^{4n+3}_{\H,\tau})$ by $o\in\s{S}^{4n+3}_{\H,\tau}$.  By Remark~\ref{rem:reductiondimension}, we can assume without loss of generality that $n=1$. Now take some non-zero element $X$ in $T_o\widehat{\Sigma}$. Let $\s{H}_o$ be the isotropy of $\s{H}$ at $o$. 
Let $X,Y$ in $T_o\widehat{\Sigma}$ be of unit length. Since $X$ and $Y$ have the same slope, see~Lemma~\ref{lemma:notwellpos(-1)}, we have an element $k$ of the isotropy of $\s{S}^n_{\F,\tau}$ such that $k_{*o}X=Y$. By the uniqueness of the totally geodesic submanifold $\widehat{\Sigma}$ (see Theorem~\ref{th:notwellsphere}), we have that $k \widehat{\Sigma}=\widehat{\Sigma}$. This proves that $\widehat{\Sigma}$ is extrinsically $2$-point homogeneous. Therefore, by the classification of transitive actions on spheres, see~\cite{MS43}, $\s{H}$ is locally isomorphic to $\s{SO}_4$.

Now, observe that $\s{H}$ must act effectively on  $\widehat{\Sigma}$, otherwise there would be some $\varphi\in\s{H}$ such that $\widehat{\Sigma}\subset \mathrm{Fix}(\varphi)$, and the set of fixed points of an isometry is a well-positioned totally geodesic submanifold, see Lemma~\ref{lemma:fixedpoint}. Thus,  $\s{H}$ must be isomorphic to $\s{SO}_4$.

By Theorem~\ref{th:notwellsphere}, every not well-positioned two-dimensional sphere can be contained in $\widehat{\Sigma}$. Now since $\s H_o=\s{SO}_3$, there is some element $\varphi\in \s{H}_o$ acting transitively on the $2$-plane Grassmannian $\s{G}_2(T_o\widehat{\Sigma})\equiv\s{G}_2(\R^3)$, and therefore there is just one isotropy-congruence class at $o$ for not well-positioned totally geodesic spheres of dimension two passing through the point $o$.\qedhere
\end{proof}
Since Hopf-Berger spheres are g.o.\ spaces, as a consequence of Proposition~\ref{prop:congquaternionic} and Lemma~\ref{lemma:finclas}, we have the following.
\begin{corollary}
\label{cor:extrinsichomquaternionic}
Every totally geodesic submanifold of $\s{S}^{4n+3}_{\H,\tau}$ is extrinsically homogeneous.
\end{corollary}
We end up this section with the proof of Theorem~\ref{th:C}.
\begin{proof}[Proof of Theorem~\ref{th:C}]
On the one hand, if $\Sigma$ is well-positioned the proof follows by Theorem~\ref{th:wellpos}, Remark~\ref{rem:intersect}, and Remark~\ref{rem:congwellpos}. On the other hand, if $\Sigma$ is not well-positioned the proof follows by combining Theorem~\ref{th:notwellpos}, Lemma~\ref{lemma:notwellposproj2}, Lemma~\ref{lemma:notwellposproj3} and Proposition~\ref{prop:congquaternionic}.
\end{proof}
\section{Totally geodesic submanifolds in octonionic Hopf-Berger spheres}
\label{sect:octonionic}
The purpose of this section is to achieve the classification of totally geodesic submanifolds in $\s{S}^{15}_{\O,\tau}$.

Firstly, let us prove that the totally geodesic submanifold $\widehat{\Sigma}$ of $\s{S}^{15}_{\O,\tau}$ (see~Theorem~\ref{th:notwellsphere}) is an orbit of a subgroup $\s{H}$ of $\s{Spin}_9$ isomorphic to $\s{Spin}_7$. By Lemma~\ref{lemma:finclas} and Theorem~\ref{th:notwellsphere}, $\widehat{\Sigma}$ is extrinsically homogeneous.  According to~\cite{varadarajan}, if $\tau\neq1$, we have the following situation for certain subgroups of $\Isom(\s{S}^{15}_{\O,\tau})=\s{Spin}_9$,
\[\begin{tikzcd}
\s{Spin}_6 \arrow[hookrightarrow]{r} & \s{Spin}_7\arrow[hookrightarrow]{r}{\text{$i$}}  & \s{Spin}_8  \arrow[hookrightarrow]{r} & \s{Spin}_9
\\
\s{SU}_3 \arrow[hookrightarrow]{u}  \arrow[hookrightarrow]{r} & \s{G}_2 \arrow[hookrightarrow]{u} \arrow[hookrightarrow]{r} & \s{Spin}_7 \arrow[hookrightarrow]{u}[right]{\text{$i'$}} 
\end{tikzcd}
\]
In this commuting diagram the lower left corner of each square is the intersection of the subgroup that is above and the subgroup to the right of it. Every inclusion in this diagram is maximal. Moreover, it is well-known that apart from the natural inclusion $i$ of $\s{Spin}_7$ in $\s{Spin}_8$, there are  $2$ more non-standard inclusions $i'$ and $i''$, which together with $i$ lead to the $3$ distinct conjugacy classes of $\s{Spin}_7$-subgroups in $\s{Spin}_8$. These are related by the triality group $\mathrm{Out}(\s{Spin}_8)$ which acts transitively on the  $3$-conjugacy classes of its $\s{Spin}_7$-subgroups, see \cite{varadarajan}.

Let  $\s{K}_v$ be the  stabilizer for the isotropy action of $\s{S}^{15}_{\O,\tau}=\s{G}/\s{K}$ at a unit $\upalpha$-isotropic vector $v\in T_o\widehat{\Sigma}$.
The stabilizer for the isotropy action of $\s{S}^{15}_{\O,\tau}$ at a non-zero vertical vector or at  a horizontal vector   is isomorphic to $\s{Spin}_6$ or $\s{G}_2$, respectively. Then $\s{K}_v$ is isomorphic to the intersection of $\s{Spin}_6$ and $\s{G}_2$, which is  $\s{SU}_3$ by the previous discussion. Recall that $\s{H}$ is the subgroup of $\s{Spin}_9$ such that $\s{H}\cdot o=\widehat{\Sigma}$. By a similar argument  as in Proposition~\ref{prop:congquaternionic}, using the uniqueness of $\widehat{\Sigma}$~(see~Theorem~\ref{th:notwellsphere}), we have that $\widehat{\Sigma}$ is an extrinsically $2$-point homogeneous submanifold of $\s{S}^{15}_{\O,\tau}$. Then, $\s{H}{_o}v=\s{\widetilde{H}}{_o}/\s{\widetilde{K}}_v=\s{S}(T_o\widehat{\Sigma})=\s{S}^6$, where $\widetilde{\s {H}}_o$ and $\s{\widetilde{K}}_v$ denote the effectivized connected subgroups corresponding to the action of $\s{H}_o$ and $\s{K}_v$ on $T_o\widehat{\Sigma}$, respectively. Since $\s{K}_v$ is simple, then $\widetilde{\s{K}}_v$ is either trivial or locally isomorphic to $\s{SU}_3$. In the first case,  we would obtain a contradiction since $\s{S}^6$ is not a Lie group. Thus, $\widetilde{\s{K}}_v$ is locally isomorphic to $\s{SU}_3$. Consequently, by the classification of transitive actions on spheres, see~\cite{MS43}, $\widetilde{\s{H}}_o$ is  isomorphic to $\s{G}_2$ and $\widetilde{\s{K}}_v$  is  isomorphic to $\s{SU}_3$. Now, observe that $\s{H}$ must act effectively on  $\widehat{\Sigma}$, otherwise there would be some $\varphi\in\s{H}$ such that $\widehat{\Sigma}\subset \mathrm{Fix}(\varphi)$. The set of fixed points of an isometry is a well-positioned totally geodesic submanifold which has dimension less or equal than $7$, see Lemma~\ref{lemma:fixedpoint} and Table~\ref{table:hopfbergertgwellpos}. Thus,  $\s{H}$  must be isomorphic to $\s{Spin}_7$ by the classification of transitive actions on spheres.

The following lemma is a crucial step towards the classification of totally geodesic submanifolds in~$\s{S}^{15}_{\O,\tau}$.
\begin{lemma}
\label{lemma:dim7octonionic}
Let $\widehat{\Sigma}$ be a not well-positioned totally geodesic sphere in $\s{S}^{15}_{\F,\tau}$ of dimension $7$. 	Then, the following statements hold:
\begin{enumerate}
	\item[i)] The totally geodesic submanifold $\widehat{\Sigma}$ is an orbit of a subgroup of $\s{Spin}_9$ isomorphic to $\s{Spin}_7$ and the isotropy of the action of $\s{Spin}_7$ on $\s{S}^{15}_{\O,\tau}$ at $o\in\widehat{\Sigma}$ is isomorphic to $\s{G}_2$.
	\item[ii)]
	Let $\Sigma$ be a totally geodesic submanifold of $\s{S}^{15}_{\O,\tau}$ with dimension greater than two. Then, $\Sigma$ is not well-positioned if and only if it is congruent to a totally geodesic submanifold of $\widehat{\Sigma}$.
	\item[iii)] The trilinear map $\upvarphi$ defined by \[\upvarphi(X,Y,Z)=\frac{2 (1-\tau )^{3/2}}{1-2 \tau }\langle \proj_{\g{p}_2} X, [\proj_{\g{p}_1} Y, \proj_{\g{p}_2} Z]\rangle \quad \text{with} \enspace X,Y,Z\in T_o\widehat\Sigma\] is a $\s{G}_2$-invariant $3$-form.
\end{enumerate}
\end{lemma}
\begin{proof}	
The statement in \textit{i)} is proved in the discussion prior to this lemma.

Clearly, every totally geodesic submanifold $\Sigma$ of $\widehat{\Sigma}$ is not well-positioned. Conversely, let $\Sigma$ be a not well-positioned of dimension greater or equal than three.  By Theorem~\ref{th:notwellpos}, $\Sigma$ contains a totally geodesic surface $N$. By Lemma~\ref{lemma:2octo}, the tangent space of $N$ is equal to the tangent space of a not well-positioned totally geodesic surface of $\s{S}^7_{\H,\tau}$. Therefore, $T_o N$ is either equal to a 2-dimensional subspace of the tangent space of a totally geodesic submanifold congruent to $\widehat{\Sigma}$ by Theorem~\ref{th:notwellsphere}, or $\tau=1/3$ and $T_o N$ is equal to the tangent space of a not well-positioned totally geodesic projective space in a well-positioned totally geodesic $\s{S}^7_{\H,1/3}$ as in Lemma~\ref{lemma:notwellposproj2} \textup{(b)}.

Thus, assume that the tangent space at $o\in\s{S}^{15}_{\O,\tau}$ of $N$ is as in Lemma~\ref{lemma:notwellposproj2}~(b). Then, $\tau=1/3$ and $T_o N$ is spanned by $u=\frac{\sqrt{3}}{2} X_2 + Y_1$ and $v= \frac{\sqrt{3}}{2} X_3+ Y_5$. Let $w\in\g{p}_{\Sigma}$ be a non-zero vector orthogonal to $u$ and $v$. Let us write $w=\sum_{i=1}^7 \alpha_i \frac{\sqrt{3}}{2} X_i + \sum_{i=1}^8 \beta_i Y_i$.
Now using that $\Sigma$ is a symmetric space,
\begin{align*}
	0&=\langle(\nabla_u R)(w,u,u),X_2\rangle=-\frac{16}{9}\beta_7, & & 	0=\langle(\nabla_u R)(w,u,u), Y_{7}\rangle=-\frac{8(\alpha_2-\beta_1)}{3\sqrt{3}},   \\
	0&=\langle(\nabla_v R)(w,v,v), Y_{7}\rangle= \frac{8(\alpha_3-\beta_5)}{3\sqrt{3}}, & &
	0=\langle(\nabla_v R)(w,u,v), Y_{5}\rangle= \frac{2(\alpha_1+3\beta_3)}{3\sqrt{3}}.
\end{align*}
Thus, 
\[\beta_7=0,\quad \beta_1=\alpha_2,\quad \beta_5=\alpha_3,\quad \alpha_1=-3\beta_3.\]
By Lemma~\ref{lemma:notwellposproj2}, $N$ is isometric to $\R \s{P}^2_{2/3}$. Moreover, $\Sigma$ contains $N$ as a totally geodesic submanifold, and thus, by  Theorem~\ref{th:notwellpos} and Corollary~\ref{cor:tghighdimension} ,  $\Sigma$ is homothetic to $\C \s P^2$. If $w\in\g{p}_{\Sigma}$ is a non-zero vector orthogonal to $X$ and $v$, we have $\sec(u,w)=\sec(v,w)=4\sec(u,v)$. Therefore,

\begin{equation}
	\label{eq:seccp1}
	\begin{aligned}
		0=&\omega_1(\sec(u,w)- 4\sec(u,v))=4\alpha_3^2+\alpha_4^2+\alpha_5^2+\alpha_6^2+\alpha_7^2 +\beta_2^2+16\beta_3^2+\beta_4^2 +\beta_6^2 +\beta^2_8  \\
		&+2(\alpha_5\beta_2-\alpha_4\beta_4+\alpha_7\beta_6-\alpha_6\beta_8),
	\end{aligned}
\end{equation}
\begin{equation}
	\label{eq:seccp2}
	\begin{aligned}
		0=&\omega_2(\sec(v,w)- 4\sec(u,v))=4\alpha_3^2+\alpha_4^2+\alpha_5^2+\alpha_6^2+\alpha_7^2 +\beta_2^2+16\beta_3^2+\beta_4^2 +\beta_6^2 +\beta^2_8\\
		&-2(\alpha_5\beta_2-\alpha_4\beta_4+\alpha_7\beta_6-\alpha_6\beta_8),
	\end{aligned}
\end{equation}
where $\omega_i:=2\alpha^2_{4-i}+\alpha^2_4+\alpha^2_5+\alpha^2_6+\alpha^2_7+\beta_2^2+10\beta^2_3+\beta^2_4+\beta^2_6+\beta^2_8$,
for each $i\in\{1,2\}$. Now adding Equation \eqref{eq:seccp1} and Equation \eqref{eq:seccp2}, we get 
\[2( 2\alpha_2^2+2 \alpha_2^2+\alpha_4^2+\alpha_5^2+\alpha_6^2+\alpha_7^2+\beta_2^2+10 \beta_3^2+\beta_4^2+\beta_6^2+\beta_8^2)=0.  \]
This implies that $w$ is zero yielding a contradiction.  Thus, we proved \textit{ii)}.

Let $\upvarphi$ be defined as in \textit{iii)}. Then, the trilinear map $\upvarphi$ is a $3$-form, since we have the equalities $\upvarphi(X,X,Y)=\upvarphi(X,Y,X)=\upvarphi(Y,X,X)=0$ for every $X,Y\in T_o\widehat\Sigma$.
Finally, notice that the isotropy of a not well-positioned totally geodesic submanifold of dimension~$7$ is isomorphic to $\s{G}_2$. The action of $\s{G}_2$ preserves $\g{p}_1$ and $\g{p}_2$, since $\s{G}_2$ lies in the isotropy of $\s{S}^{15}_{\O,\tau}$. Moreover, $\s{G}_2$ preserves the metric and the Lie bracket of the Lie algebra $\g{spin}_9$. Consequently,  $\upvarphi$ is a $\s{G}_2$-invariant $3$-form.
\end{proof}

In the following we will study the congruence problem for totally geodesic submanifolds in $\s{S}^{15}_{\O,\tau}$.

\begin{theorem}
\label{th:congoctonionic}
The moduli space of all not well-positioned totally geodesic spheres in $\s{S}^{15}_{\O,\tau}$ of dimension greater or equal than two  (up to congruence in $\s{S}^{15}_{\O,\tau}$) is isomorphic to the disjoint union
\[ \{2,4,5,6,7\} \sqcup (\{3\}\times[0,1]).     \]
The number $k\in\{2,4,5,6,7\}$ parametrizes the unique (up to congruence in $\s{S}^{15}_{\O,\tau}$) totally geodesic submanifold in $\s{S}^{15}_{\O,\tau}$ of dimension $k$. The set $\{3\}\times[0,1]$ parametrizes the not well-positioned totally geodesic submanifolds of dimension $3$ (up to congruence in $\s{S}^{15}_{\O,\tau}$).
\end{theorem}
\begin{proof}
Let $\Sigma_1$ and $\Sigma_2$ be two not well-positioned totally geodesic spheres of $\s{S}^{15}_{\O,\tau}$  with dimension $k$. When $k=7$, $\Sigma_1$ and $\Sigma_2$ are congruent by Lemma~\ref{lemma:dim7octonionic} \textit{ii)}.

Now assume that $\Sigma_1$ and $\Sigma_2$ are two not well-positioned totally geodesic spheres of dimension $2$. By Lemma~\ref{lemma:2octo} and Lemma~\ref{lemma:dim7octonionic}, they are contained in a well-positioned totally geodesic $\s{S}^7_{\H,\tau}$. Hence, $\Sigma_1$ and $\Sigma_2$ are congruent in $\s{S}^{15}_{\O,\tau}$ by Proposition~\ref{prop:congquaternionic}.

In what follows, let $\Sigma_1$ and $\Sigma_2$ be two not well-positioned totally geodesic spheres of dimension $k\in\{3,4,5,6\}$. Then, by Theorem~\ref{th:notwellsphere}  and Lemma~\ref{lemma:dim7octonionic}, the problem is reduced to analyze the congruence classes of $k$-dimensional totally geodesic submanifolds of the round sphere $\s{S}^7$ under the subgroup $\s{Spin}_7$. Observe that a $k$-dimensional totally geodesic submanifold of $\s{S}^7$ is uniquely determined by a $(k+1)$-plane of $\R^8$. Moreover, $\s{Spin}_7$ acts transitively on $\s{G}_{k+1}(\R^8)$, the Grassmannian of $(k+1)$-planes of $\R^8$, if and only if $k+1\neq 4$; and with cohomogeneity one when $k+1=4$. Furthermore this cohomogeneity one action has two singular orbits. Consequently, the set of congruence classes for $3$-dimensional not well-positioned totally geodesic spheres is identified with $[0,1]$ and for $k\neq 3$ there is just one congruence class. \qedhere
\end{proof}

In the sequel, we  explain how the $3$-form constructed in Lemma~\ref{lemma:dim7octonionic} \textit{iv)} defines an invariant which allows to distinguish between non-congruent not well-positioned totally geodesic $3$-spheres of $\s{S}^{15}_{\O,\tau}$. Let us identify $\Im \O$ with $\R^7$ and consider the 3-form $\upvarphi$ given by
\[\upvarphi(x,y,z)=\langle x, y\cdot z\rangle, \enspace x,y,z\in\Im\O,  \]
where $\langle\cdot,\cdot\rangle$ denotes the Euclidean inner product in $\R^7$, and $\cdot$ denotes the  product in $\O$. Also, $\upvarphi$ satisfies $||\upvarphi||\leq 1$, where  $||\upvarphi||:=\sup\{\upvarphi(\xi): \text{$\xi$ is a unit, simple 3-vector in $\R^7$}   \}$.
This is  a $\s{G}_2$-invariant 3-form and it was called by Harvey and Lawson~\cite{harvey-lawson} the associative calibration on $\Im \O$. The reason for this name is the following. Let $\xi=x\wedge y\wedge z$, where  $x, y$ and $z$ are elements of $\Im\O$ of unit length. Then,  $|\upvarphi(\xi)|=1$  if and only if the $3$-dimensional subspace spanned by $\{x,y,z\}$ is the imaginary part of a subalgebra of $\O$ isomorphic to $\H$.

We can consider the smooth function $\upvarphi\colon \s{G}_3\R^7\rightarrow[-1,1]$, where $\s{G}_3\R^7$ denotes the $3$-plane Grassmannian over $\R^7$. A direct computation shows that the critical points of $\upvarphi$ are precisely the associative $3$-planes. Moreover, $\s{G}_2$ acts on $\s{G}_3\R^7$ with cohomogeneity one (see \cite[Table 1]{Ko02}).
Thus, $|\upvarphi(V)|=|\upvarphi(V')|$ if and only if $V$ and $V'$ in $\s{G}_3\R^7$ belong to the same orbit for the $\s{G}_2$ action on $\s{G}_3\R^7$.

Furthermore, notice that the representation of $\s{G}_2$ on $\Lambda^3(\R^7)$ splits as the direct sum of the trivial, $7$-dimensional, and $27$-dimensional irreducible representations of $\s{G}_2$, see \cite[Theorem 8.5]{salamonwalpuski}. Thus,  there is just one (up to scaling) $\s G_2$-invariant $3$-form defined in $\Im\O$.  Moreover, the $3$-form $\upvarphi$ defined in Lemma~\ref{lemma:dim7octonionic} \textit{iv)} also satisfies that $||\upvarphi||\leq1$, so it agrees with the associative calibration up to sign. Consequently, given two not well-positioned totally geodesic $3$-spheres $\Sigma_1=\exp(V_1)$ and $\Sigma_2=\exp(V_2)$  contained in the same maximal not well-positioned totally geodesic sphere of $\s{S}^{15}_{\O,\tau}$, we have that they are congruent if and only if $|\upvarphi(V_1)|=|{\upvarphi}(V_2)|$. Furthermore, we can always find an element $g\in\s{Spin}_7$ such that $\Sigma_1$ and $g\Sigma_2$ are contained in the same maximal not well-positioned totally geodesic sphere, and since $\upvarphi$ is also invariant under $\s{Spin}_7$,  we have the following result.
\begin{proposition}
\label{prop:calibcong}
Let $\Sigma_i=\exp(V_i)$ be a not well-positioned totally geodesic $3$-sphere of $\s{S}^{15}_{\O,\tau}$ for each $i\in\{1,2\}$. Then, $\Sigma_1$ and $\Sigma_2$ are congruent in $\s{S}^{15}_{\O,\tau}$ if and only if $|\upvarphi(V_1)|=|\upvarphi(V_2)|$.
\end{proposition}
We ilustrate the usefulness of the invariant defined above with the following example.
\begin{example}
\label{ex:3dimoctonionic}
Let us consider the basis $\{v_i\}^7_{i=1}$ of $T_o\widehat{\Sigma}$ as defined by Equation~\eqref{eq:tangentwidehat}. Now for $\theta\in[0,1]$ consider the subspaces $V_\theta$ and $V'_\theta$ spanned by \[\spann\{\sin\left(\frac{\pi\theta}{2}\right) v_1+\cos\left(\frac{\pi\theta}{2}\right) v_4,v_2,v_3\},  \quad \spann\{\sin\left(\frac{\pi\theta}{2}\right) v_1+\cos\left(\frac{\pi\theta}{2}\right) v_2,v_5,v_6\};\] 
respectively. It turns out that $\upvarphi(V_{\theta})=\sin(\tfrac{\pi\theta}{2})$ and  $\upvarphi(V'_{\theta})=\cos(\tfrac{\pi\theta}{2})\in[0,1]$. Therefore, $\Sigma_\theta=\exp_o(V_\theta)$ and $\Sigma_\theta'=\exp_o(V'_\theta)$ are congruent if and only if $\theta=1/2$. 
\end{example}
\begin{remark}
We observe the following interesting fact.  Let $V$ be a $3$-dimensional subspace of $T_o\widehat{\Sigma}$. Then, $\Sigma=\exp_o(V)$ is a  not well-positioned totally geodesic submanifold $3$-sphere of a $7$-dimensional well-positioned totally geodesic submanifold isometric to $\s{S}^{7}_{\H,\tau}$ of $\s{S}^{15}_{\O,\tau}$ if and only if $\upvarphi(V)=1$, i.e.\ $V$ is associative when regarded as a subspace of $\Im\O$.

This follows by considering the subspace $V_\theta$ with $\theta=1$ in Example~\ref{ex:3dimoctonionic} and the subspace given by $(1)_{\O}$ in Subsection~\ref{subsect:examplestg} combined with the uniqueness (up to congruence) of the totally geodesic inclusions $\s{S}^7_{\H,\tau}\hookrightarrow\s{S}^{15}_{\O,\tau}$ and $\s{S}^3_{4-4\tau}\hookrightarrow\s{S}^{7}_{\H,\tau}$, respectively; see~Remark~\ref{rem:congwellpos} and Proposition~\ref{prop:congquaternionic}.
\end{remark}

Now we are in a position to prove Theorem~\ref{th:D}.
\begin{proof}[Proof of Theorem~\ref{th:D}] 
On the one hand, if $\Sigma$ is well-positioned the proof follows by Theorem~\ref{th:wellpos}, Remark~\ref{rem:intersect}, and Remark~\ref{rem:congwellpos}. On the other hand, if $\Sigma$ is not well-positioned the proof follows by combining Theorem~\ref{th:notwellpos}, Lemma~\ref{lemma:dim7octonionic} and Theorem~\ref{th:congoctonionic}.
\end{proof}

To conclude this section we will discuss the homogeneity of the not well-positioned totally geodesic submanifolds of $\s{S}^{15}_{\O,\tau}$. 

\begin{theorem}\label{th:4inhomo}
Let $\Sigma$ be a  totally geodesic submanifold  of $\s{S}^{15}_{\O,\tau}$. Then, $\Sigma$ is not extrinsically homogeneous if and only if $\Sigma$ is congruent to a $4$-dimensional not well-positioned totally geodesic sphere.

\end{theorem}

\begin{proof}
Firstly, let us assume that $\Sigma$ is equal to a $4$-dimensional not well-positioned totally geodesic sphere of $\s{S}^{15}_{\O,\tau}$.
By Theorem~\ref{th:notwellsphere}, we can always assume that $\Sigma$ is contained in a $7$-dimensional not well-positioned totally geodesic sphere $\s{S}^7_{4-4\tau}$ of $\s{S}^{15}_{\O,\tau}$. Then, by Lemma~\ref{lemma:dim7octonionic}, we can reduce the discussion to decide if a totally geodesic  $\s{S}^4$ in a round $\s{S}^7=\s{Spin}_7/\s{G}_2$ is an orbit of some subgroup of $\s{Spin}_7$. 

Let $\s{K} =\{g\in \s{SO}_8: g\s{S}^4= \s{S}^4\}^0$, where $\cdot^0$ denotes the connected component containing the identity. Then $\s{K}=\s{K}_1\times \s{K}_2$ where $\s{K}_1\simeq \s{SO}_5$   and $\s{K}_2\simeq \s{SO}_3$. Notice that $\s{K}_1$ acts on $\s{S}^4\subset\R^5$, and $\s{K}_2$ acts on $\R^3\simeq\R^8\ominus\R^5$. Let us assume that there exists   a Lie subgroup $\s{F}$ of $\s{Spin} _7\subset \s{SO}_8$ that acts transitively on $\s{S}^4$ and let $\proj_1\colon \s{F}\to \s{K}_1$ be the projection. Let  $\proj_1\colon\g{f}\rightarrow\g{k}_1$ be the corresponding map at the Lie algebra level. Since  $\s{F}$ is compact, we may assume that $\s{F}$ acts almost effectively on $\s{S}^4$.  Indeed, if the action of $\s{F}$ is not almost effectively on $\s{S}^4$, we can replace $\s{F}$ by the Lie group associated with a complementary ideal in  $\mathfrak f$ of 
$\mathfrak f \cap \mathfrak {k}_2=\ker\proj_1$. Then $\proj_1:\s{F}\to \s{K}_1$ is almost injective and $\proj_1(\s{F})=\s{SO}_5$, since $\s{SO}_5$ is the only compact Lie group  acting transitively and effectively on $\s{S}^4$ by the classification of transitive actions on spheres, see~\cite{MS43}. Then $\mathfrak f$ is isomorphic to $\mathfrak{so}_5$. Hence, since $\mathfrak f$ is simple and $\dim (\mathfrak f)>\dim (\mathfrak {so}_3)$ the Lie algebra morphism 
$\proj_2:\mathfrak f \to \g{k}_2$ is trivial. This proves that $\mathfrak {f}=\g{k}_1$. Let us consider the Cartan decomposition of the symmetric spaces $\s{S}^4$ and $\s{S}^7$, given by $\g{k}_1=\g{p}_{1}\oplus\g{h}_{1}$ and $\g{so}_8=\g{p}\oplus\g{so}_7$, respectively. Let us choose a non-zero element $X\in \g{p}_1$. Then, $X\in\g{p}\cap\g{spin}_7$. Thus, since $\s{G}_2$ acts transitively on the unit sphere of $T_o\s{S}^7$, for each direction $X\in T_o\s{S}^7$, there exists a Killing vector field of $\s{S}^7$ parallel at $o$. Consequently, $\g{f}$ contains $\g{p}$, and thus $\g{f}$ contains $\g{so}_8=\g{p}\oplus[\g{p},\g{p}]$ yielding a contradiction.

Secondly, we will prove that every totally geodesic submanifold of $\s{S}^{15}_{\O,\tau}$ (see Table~\ref{table:octonionictgnotwellpos}) that is not congruent to a $4$-dimensional not well-positioned totally geodesic sphere is extrinsically homogeneous. 
The geodesics of $\s{S}^{15}_{\O,\tau}$ are orbits, since $\s{S}^{15}_{\O,\tau}$ is a g.o.\ space; and well-positioned totally geodesic submanifolds of $\s{S}^{15}_{\O,\tau}$ are also orbits, by Remark~\ref{rem:congwellpos}.

Let $\Sigma$ be congruent to the totally geodesic projective plane $\R\s P^2_{2/3}$. Then $\tau=1/3$, see~Lemma~\ref{lemma:notwellposproj2}. In this case $\Sigma$ is contained in a well-positioned totally geodesic submanifold of $\s{S}^{15}_{\O,1/3}$ isometric to $\s{S}^7_{\H,1/3}$ by Lemma~\ref{lemma:2octo}. Then, $\Sigma$ is an extrinsically homogeneous submanifold of $\s{S}^{15}_{\O,\tau}$.

Let $\Sigma$ be congruent to a not well-positioned totally geodesic sphere of $\s{S}^{15}_{\O,\tau}$. By Theorem~\ref{th:notwellsphere}, we can always assume that $\Sigma$ is contained in a $7$-dimensional not well-positioned totally geodesic sphere of $\s{S}^{15}_{\O,\tau}$. Then,  Lemma~\ref{lemma:dim7octonionic} implies that we can reduce the discussion to analyze which  totally geodesic submanifolds of the round sphere $\s{S}^{7}$ are subgroups of $\s{Spin}_7$. 
Let us assume that $\Sigma$ is isometric to a not well-positioned totally geodesic sphere. Let $k\neq 4$ be the dimension of $\Sigma$. Now, if $k=7$, $\Sigma$ is extrinsically homogeneous by Theorem~\ref{th:notwellsphere} and Lemma~\ref{lemma:finclas}.

Observe that for every $k$-dimensional subspace  $V^k \subset T_o\s{S}^7$ there is a totally geodesic $\s{S}^k$ such that 
$T_o\s{S}^k=  V^k$. 
Let  us consider the action of $\s{G}_2$ on the Grassmanian 
$\s{G}_{k}(T_o\s{S}^7)$ of $k$-subspaces of $T_o\s{S}^7$.
If $k\not\in\{3,4\}$, then this action is transitive (see~\cite[Lemma 1.1]{kerr}). Hence,  there is only one congruence class at $o$ of totally geodesic $\s{S}^k$. Then, by Lemma \ref{lemma:finclas},  $\s{S}^k$ is extrinsically homogeneous for $k\not\in\{3,4\}$.

Let us consider the action of $\s{G}_2$ on the Grassmannian $\s{G}_{3}(T_o\s{S}^7)$ of $3$-planes in $T_o\s{S}^7$. It is known that this action has cohomogenity one (see e.g.\ \cite[Table 1]{Ko02}).  The isotropy subgroup $(\s{G}_2)_\pi$ at  any $3$-plane $\pi$ has dimension at least $3$, since $\dim (\s{G}_2)=14$ and $\dim (\s{G}_{3}(T_o\s{S}^7))= 12$. Moreover, since $(\s{G}_2)_\pi\subset \s{G}_2$, the rank of $(\s{G}_2)_\pi$ is at most $2$ and thus $(\s{G}_2)_\pi$ is not abelian. Hence, $(\s{G}_2)_\pi$   has a connected and (non-trivial)  compact simple   normal subgroup~$ \s{K}^\pi$. Note that   the isotropy of a principal orbit $\s{G}_2\cdot\pi$ has dimension $3$, and thus the Lie algebra of 
$\s{K}^\pi$ is isomorphic to $\mathfrak{so}_3$. 

Let $\pi$ be an arbitrary $3$-plane of $T_o\s{S}^7$ and let $\s{S}^3$ be the $3$-dimensional totally geodesic sphere of $\s{S}^7$ with $T_o\s{S}^3=\pi$.
Let us consider the action of $\s{K}^\pi$ on $T_o\s{S}^7$. Observe that this action leaves invariant the subspace $\pi$. Now we will consider the case in which $\s{K}^\pi$ acts trivially on $\pi$,  and the case in which $\s{K}^\pi$ acts non-trivially on $\pi$.

Case (i): $\s{K}^\pi$ acts trivially on $\pi$. 
Let $V\subset T_o\s{S}^7$ be the subspace of fixed vectors of 
$\s{K}^\pi$. Then $V$ contains $\pi$. If $ {V} = \pi$, then $\s{S}^3$ is extrinsically homogeneous since $\s{S}^3$ is a connected component of the  fixed set of a subgroup of the presentation group of $\s{S}^7$, in this case $\s{Spin}_7$, see~\cite[Lemma 9.1.1]{BCO}. Hence, if $V=\pi$, we are done.
If $\dim (V)=4$, by the same argument  the totally geodesic $\s{S}^4$ with $T_o\s{S}^4= V$ is extrinsically homogeneous, contradicting the discussion above.
Assume that $\dim( V)\geq 5$. Thus $\s{K}^\pi$  acts trivially on $T_o\s{S}^7\ominus V$, since $\s{K}^\pi$ is simple and $\dim(T_o\s{S}^7\ominus V)\leq 2$. Then $\s{K}^\pi$ acts trivially on $T_o\s{S}^7$, but this yields a contradiction that proves that $V=\pi$, and then we are done.

Case (ii): $\s{K}^\pi$ acts non-trivially on $\pi$. Then, since $\s{K}^\pi$ is simple, $\s{K}^\s{\pi}$ must act almost effectively on~$\pi$. Then $\{k_{\vert \pi}: k \in \s{K}^\pi\} = \s{SO}(\pi)\cong \s{SO}_3$. Assume that $\s{S}^3$ is not homogeneous. Then,  $\{k_{\vert T_p\s{S}^3}: k \in \s{K}^{T_p\s{S}^3}\} = \s{SO}(T_p\s{S}^3)$ for every, $p\in \s{S}^3$. Otherwise, $\s{K}^{T_p\s{S}^3}$ would act trivially for some  $p\in \s{S}^3$, yielding a contradiction by the discussion in Case (i).
This implies that
$\s{S}^3$ is extrinsically homogeneous. 
\qedhere	
\end{proof}

\end{document}